\documentclass[11pt,reqno]{article}

\usepackage[margin=1in]{geometry}
\usepackage{amsthm,amsmath,epsfig,amsfonts,amssymb,bm,xcolor,amsbsy,times, dsfont, amssymb,amsmath,amscd,latexsym, amsthm, amsxtra,amsfonts}
\usepackage{bm}
\usepackage{bbm}
\usepackage{enumerate}
\usepackage{mathrsfs}
\usepackage{graphicx}
\usepackage{subfig}
\usepackage{comment}
\usepackage{lineno}
\usepackage{mathtools}
\usepackage{cases}
\usepackage{dsfont, amssymb,amsmath,amscd,latexsym, amsxtra,amsfonts,amsthm}
\numberwithin{equation}{section}
\usepackage{ifpdf}
\usepackage{color}
\definecolor{webgreen}{rgb}{0,.5,0}
\definecolor{webbrown}{rgb}{.8,0,0}
\definecolor{emphcolor}{rgb}{0.95,0.95,0.95}

\usepackage{hyperref}
\hypersetup{%
          colorlinks=true,
          linkcolor=blue,
          filecolor=blue,
          citecolor=blue,
          urlcolor=blue,
          breaklinks=true}

\usepackage[round]{natbib}
\newtheorem{theorem}{Theorem}[section]

\newtheorem{proposition}[theorem]{Proposition}
\newtheorem{lemma}[theorem]{Lemma}

\theoremstyle{definition}
\newtheorem{definition}[theorem]{Definition}

\theoremstyle{definition}

\theoremstyle{definition}
\newtheorem{remark}{Remark}
\theoremstyle{definition}
\newtheorem{assumption}{Assumption}

\title{A Mean Field Game Approach to Relative Investment-Consumption Games with Habit Formation}

\author{Zongxia Liang\thanks{ Email: liangzongxia@mail.tsinghua.edu.cn, 
		Department of Mathematical Sciences, Tsinghua University, Beijing, China.}
	\and Keyu Zhang\thanks{ Email: Zhangky21@mails.tsinghua.edu.cn, 
		Department of Mathematical Sciences, Tsinghua University, Beijing, China.}}
\date{}
\begin{document}
\maketitle

\begin{abstract}
This paper studies an optimal investment-consumption  problem  for competitive agents with exponential or power utilities and a common finite time horizon. Each agent regards the average of  habit formation and  wealth from all peers  as benchmarks to evaluate the performance of her decision.  We formulate the $n$-agent game problems and the corresponding mean field game problems under the two utilities. One mean field equilibrium is derived in a closed form in each problem. In each problem with $n$ agents,  an approximate Nash equilibrium is then constructed using the obtained mean field equilibrium when $n$ is sufficiently large. The explicit convergence order in each problem can also be obtained. In addition, we provide some numerical illustrations of our results. 

\end{abstract}
\section{Introduction}
Relative performance is of tantamount importance in  both the economic and sociological studies. As stated in \cite{2017price},  ``making a 1 EUR profit when everyone else made 2 EUR “feels” distinctly different had everyone else lost 2 EUR". The intuition behind is that human beings tend to compare themselves to their peers or some benchmarks. Such facts have been very well documented, see, for example, \cite{abel90}, \cite{GOMEZ200795} and \cite{Gali94}.  As a natural way to model the interactions among agents, relative performance has been also introduced in portfolio optimization problems. To name a few recent studies, we refer to \cite{2013OPTIMAL}, \cite{D2017Mean}, \cite{2020Many}, \cite{doi:10.1137/20M138421X}, \cite{2021Meanfieldito}, \cite{2022Meanfieldpg}, \cite{2022Meanfieldpgc}, \cite{2022habit} among others.

The study of consumption habit formation is a classical topic in financial economics. The time non-separable habit formation preference can address the shortcomings of the classical time-separable preferences and provide an explanation for the  equity premium puzzle, see \cite{f8b89b64-c09e-3f5f-b47c-1a4cf3168fdc}. As a benchmark reflecting the past consumption of agents,  habit formation can be regarded as a relative performance index and  explain the effect of past consumption patterns on current and future decisions. In a multi-agent scenario, each agent's habit level is usually set to depend on the average of the aggregate consumption from all agents; see \cite{abel90}, \cite{bd6761ac-561d-37c7-856a-f5e187d68cd9}, \cite{ABEL19993}, \cite{2022habit} and references therein. Equilibrium consumption behavior in this setting fits naturally  with the models of ``catching up with the Joneses": each agent makes decisions by competing with the historical consumption record from others.

In this paper, we revisit the many-agent games of investment-consumption optimization, however, under a new setting of relative performance criteria for agents having exponential or power utilities. Each individual regards the average of  habit formation and  wealth from all peers  as benchmarks to evaluate the performance of her decision. Our model allows agents to suppress the consumption below the average habit level form time to time to accumulate higher wealth from the financial market. It can be linked to the concept of nonaddictive habit formation, for more studies on nonaddictive habit formation, see \cite{abel90}, \cite{dda705ff-8999-3974-bcb9-25a27cb27b25} and \cite{ABEL19993}.  

We aim to investigate the equilibrium investment and consumption behavior in a large population game, in which each agent trades between their individual stock and common riskless asset. However, stochastic differential games with a large number $n$ of players are rarely tractable. \cite{2006Large} and \cite{Lasry-2007} innovatively introduce a mean field game approach to handle large population games. The MFG approach amounts to looking for a Nash equilibrium in the limiting regime, which turns out to be an approximate Nash equilibrium for the corresponding $n$-agent game. Inspired by this, we turn to consider the mean field game problems with infinitely many heterogeneous agents. In each problem with exponential utility and power utility respectively, we can obtain one mean field equilibrium  in an analytical form. Moreover, we further construct the approximate Nash equilibrium for the corresponding $n$-agent game using the obtained mean field equilibrium with the explicit order of approximation error. 

The most closely related works to ours are  \cite{2020Many} and \cite{2022habit}.  {In contrast to \cite{2020Many},  we do not incorporate common noise into our model due to technical issues, which will be elaborated upon in Section \ref{sect6}.} The benchmark of agent's consumption decision in the present paper is generated by the weighted average integral of consumption of all agents, which makes the analysis involved and the model realistic as it captures the influence of historical consumption behavior on current and future decisions.  Despite the absence of common noise, we do obtain elements not present in \cite{2020Many}, for example, the hump-shaped consumption path in our model. Various empirical studies indicate that consumption spending of individuals usually have a hump-shaped pattern; see \cite{RePEc:aea:aecrev:v:59:y:1969:i:3:p:324-30} and \cite{10.1162/rest.89.3.552}. Recently, \cite{doi:10.1137/22M1471560} provided a theoretical justification for the consumption hump through addictive habit formation. As an important add-on to the literature of consumption hump, our paper puts forward a new understanding of its formation. Moreover, it is shown that the exponential utility model in our framework is also tractable, while \cite{2020Many} only focus on the CRRA utilities. Comparing to  \cite{2022habit}, we additionally consider the relative wealth concern, increasing the complexity of agent interaction.  From a mathematical point of view, the more complex interaction gives rise to some new challenges in the proof of the consistency condition for the conjectured mean field equilibrium. To be specific, the consistency condition in our paper leads to a fixed point system. In the exponential utility case, the fixed point system (\ref{orgfix}) can be decoupled and reduces to a functional ODE (\ref{new fixed point}). The well-posedness of the functional ODE follows from a standard contraction mapping argument. However, in the power utility case, the fixed point system (\ref{general fixedpower}) has a strong coupling structure.  Some technical arguments are required to derive some estimations and the existence of the fixed point is guaranteed by Schauder's fixed point theorem, see  Proposition \ref{prop2.3power}. It is worth noting that \cite{2022habit} only addressed the case of homogeneous agents in the MFG problems, while our work solves the MFG problems of heterogeneous agents. Thanks to the analytical form of the mean field equilibrium and the approximate Nash equilibrium, we can make some sensitivity analysis. We find that competition mainly affects agent's consumption behavior and  the monotonicity of MFE consumption  highly depends on model parameters.

In summary, the contributions of our paper are as follows: first, our paper contributes to the literature of portfolio games with consumption by featuring our path-dependent competition benchmark. Portfolio games with consumption  are poorly explored, except in \cite{2020Many}, \cite{doi:10.1137/20M138421X} and \cite{2022Meanfieldpgc}. { In contrast to the aforementioned research, one of competitive benchmarks in our paper  is generated by the weighted average integral of consumption controls from all competitors, which  makes the benchmark path-dependent.}  Second, our work  gives an explanation to the hump-shaped consumption from the perspective of competition. Note that \cite{2022habit} only observe the hump-shaped consumption when the initial habit is very high, and our model shows the hump-shaped consumption even when the initial habit level is relatively low, indicating that the heterogeneity of agents is particularly crucial in explaining this phenomenon.

The rest of the paper is organized as follows. In Section  \ref{sect2}, we introduce the $n$-agent game problems for agents having exponential or power utilities.  In Section \ref{sect3}, we formulate two MFG problems with two types of utility function and infinitely many agents. A mean field equilibrium in each problem is established in analytical form. In Section \ref{sect4}, we use the mean field equilibrium to construct an approximate Nash equilibrium in each $n$-agent game. The explicit order of the approximation error in each problem is derived. Some numerical results and their financial implications are presented in Section \ref{sect5}. We conclude in Section \ref{sect6} with a discussion of open problems.  Finally, the proofs of some auxiliary results are given in Appendix.
\section{Problem Formulation}\label{sect2}
We consider a large population system with $n$ negligible agents, and each agent invests in their own specific stock or in a common riskless bond which offers zero interest rate.  The common time horizon for all agents is $[0,T]$ with $T>0$. The price process of stock $i$, in which only agent $i$ trades, follows the following stochastic differential equation (SDE)
\begin{equation}
	\frac{dS_{t}^{i}}{S_{t}^{i}}=\mu_{i}dt+\sigma_{i}dW_{t}^{i},
\end{equation}
with constant parameters $\mu_{i}>0$ and $\sigma_{i}>0$, where the Brownian motions $W^{1},\cdots,W^{n}$ are independent  on a filtered probability space $(\Omega,\mathcal{F},\mathbb{F},\mathbb{P})$, where the filtration $\mathbb{F}:=(\mathcal{F}_{t})_{t\in[0,T]}$ satisfies the usual conditions. A similar model setting with asset specialization can be found in \cite{D2017Mean}, \cite{2020Many} and \cite{2022habit}. {However, we do not assume the presence of common noise as it poses technical challenges. We will elaborate on this in Section \ref{sect6}. }

Each agent $i$ trades, using a self-financing strategy, $\Pi^{i}:=\{\Pi_{t}^{i}, 0\leq t\leq T \}$,  representing the amount invested in the stock $i$, and a consumption policy, $C^{i}:=\{C^{i}_{t}, 0\leq t\leq T \}$, representing the consumption rate process of agent $i$. We also introduce the portfolio-to-wealth proportion $\pi^{i}_{t}:=\Pi^{i}_{t}/X^{i}_{t}$ and consumption-to-wealth proportion $c^{i}_{t}:=C^{i}_{t}/X^{i}_{t}$. Then the $i$-th agent's wealth process $X^{i}:=\{X^{i}_{t}, 0\leq t\leq T\}$ satisfies  
\begin{equation}\label{wealth..}
	dX^{i}_{t}=\Pi_{t}^{i}\mu_{i}dt+\Pi_{t}^{i}\sigma_{i}dW_{t}^{i}-C^{i}_{t}dt,\quad X^{i}_{0}=x^{i}_{0},
\end{equation}
 or  
 \begin{equation}\label{wealthpower}
 	\frac{dX^{i}_{t}}{X_{t}^{i}}=\pi_{t}^{i}\mu_{i}dt+\pi_{t}^{i}\sigma_{i}dW_{t}^{i}-c^{i}_{t}dt,\quad X^{i}_{0}=x^{i}_{0}.
 \end{equation}

Each individual has a  habit formation process $Z^{i}=\left\{Z^{i}_{t},0\leq t\leq T\right\}$  generated by the consumption rate process $C^{i}$ satisfying
\begin{equation}
	dZ^{i}_{t}=-\delta_{i}(Z^{i}_{t}-C^{i}_{t})dt,\quad Z^{i}_{0} = z^{i}_{0}>0,
\end{equation}
where $z^{i}_{0}$ represents the initial consumption habit of the individual and  $\delta_{i}>0$ is a constant determining how much current habit is influenced by the recent rate of consumption relative to the consumption rate farther in the past. It follows that
\begin{equation}
	Z^{i}_{t} = e^{-\delta_{i}t}\left(z_{0}^{i}+\int_{0}^{t}\delta_{i}e^{\delta_{i}s}C^{i}_{s}ds\right),\quad t\in[0,T].
\end{equation}
For the group of $n$ agents in the financial market, let us define their average habit formation process by \[\overline{Z}^{n}_{t}:=\frac{1}{n}\sum_{i=1}^{n}Z^{i}_{t},\quad t\in[0,T],\]
which depicts the average of aggregate consumption trend in the economy.

In this paper, we analyze the relative investment-consumption games, for agents having exponential or power utilities. In the exponential case, the objective function is defined by
\begin{equation}\label{ojb}
	\begin{split}
		J_{i}\mathbf{(\Pi,C)}:=\mathbb{E}\left[\int_{0}^{T}U_{i}\left(C^{i}_{t}-\theta_{i}\overline{Z}^{n}_{t}\right)dt+U_{i}\left(X^{i}_{T}-\theta_{i}\overline{X}_{T}^{n}\right)\right],
	\end{split}
\end{equation}
where $\mathbf{(\Pi,C)}:=\left((\Pi^{1},C^{1}),\cdots,(\Pi^{n},C^{n})\right)$, $\overline{X}_{T}^{n}:=\frac{1}{n}\sum_{i=1}^{n}X_{T}^{i}$, $\theta_{i}\in(0,1]$ represents the competition level of the relative performance, and  $U_{i}:\mathbb{R}\rightarrow\mathbb{R}$ is the exponential utility of agent $i$ that
\begin{equation}
	U_{i}(y) = -\exp\left\{-\frac{1}{\beta_{i}}y\right\},\quad \beta_{i}>0 ,\,\,\ y\in\mathbb{R}.
\end{equation}
 In the power case, the objective function is 
\begin{equation}\label{ojbpower}
	\begin{split}
		J_{i}\mathbf{(\pi,c)}:=\mathbb{E}\left[\int_{0}^{T}U_{i}\left(\frac{c^{i}_{t}X^{i}_{t}}{\left(\overline{Z}^{n}_{t}\right)^{\theta_{i}}}\right)dt+U_{i}\left(\frac{X^{i}_{T}}{\left(\overline{X}_{T}^{n}\right)^{\theta_{i}}}\right)\right],
	\end{split}
\end{equation}
where $\mathbf{(\pi,c)}=((\pi^{1},c^{1}),\cdots,(\pi^{n},c^{n}))$, $\overline{X}_{T}^{n}:=\left(\prod_{i=1}^{n} X_{T}^{i}\right)^{1/n}$, $\theta_{i}\in(0,1]$ represents the relative competition weight,  and  $U_{i}:\mathbb{R}_{+}\rightarrow\mathbb{R}_{+}$ is the power utility of agent $i$ that
\begin{equation}
	U_{i}(y) = \frac{1}{p_{i}}y^{p_{i}},\quad p_{i}\in (0,1),\quad y\geq0.
\end{equation}
{
\begin{remark}
The term $C^{i}_{t}-\theta_{i}\overline{Z}^{n}_{t}$ in $(\ref{ojb})$, commonly known as linear habit formation in the literature, measures the difference between the current consumption rate and the average of the aggregate consumption from all agents. On the other hand, the term $(c^{i}_{t}X^{i}_{t})/\left(\overline{Z}^{n}_{t}\right)^{\theta_{i}}$ in (\ref{ojbpower}), referred to as multiplicative habit formation, represents the ratio of the current consumption rate to the average aggregate consumption.

Inspired by \cite{D2017Mean} and \cite{2020Many}, we combine exponential utility with linear habit formation and power utility with multiplicative habit formation to investigate the equilibrium behavior of agents under external habit formation.
\end{remark}}
The admissible control sets for both problems are defined as follows.
\begin{definition}{(Admissible controls in the exponential case)}\label{admiss}
	
	We denote $\mathcal{A}^{i,e}_{t}$ as the set of $\mathbb{F}$-progressively measurable processes $(\Pi^{i}, C^{i})$ such that $$\mathbb{E}\left[\int_{t}^{T}\left[\left(\Pi^{i}_{s}\right)^{2}+\left(C^{i}_{s}\right)^{2}\right]ds\right]<\infty.$$
\end{definition}
Then, for any $\left(x^{i}_{0},\Pi^{i},C^{i}\right)\in\mathbb{R}\times\mathcal{A}_{0}^{i,e}$, we have the following estimate from (\ref{wealth..}):
\begin{equation}\label{estimate}
	\mathbb{E}\left[\sup_{t\in[0,T]}|X^{i}_{t}|^{2}\right]\leq K\left(1+(x^{i}_{0})^{2}\right)
	\end{equation}
for some constant $K$.

\begin{definition}{(Admissible controls in the power case)}\label{admisspower}

	We denote $\mathcal{A}^{i,p}_{t}$ as the set of $\mathbb{F}$-progressively measurable processes $(\pi^{i}, c^{i})$ such that $c^{i}_{s}\geq0$ a.s. for $s\in[t,T]$, and no bankruptcy is allowed that $X^{i}_{s}\geq0$ a.s. for $s\in[t,T]$. 
\end{definition}
{
	\begin{remark}
	Let us explain why we have two distinct settings for admissible controls in the two problems.  In the second problem, it is necessary for $c^{i}$ and $X^{i}$ to be nonnegative, as the power utility function is defined on $\mathbb{R}_{+}$. On the other hand, as the exponential utility function is defined on $\mathbb{R}$, the wealth process $X^{i}$ and consumption $C^{i}$ are no longer required to be nonnegative. This also makes the problem more tractable compared to the scenario with a  nonnegativity constraint on the consumption rate.
	In fact, there are many papers on optimal consumption problems with exponential utility that assume $C_{t}\in\mathbb{R}$ for technical convenience, see for example \cite{1969Merton}, \cite{2004Liu}, \cite{2015Dimitri}, \cite{2019ma} and \cite{2021hamag}. Although the first problem lacks economic sense relative to the second one, we choose to keep it because it is tractable and possesses a unique mean field equilibrium as outlined in Proposition \ref{prop2.3}, which can be contrasted with the results of the second problem.
	\end{remark}
}
 For technical convenience, we make the following assumption:
\begin{assumption}\label{Assumption 1}
	
   Assume that all agents are homogeneous in their initial wealth, the initial habit, the habit discounting factor such that $\left(x_{0}^{i},z_{0}^{i},\delta_{i}\right)=\left(x_{0},z_{0},\delta\right)\in\mathbb{R}_{+}^{3}$, $\forall i\in\left\{1,\cdots,n\right\}$.
\end{assumption}

 For agent $i\in\left\{1,\cdots,n\right\}$, we define the type vector $o_{i}=(\mu_{i},\sigma_{i},\beta_{i}, \theta_{i})$ in the exponential case (resp. $o_{i}=(\mu_{i},\sigma_{i},p_{i}, \theta_{i})$ in the power case). Then, the heterogeneity of $n$ agents is captured via their different type vectors $o_{i}\in\mathcal{O}_{e}$ (resp. $o_{i}\in\mathcal{O}_{p}$). The following assumption on type space $\mathcal{O}_{e}$ (resp. $\mathcal{O}_{p}$) will be needed throughout this paper. 
 \begin{assumption}\label{assump2}
  Assume that the type spaces  has the following structure
  \begin{align*}
  	&\mathcal{O}_{e}=\left\{(\mu_{1}, \sigma_{1}, \beta_{1}, \theta_{1}),\cdots,(\mu_{K}, \sigma_{K}, \beta_{K}, \theta_{K})\right\},\\
	&\mathcal{O}_{p}=\left\{(\mu_{1}, \sigma_{1}, p_{1}, \theta_{1}),\cdots,(\mu_{K}, \sigma_{K}, p_{K}, \theta_{K})\right\}, 
\end{align*}
where $K$ is a finite number.
 \end{assumption} 
Under Assumption \ref{assump2}, all the agents are divided into $K$ classes and the wealth process can be rewritten as 
\begin{equation}\label{dynaexp}
 	dX^{i}_{t}=\Pi_{t}^{i}\mu_{\alpha(i)}dt+\Pi_{t}^{i}\sigma_{\alpha(i)}dW_{t}^{i}-C^{i}_{t}dt,\quad X^{i}_{0}=x_{0}
\end{equation}
in the exponential case and  
\begin{equation}\label{dynapower}
	\frac{dX^{i}_{t}}{X_{t}^{i}}=\pi_{t}^{i}\mu_{\alpha(i)}dt+\pi_{t}^{i}\sigma_{\alpha(i)}dW_{t}^{i}-c^{i}_{t}dt,\quad X^{i}_{0}=x_{0}
\end{equation}
in the power case, where $\alpha(i)\in\left\{1,\cdots,K\right\}$ indicates the class to which the agent $i$ belongs. Similarly, in both cases the objective functions can be reformulated as follows, respectively:
\begin{equation*}
	\begin{split}
		&J_{i}\mathbf{(\Pi,C)}=\mathbb{E}\left[\int_{0}^{T}U_{\alpha(i)}\left(C^{i}_{t}-\theta_{\alpha(i)}\overline{Z}^{n}_{t}\right)dt+U_{\alpha(i)}\left(X^{i}_{T}-\theta_{\alpha(i)}\overline{X}_{T}^{n}\right)\right],\\
		&J_{i}\mathbf{(\pi,c)}=\mathbb{E}\left[\int_{0}^{T}U_{\alpha(i)}\left(\frac{c^{i}_{t}X^{i}_{t}}{\left(\overline{Z}^{n}_{t}\right)^{\theta_{\alpha(i)}}}\right)dt+U_{\alpha(i)}\left(\frac{X^{i}_{T}}{\left(\overline{X}_{T}^{n}\right)^{\theta_{\alpha(i)}}}\right)\right].
	\end{split}
\end{equation*}
{ It is worth mentioning that both $n$-agent games  cannot be solved in closed form. Consequently, we employ a mean field game approach to seek the approximate Nash equilibria. }
\section{Mean Field Games}\label{sect3}
We now proceed to formulate the auxiliary limiting mean field games when the number of agents grows to infinity. For the sequence $\left\{\alpha(i), i\geq1\right\}$, we define the empirical distribution associated with the first $n$ agents
\begin{equation*}
	F_{n}(\left\{k\right\})=\frac{1}{n}\sum_{i=1}^{n}1_{\{\alpha(i)=k\}}.
\end{equation*} 
The following assumption is needed to formulate the MFGs:
\begin{assumption}
 There exists a distribution function on $[K]:=\left\{1,\dots,K\right\}$, denoted as $\left(F(\left\{1\right\}),\cdots,F(\left\{K\right\})\right)$ such that $\lim_{n\rightarrow\infty}F_{n}\left(\left\{k\right\}\right)=F\left(\left\{k\right\}\right)$ for each $k\in[K]$.
\end{assumption}
\subsection{Exponential Utility}
We introduce the wealth process of a representative agent as follows
\begin{equation}\label{wealthrep.2}
	dX_{t}=\Pi_{t}\mu dt+\Pi_{t}\sigma dW_{t}-C_{t}dt,\quad X_{0}=x_{0},
\end{equation}
where $W$ is a Brownian motion on the probability space $\left(\Omega,\mathcal{F}, \mathbb{F},\mathbb{P}\right)$, the $\mathcal{F}_{0}$-measurable random type vector $o=(\mu,\sigma,\beta,\theta)$ is independent of $W$ and has a discrete distribution $m$, such that $$m\left(o=o_{k}\right)=F\left(\left\{k\right\}\right)$$ for each $o_{k}=(\mu_{k},\sigma_{k},\beta_{k},\theta_{k})\in\mathcal{O}_{e}$. 

By the law of large numbers, we should expect that $\overline{Z}^{n}=(\overline{Z}^{n}_{t})_{t\in[0,T]}$ and $\overline{X}^{n}_{T}$ can be approximated by a deterministic function $\overline{Z}=(\overline{Z}_{t})_{t\in[0,T]}\in\mathcal{C}_{T}:=C([0,T];\mathbb{R})$ and a constant $\overline{X}_{T}\in\mathbb{R}$ respectively when $n$ is sufficiently large.

Given  $\overline{Z}\in\mathcal{C}_{T}$ and $\overline{X}_{T}\in\mathbb{R}$, the dynamic version of the objective function for the representative agent is defined by 
\begin{equation}\label{obf.2}
	\begin{split}
		J^{r}((\Pi,C),t,x;\overline{Z},\overline{X}_{T}):=\mathbb{E}_{t,x}\left[\int_{t}^{T}-\exp\left\{-\frac{1}{\beta}\left(C_{t}-\theta\overline{Z}_{t}\right)\right\}dt-\exp\left\{-\frac{1}{\beta}\left(X_{T}-\theta\overline{X}_{T}\right)\right\}\right],
	\end{split}
\end{equation}
where $\mathbb{E}_{t,x}:=\mathbb{E}\left[\cdot|X_{t}=x\right]$. The stochastic control problem is given by 
\begin{equation}\label{scp}
	\sup_{\left(\Pi,C\right)\in\mathcal{A}_{t}^{r,e}}J^{r}((\Pi,C),t,x;\overline{Z},\overline{X}_{T}) = \mathbb{E}_{m}\left[V(t,x,o)\right]:=\int_{\mathcal{O}}V(t,x,o)m(do),
\end{equation}
where $\mathcal{A}_{t}^{r,e}$ is the admissible control set for the MFG problem defined similarly to $\mathcal{A}_{t}^{i,e}$, and $V(t,x,o)$ is the value function defined for a given type vector $o$.

We now formally give the definition of mean field equilibrium as follows.
\begin{definition}{(Mean Field Equilibrium)}\label{mfe}
	
	For given $\overline{Z}\in\mathcal{C}_{T}$ and $\overline{X}_{T}\in\mathbb{R}$, let $(\Pi^{*},C^{*})\in\mathcal{A}_{0}^{r,e}$ be the best response to the stochastic control problem (\ref{scp}) from the initial time $t=0$.    
	The control $(\Pi^{*},C^{*})$ is called a mean field equilibrium, if it is the best response  to  itself in the sense that $\overline{Z}_{t} =z_{0}e^{-\delta t}+\int_{0}^{t}\delta e^{\delta (s-t)}\mathbb{E}\left[C^{*}_{s}\right]ds,\ t\in[0,T]$ and $\overline{X}_{T} = \mathbb{E}\left[X^{*}_{T}\right]$, where $X^{*}$ is the wealth process under the optimal control $(\Pi^{*},C^{*})$ with $X^{*}_{0}=x_{0}$.
\end{definition}
By Definition \ref{mfe}, we shall first solve the stochastic control problem (\ref{scp}) for given $\overline{Z}\in\mathcal{C}_{T}$ and $\overline{X}_{T}\in\mathbb{R}$.  Let $o=\left(\mu, \sigma, \beta, \theta\right)$ now represent a deterministic sample from its random type distribution $m$. By dynamic programming principle, we derive that the associated HJB equation of value function $V^{r}(t,x):=V(t,x,o)$ is given by
\begin{equation}\label{HJB}
	\partial_{t}V^{r}(t,x)+\sup_{\Pi\in\mathbb{R}}\left(\mu\Pi \partial_{x}V^{r}+\frac{\sigma^{2}}{2}\Pi^{2}\partial^{2}_{x}V^{r}\right)+\sup_{C\in\mathbb{R}}\left(-C\partial_{x}V^{r}-\exp\left\{-\frac{1}{\beta}\left(C-\theta\overline{Z}_{t}\right)\right\}\right)=0
\end{equation}
with the terminal condition $V^{r}(T,x)=-\exp\left\{-\frac{1}{\beta}\left(x-\theta\overline{X}_{T}\right)\right\}$. The optimal control is given in the following lemma.
\begin{lemma}\label{lemma7.8}
	Given $\overline{Z}=(\overline{Z}_{t})_{t\in[0,T]}\in\mathcal{C}_{T}$ and $\overline{X}_{T}\in\mathbb{R}$, the classical solution to the HJB equation (\ref{HJB}) admits the following closed-form 
	\begin{equation}
		V^{r}(t,x)=-\exp\left\{a(t)x+b(t)\right\},\quad (t,x)\in[0,T]\times\mathbb{R},
	\end{equation}
	where \begin{equation}
		\begin{split}
			&a(t)=-\frac{1}{\beta}\frac{1}{T+1-t},\\
			&b(t)=\frac{1}{T+1-t}\frac{\theta}{\beta}\overline{X}_{T}+\frac{1}{T+1-t}\frac{\theta}{\beta}\int_{t}^{T}\overline{Z}_{s}ds+\log\left(T+1-t\right)\\&-\frac{1}{4}\left(\frac{\mu}{\sigma}\right)^{2}\left((T+1-t)-\frac{1}{T+1-t}\right).             
		\end{split}
	\end{equation}
	{ There exists a unique optimal control to Problem (\ref{scp}) and the feedback functions of the optimal investment and consumption are given by}
	\begin{equation}\label{strategy}
		\begin{split}
				&\Pi^{*}(t,x)=\beta\frac{\mu}{\sigma^{2}}\left(T+1-t\right),\\ 
				&C^{*}(t,x) =\frac{(x-\theta\overline{X}_{T})}{T+1-t}+\theta\overline{Z}_{t}-\frac{\theta}{T+1-t}\int_{t}^{T}\overline{Z}_{s}ds\\&+\frac{1}{4}\beta\left(\frac{\mu}{\sigma}\right)^{2}\left((T+1-t)-\frac{1}{T+1-t}\right).      
		\end{split}
	\end{equation}
\end{lemma}
\begin{proof}
	We first assume that the classical solution $V^{r}$ is strictly concave (i.e., $\partial^{2}_{x}V<0$). Then, the first-order condition gives the feedback functions: for $\left(t,x\right)\in[0,T]\times\mathbb{R}$,
	\begin{equation}\label{feedback}
	\Pi^{*}(t,x) = -\frac{\mu}{\sigma^{2}}\frac{\partial_{x}V^{r}(t,x)}{\partial^{2}_{x}V^{r}(t,x)}, \quad C^{*}(t,x) =\theta\overline{Z}_{t}-\beta\log\left(\beta\partial_{x}V^{r}(t,x)\right).
\end{equation}
  Plugging the functions (\ref{feedback}) into (\ref{HJB}), we have
  \begin{equation}\label{pde1}
  	0=\partial_{t}V^{r}(t,x)-\frac{\mu^{2}}{2\sigma^{2}}\frac{\left(\partial_{x}V^{r}(t,x)\right)^{2}}{\partial^{2}_{x}V^{r}(t,x)}+\partial_{x}V^{r}(t,x)\left(\beta\log\left(\beta\partial_{x}V^{r}(t,x)\right)-\theta\overline{Z}_{t}-\beta\right).
  \end{equation}
To solve (\ref{pde1}), we make the ansatz: 
\begin{equation}\label{ansatz}
	V^{r}(t,x)=-\exp\left\{a(t)x+b(t)\right\}, \,\ \left(t,x\right)\in[0,T]\times\mathbb{R}.
\end{equation}
Substituting (\ref{ansatz}) into (\ref{pde1}), we get 
\begin{equation}
	\left(\dot{a}(t)+\beta\left(a(t)\right)^{2}\right)x+\left(\dot{b}(t)+\beta a(t)b(t)+a(t)\left[\beta\log\left(-\beta a(t)\right)-\theta\overline{Z}_{t}-\beta\right]-\frac{\mu^{2}}{2\sigma^{2}}\right)=0.
\end{equation}
Then, we obtain the ODE for $a(\cdot)$ and $b(\cdot)$ 
\begin{equation}
	\begin{cases}
		\dot{a}(t)+\beta\left(a(t)\right)^{2}=0 ,\\
		\dot{b}(t)+\beta a(t)b(t)+a(t)\left[\beta\log\left(-\beta a(t)\right)-\theta\overline{Z}_{t}-\beta\right]-\frac{\mu^{2}}{2\sigma^{2}}=0,\\
	    a(T)=-\frac{1}{\beta},\quad b(T)=\frac{\theta}{\beta}\overline{X}_{T} ,
	\end{cases}
\end{equation}
thus
\begin{equation}\label{ab}
	\begin{split}
		&a(t)=-\frac{1}{\beta}\frac{1}{T+1-t},\\
		&b(t)=\frac{1}{T+1-t}\frac{\theta}{\beta}\overline{X}_{T}+\frac{1}{T+1-t}\frac{\theta}{\beta}\int_{t}^{T}\overline{Z}_{s}ds+\log\left(T+1-t\right)\\&-\frac{1}{4}\left(\frac{\mu}{\sigma}\right)^{2}\left((T+1-t)-\frac{1}{T+1-t}\right).             
	\end{split}
\end{equation}
Finally, it follows from (\ref{ansatz}) and (\ref{ab}) that $\partial^{2}_{x}V<0$ indeed holds. By some standard verification arguments, the (\ref{strategy}) follows. { Moreover, by the strict concavity of the exponential utility, we obtain that the optimal control $\left(\Pi^{*},C^{*}\right)$ is unique.} 
\end{proof}
Let $X^{*}=\left\{X^{*}_{t},0\leq t\leq T\right\}$ be the wealth process under the optimal investment-consumption strategy in (\ref{strategy}) that
\begin{equation}\label{a}
	dX^{*}_{t}=\Pi^{*}(t)\mu dt+\Pi^{*}(t)\sigma dW_{t}-C^{*}(t,X^{*}_{t})dt,\quad X^{*}_{0}=x_{0}.
\end{equation}
Recalling that the consistency conditions for $\overline{Z}$ and $\overline{X}_{T}$ are
\begin{equation}\label{fixexp}
	\begin{split}
		&\overline{Z}_{t}=e^{-\delta t}\left\{z_{0}+\int_{0}^{t}\delta e^{\delta s}\mathbb{E}\left[C^{*}(s,X^{*}_{s})\right]ds\right\},\quad t\in[0,T], \\
		&\overline{X}_{T} = \mathbb{E}\left[ X^{*}_{T}\right],
	\end{split}
\end{equation}
which are equivalent to 
\begin{equation}\label{orgfix}
	\begin{split}
		&d\overline{Z}_{t}=-\delta\left\{\overline{Z}_{t}-\mathbb{E}\left[C^{*}(t,X^{*}_{t})\right]\right\}dt, \quad t\in[0,T],\,\, Z_{0}=z_{0},\\
		&\overline{X}_{T} =\mathbb{E}\left[ X^{*}_{T}\right].
	\end{split}
\end{equation}
For a given $\left(\overline{X}_{T},\overline{Z}\right)\in\mathbb{R}\times\mathcal{C}_{T}$, we can solve $X^{*}$ explicitly
\begin{equation}
	\begin{split}
		&X^{*}_{t}=x_{0}\frac{T+1-t}{T+1}+\left(T+1-t\right)\int_{0}^{t}\left[\frac{\theta}{\left(T+1-s\right)^{2}}\left(\int_{s}^{T}\overline{Z}_{v}dv+\overline{X}_{T}\right)-\frac{\theta}{T+1-s}\overline{Z}_{s}\right.\\&\left.+\frac{1}{4}\left(\frac{\mu}{\sigma}\right)^{2}\beta\frac{1}{\left(T+1-s\right)^{2}}+\frac{3}{4}\left(\frac{\mu}{\sigma}\right)^{2}\beta\right]ds+\left(T+1-t\right)\frac{\mu}{\sigma}\beta W_{t}.\\
	\end{split}
\end{equation}
Therefore, we get 
\begin{equation*}
	\begin{split}
		f_{i}(t,\overline{X}_{T},\overline{Z})&:=\mathbb{E}\left[X^{*}_{t}|o=o_{i}\right]=x_{0}\frac{T+1-t}{T+1}+\left(T+1-t\right)\int_{0}^{t}\left[\frac{\theta_{i}}{\left(T+1-s\right)^{2}}\left(\int_{s}^{T}\overline{Z}_{v}dv+\overline{X}_{T}\right)\right.\\&\left.-\frac{\theta_{i}}{T+1-s}\overline{Z}_{s}+\frac{1}{4}\left(\frac{\mu_{i}}{\sigma_{i}}\right)^{2}\beta_{i}\frac{1}{\left(T+1-s\right)^{2}}+\frac{3}{4}\left(\frac{\mu_{i}}{\sigma_{i}}\right)^{2}\beta_{i}\right]ds,
	\end{split}
\end{equation*}
and 
\begin{equation}\label{Cexp}
	\begin{split}
		&\mathbb{E}\left[C^{*}(t,X^{*}_{t})\right]=\frac{\left(x_{0}-\mathbb{E}\left[\theta\right]\overline{X}_{T}\right)}{T+1}+\mathbb{E}\left[\theta\right]\overline{Z}_{t}+\int_{0}^{t}\frac{\mathbb{E}\left[\theta\right]}{\left(T+1-s\right)^{2}}\int_{s}^{T}\overline{Z}_{v}dvds\\&-\int_{0}^{t}\frac{\mathbb{E}\left[\theta\right]}{T+1-s}\overline{Z}_{s}ds-\frac{\mathbb{E}\left[\theta\right]}{T+1-t}\int_{t}^{T}\overline{Z}_{s}ds+\frac{1}{4}\mathbb{E}\left[\left(\frac{\mu}{\sigma}\right)^{2}\beta\right]\left[T+1+2t-\frac{1}{T+1}\right].
	\end{split}
\end{equation}
Thus, we have from (\ref{orgfix})
\	\begin{equation}\label{general fixed}
\begin{split}
	&d\overline{Z}_{t}=-\delta\left\{\overline{Z}_{t}-\mathbb{E}\left[C^{*}(t,X^{*}_{t})\right]\right\}dt, \quad t\in[0,T],\,\, Z_{0}=z_{0},\\
	&\overline{X}_{T}=\mathbb{E}\left[f(T,\overline{X}_{T},\overline{Z})\right],
\end{split}
\end{equation}
where \begin{equation}
	\begin{split}
		&f(T,\overline{X}_{T},\overline{Z}):=\frac{x_{0}}{T+1}+\int_{0}^{T}\left[\frac{\theta}{\left(T+1-s\right)^{2}}\left(\int_{s}^{T}\overline{Z}_{v}dv+\overline{X}_{T}\right)\right.\\&\left.-\frac{\theta}{T+1-s}\overline{Z}_{s}+\frac{1}{4}\left(\frac{\mu}{\sigma}\right)^{2}\beta\frac{1}{\left(T+1-s\right)^{2}}+\frac{3}{4}\left(\frac{\mu}{\sigma}\right)^{2}\beta\right]ds.
	\end{split}
\end{equation}
By the second equation of (\ref{general fixed}), we can decouple the fixed point system to obtain $\overline{X}_{T}$: 
\begin{equation}\label{decouple}
	\begin{split}
		\overline{X}_{T}&=\frac{x_{0}}{\left(1-\mathbb{E}\left[\theta\right]\right)T+1}+\frac{1}{4}\mathbb{E}\left[\left(\frac{\mu}{\sigma}\right)^{2}\beta\right]\left(\frac{3T^{2}+4T}{\left(1-\mathbb{E}\left[\theta\right]\right)T+1}\right)\\&+\frac{\mathbb{E}\left[\theta\right]\left(T+1\right)}{\left(1-\mathbb{E}\left[\theta\right]\right)T+1}\int_{0}^{T}\left(\frac{1}{(T+1-s)^{2}}\int_{s}^{T}\overline{Z}_{v}dv-\frac{1}{T+1-s}\overline{Z}_{s}\right)ds.
	\end{split}
\end{equation}
Substituting (\ref{decouple}) into (\ref{Cexp}), the fixed point system reduces to  
\begin{equation}\label{new fixed point}
	\begin{split}
		&d\overline{Z}_{t}=\left[\left(-\delta+\delta\mathbb{E}\left[\theta\right]\right)\overline{Z}_{t}+\delta G(t,\overline{Z})\right]dt, \\
		&Z_{0}=z_{0},
	\end{split}
\end{equation}
where
\begin{equation}\label{G}
	\begin{split}
		G(t,\overline{Z})&:=\frac{x_{0}}{T+1}\left(1-\frac{\mathbb{E}\left[\theta\right]}{\left(1-\mathbb{E}\left[\theta\right]\right)T+1}\right)+\int_{0}^{t}\frac{\mathbb{E}\left[\theta\right]}{\left(T+1-s\right)^{2}}\int_{s}^{T}\overline{Z}_{v}dvds\\&-\frac{\left(\mathbb{E}\left[\theta\right]\right)^{2}}{\left(1-\mathbb{E}\left[\theta\right]\right)T+1}\int_{0}^{T}\left(\frac{1}{(T+1-s)^{2}}\int_{s}^{T}\overline{Z}_{v}dv-\frac{1}{T+1-s}\overline{Z}_{s}\right)ds\\&-\int_{0}^{t}\frac{\mathbb{E}\left[\theta\right]}{T+1-s}\overline{Z}_{s}ds-\frac{\mathbb{E}\left[\theta\right]}{T+1-t}\int_{t}^{T}\overline{Z}_{s}ds+\frac{1}{4}\mathbb{E}\left[\left(\frac{\mu}{\sigma}\right)^{2}\beta\right]\left[T+1+2t-\frac{1}{T+1}\right]\\&-\frac{1}{4}\frac{\mathbb{E}\left[\theta\right]}{T+1}\mathbb{E}\left[\left(\frac{\mu}{\sigma}\right)^{2}\beta\right]\left(\frac{3T^{2}+4T}{\left(1-\mathbb{E}\left[\theta\right]\right)T+1}\right).
	\end{split}
\end{equation}
We then have the next result.
\begin{proposition}\label{prop2.3}
	There exists a unique deterministic fixed point $\overline{Z}\in\mathcal{C}_{T}$ to (\ref{new fixed point}), and hence $\left(\Pi^{*},C^{*}\right)$, defined via the strategy (\ref{strategy}),  is { the unique} mean field equilibrium.
\end{proposition}
\begin{proof}
Let us define $\Phi(\cdot,\cdot) $: for all $\left(t,Z\right)\in[0,T]\times\mathcal{C}_{T}$,
\begin{equation}\label{Phi}
	\Phi(t,Z):=z_{0}+\int_{0}^{t}\left[\left(-\delta+\delta\mathbb{E}\left[\theta\right]\right)Z_{s}+\delta G(s,Z)\right]ds.
\end{equation}
In light of (\ref{G}), for any $Z^{1}$, $Z^{2}\in\mathcal{C}_{T}$, we have 
\begin{equation}
	||G(\cdot,Z^{1})-G(\cdot,Z^{2})||_{\mathcal{C}_{T}}\leq C(T)||Z^{1}-Z^{2}||_{\mathcal{C}_{T}},
\end{equation}
where $T\rightarrow C(\cdot)$ is continuous and independent of $z_{0}$, and it satisfies $\lim_{T\rightarrow0}C(T)=0$. 
Then 
\begin{equation}\label{contraction}
	||\Phi(\cdot,Z^{1})-\Phi(\cdot,Z^{2})||_{\mathcal{C}_{T}}\leq C_{2}(T)||Z^{1}-Z^{2}||_{\mathcal{C}_{T}},
\end{equation}
where $C_{2}(T):=\left(\delta+\delta\mathbb{E}\left[\theta\right]\right)T+\delta TC(T)$ and  satisfies $\lim_{T\rightarrow0}C_{2}(T)=0$. We can choose $t_{1}\in(0,T]$  small enough such that $C_{2}(t_{1})\in(0,1)$. Thus, $\Phi$ is a contraction map on $\mathcal{C}_{t_{1}}$ by (\ref{contraction}), and there exists a unique fixed point of $\Phi$ on $[t_{0},t_{1}]$ with $t_{0}=0$. As $C_{2}(\cdot)$ is independent of $z_{0}$, we apply this similar argument to conclude that there exists a unique fixed point of $\Phi$ on $[t_{1},t_{2}]$ for some $t_{2}>t_{1}$. Repeating this procedure, we obtain the existence of a unique fixed point $\overline{Z}$ of $\Phi$ on $[0,T]$. Hence, $\left(\Pi^{*},C^{*}\right)$, defined via (\ref{strategy}), is a mean field equilibrium, where $\overline{Z}$ is the unique solution of (\ref{new fixed point}) and $\overline{X}_{T}$ is given by (\ref{decouple}).
\end{proof}
\subsection{Power Utility}
In the power utility case, the wealth process of a representative agent is
\begin{equation}\label{wealthrep.2power}
	\frac{dX_{t}}{X_{t}}=\pi_{t}\mu dt+\pi_{t}\sigma dW_{t}-c_{t}dt,\quad X_{0}=x_{0}.
\end{equation}
Now, the $\mathcal{F}_{0}$-measurable random type vector $o=(\mu,\sigma,p,\theta)$ has a discrete distribution $m$, such that $$m\left(o=o_{k}\right)=F\left(\left\{k\right\}\right)$$ for each $o_{k}=(\mu_{k},\sigma_{k},p_{k},\theta_{k})\in\mathcal{O}_{p}$. 

 Similarly, $\overline{Z}^{n}$ and $\overline{X}^{n}_{T}$ should be approximated by a deterministic function $\overline{Z}=(\overline{Z}_{t})_{t\in[0,T]}\in\mathcal{C}_{T,+}:=C([0,T];\mathbb{R}_{+})$ and a constant $\overline{X}_{T}\in\mathbb{R}_{+}$ respectively when $n$ is sufficiently large.

Given  $\overline{Z}\in\mathcal{C}_{T,+}$ and $\overline{X}_{T}\in\mathbb{R}_{+}$, the dynamic version of the objective function for the representative agent is 
\begin{equation}\label{obf.2power}
	\begin{split}
		J^{r}((\pi,c),t,x;\overline{Z},\overline{X}_{T}):=\mathbb{E}_{t,x}\left[\int_{t}^{T}\frac{\left(c_{s}X_{s}\right)^{p}}{p\left(\overline{Z}_{s}\right)^{p\theta}}ds+\frac{\left(X_{T}\right)^{p}}{p\left(\overline{X}_{T}\right)^{p\theta}}\right].
	\end{split}
\end{equation}
 The stochastic control problem is then given by 
\begin{equation}\label{scppower}
	\sup_{\left(\pi,c\right)\in\mathcal{A}_{t}^{r,p}}J^{r}((\pi,c),t,x;\overline{Z},\overline{X}_{T}) = \mathbb{E}_{m}\left[V(t,x,o)\right]:=\int_{\mathcal{O}}V(t,x,o)m(do).
\end{equation}
 The definition of mean field equilibrium can be obtained in a similar fashion.
\begin{definition}{(Mean Field Equilibrium)}\label{mfepower}
	
	For given $\overline{Z}\in\mathcal{C}_{T,+}$ and $\overline{X}_{T}\in\mathbb{R}_{+}$, let $(\pi^{*},c^{*})\in\mathcal{A}_{0}^{r,p}$ be the best response  to the stochastic control problem (\ref{scppower}) from the initial time $t=0$.    
	The control $(\pi^{*},c^{*})$ is called a mean field equilibrium, if it is the best response  to  itself in the sense that $\overline{Z}_{t} =z_{0}e^{-\delta t}+\int_{0}^{t}\delta e^{\delta (s-t)}\mathbb{E}\left[c^{*}_{s}X^{*}_{s}\right]ds,\ t\in[0,T]$ and $\overline{X}_{T} = \exp\left\{\mathbb{E}\left[\log X^{*}_{T}\right]\right\}$, where $X^{*}$ is the wealth process under the optimal control $(\pi^{*},c^{*})$ with $X^{*}_{0}=x_{0}$.
\end{definition}

Likewise, we shall first solve the stochastic control problem (\ref{scppower}) for given $\overline{Z}\in\mathcal{C}_{T,+}$ and $\overline{X}_{T}\in\mathbb{R}_{+}$.  Let $o=\left(\mu,\sigma,p,\theta\right)$ now represent a deterministic sample from its random type distribution $m$.  The associated HJB equation of value function $V^{r}(t,x):=V(t,x,o)$ is 
\begin{equation}\label{HJBpower}
	\partial_{t}V^{r}(t,x)+\sup_{\pi\in\mathbb{R}}\left(\mu\pi x\partial_{x}V^{r}+\frac{\sigma^{2}}{2}\pi^{2}x^{2}\partial^{2}_{x}V^{r}\right)+\sup_{c\geq0}\left(-cx\partial_{x}V^{r}+\frac{1}{p}c^{p}x^{p}\left(\overline{Z}_{t}\right)^{-\theta p}\right)=0
\end{equation}
with the terminal condition $V^{r}(T,x)=\frac{x^{p}}{p\left(\overline{X}_{T}\right)^{p\theta}}$ for all $x>0$. 
\begin{lemma}\label{lemmapower}
	Given $\overline{Z}=(\overline{Z}_{t})_{t\in[0,T]}\in\mathcal{C}_{T,+}$ and $\overline{X}_{T}>0$, the classical solution to the HJB equation (\ref{HJBpower}) admits the following closed-form 
	\begin{equation}
		V^{r}(t,x) = \frac{1}{p}x^{p}g(t),\quad t\in[0,T],
	\end{equation}
	where \begin{equation}
		g(t) = \left(e^{a\left(T-t\right)}\left(\frac{1}{\left(\overline{X}_{T}\right)^{p\theta}}\right)^{\frac{1}{1-p}}+e^{-at}\int_{t}^{T}e^{as}\left(\overline{Z}_{s}\right)^{\frac{\theta p}{p-1}}ds\right)^{1-p},   \quad a:=\frac{1}{2}\frac{\mu^{2}}{\sigma^{2}}\frac{p}{\left(1-p\right)^{2}}>0.
	\end{equation}
	{ There exists a unique optimal control to Problem (\ref{scppower}) and the feedback functions of the optimal investment and consumption  are given by}
	\begin{equation}\label{strategypower}
		\pi^{*}(t,x) = \frac{\mu}{(1-p)\sigma^{2}}, \quad c^{*}(t,x) =\left(\overline{Z}_{t}\right)^{\frac{\theta p}{p-1}}g(t)^{\frac{1}{p-1}},\quad t\in[0,T].
	\end{equation}
\end{lemma}
\begin{proof}
The proof is similar to the proof of  Lemma \ref{lemma7.8}, and we omit it here.
\end{proof}
Let $X^{*}=\left\{X^{*}_{t},0\leq t\leq T\right\}$ be the wealth process under the optimal investment-consumption strategy in (\ref{strategypower}):
\begin{equation}\label{apower}
	\frac{dX^{*}_{t}}{X^{*}_{t}}=\pi^{*}_{t}\mu dt+\pi^{*}_{t}\sigma dW_{t}-c^{*}_{t}dt,\quad X^{*}_{0}=x_{0}.
\end{equation}
 The consistency conditions for $\overline{Z}$ and $\overline{X}_{T}$ are
\begin{equation}
	\begin{split}
		&\overline{Z}_{t}=e^{-\delta t}\left\{z_{0}+\int_{0}^{t}\delta e^{\delta s}\mathbb{E}\left[c^{*}_{s}X^{*}_{s}\right]ds\right\},\quad t\in[0,T], \\
		&\overline{X}_{T} = \exp\left\{\mathbb{E}\left[\log X^{*}_{T}\right]\right\},
	\end{split}
\end{equation}
which are equivalent to 
\begin{equation}\label{orgfixpower}
	\begin{split}
		&d\overline{Z}_{t}=-\delta\left\{\overline{Z}_{t}-\mathbb{E}\left[c^{*}_{t}X^{*}_{t}\right]\right\}dt, \quad t\in[0,T],\,\, Z_{0}=z_{0},\\
		&\overline{X}_{T} = \exp\left\{\mathbb{E}\left[\log X^{*}_{T}\right]\right\}.
	\end{split}
\end{equation}
For a given $\left(\overline{X}_{T},\overline{Z}\right)\in\mathbb{R}_{+}\times\mathcal{C}_{T,+}$, it follows that
\begin{equation}
	\begin{split}
		&X^{*}_{t}=x_{0}\exp\left\{\int_{0}^{t}\left(\pi^{*}_{s}\mu-\frac{\sigma^{2}}{2}\left(\pi^{*}_{s}\right)^{2}-c^{*}_{s}\right)ds+\pi^{*}_{t}\sigma W_{t}\right\},\\
		&\log X^{*}_{t}=\log x_{0} + \int_{0}^{t}\left(\pi^{*}_{s}\mu-\frac{\sigma^{2}}{2}\left(\pi^{*}_{s}\right)^{2}-c^{*}_{s}\right)ds+\pi^{*}_{t}\sigma W_{t}.
	\end{split}
\end{equation}
Therefore, 
\begin{equation*}
	\begin{split}
		&f_{i}(t,\overline{X}_{T},\overline{Z}):=\mathbb{E}\left[X^{*}_{t}|o=o_{i}\right]=x_{0}\exp\left\{\int_{0}^{t}\left(\frac{\mu_{i}^{2}}{(1-p_{i})\sigma_{i}^{2}}-\left(\overline{Z}_{s}\right)^{\frac{\theta_{i}p_{i}}{p_{i}-1}}G_{i}(s,\overline{X}_{T},\overline{Z})\right)ds\right\},\\
		&\mathbb{E}\left[\log X^{*}_{t}|o=o_{i}\right]=\log x_{0}+\int_{0}^{t}\left(\frac{\mu_{i}^{2}}{2\sigma_{i}^{2}}\frac{\left(1-2p_{i}\right)}{(1-p_{i})^{2}}-\left(\overline{Z}_{s}\right)^{\frac{\theta_{i}p_{i}}{p_{i}-1}}G_{i}(s,\overline{X}_{T},\overline{Z})\right)ds,
	\end{split}
\end{equation*}
where \begin{equation*}
	\begin{split}
		&G_{i}(t,\overline{X}_{T},\overline{Z}): = \left(e^{a_{i}\left(T-t\right)}\left(\frac{1}{\left(\overline{X}_{T}\right)^{p_{i}\theta_{i}}}\right)^{\frac{1}{1-p_{i}}}+e^{-a_{i}t}\int_{t}^{T}e^{a_{i}s}\left(\overline{Z}_{s}\right)^{\frac{\theta_{i} p_{i}}{p_{i}-1}}ds\right)^{-1},\\ &a_{i}:=\frac{1}{2}\frac{\mu_{i}^{2}}{\sigma_{i}^{2}}\frac{p_{i}}{\left(1-p_{i}\right)^{2}}.
	\end{split}
\end{equation*}
Thus, we have from (\ref{orgfixpower})
\	\begin{equation}\label{general fixedpower}
	\begin{split}
		&d\overline{Z}_{t}=-\delta\left\{\overline{Z}_{t}-\mathbb{E}_{m}\left[\left(\overline{Z}_{t}\right)^{\frac{\theta p}{p-1}}	G(t,\overline{X}_{T},\overline{Z})f(t,\overline{X}_{T},\overline{Z})\right]\right\}dt, \quad t\in[0,T],\,\, Z_{0}=z_{0},\\
		&\overline{X}_{T} =x_{0}\exp\left\{\int_{0}^{T}\mathbb{E}_{m}\left[\frac{\mu^{2}}{2\sigma^{2}}\frac{\left(1-2p\right)}{(1-p)^{2}}-\left(\overline{Z}_{s}\right)^{\frac{\theta p}{p-1}}	G(s,\overline{X}_{T},\overline{Z})\right]ds\right\},
	\end{split}
\end{equation}
where \begin{equation}
	\begin{split}
		&G(t,\overline{X}_{T},\overline{Z}):= \left(e^{a\left(T-t\right)}\left(\frac{1}{\left(\overline{X}_{T}\right)^{p\theta}}\right)^{\frac{1}{1-p}}+e^{-at}\int_{t}^{T}e^{as}\left(\overline{Z}_{s}\right)^{\frac{\theta p}{p-1}}ds\right)^{-1},   \\
		&f(t,\overline{X}_{T},\overline{Z}):=x_{0}\exp\left\{\int_{0}^{t}\left(\frac{\mu^{2}}{(1-p)\sigma^{2}}-\left(\overline{Z}_{s}\right)^{\frac{\theta p}{p-1}}G(t,\overline{X}_{T},\overline{Z})\right)ds\right\}.
	\end{split}
\end{equation}
 Unlike the exponential utility case, we can not  decoupled the fixed point system (\ref{general fixedpower}) which makes it more complicated to prove the existence of fixed points.

Given a positive constant $\epsilon$ and a positive continuous function $M(t)$, define 
\begin{equation}
	\begin{split}
		\mathcal{C}_{T,\epsilon,M(\cdot)}:=\left\{\overline{Z}\in\mathcal{C}_{T}: \epsilon\leq\overline{Z}_{t}\leq M(t),\,\, \forall t\in[0,T] \right\}.
	\end{split}
\end{equation}

The next proposition gives the existence of the fixed point of (\ref{general fixedpower}) over a time interval of arbitrarily prescribed length.

\begin{proposition}\label{prop2.3power}
	The system (\ref{general fixedpower}) has a solution $\left(\overline{X}_{T},\overline{Z}\right)\in\mathbb{R}_{+}\times\mathcal{C}_{T,+}$, and hence $\left(\pi^{*},c^{*}\right)$, defined via the strategy (\ref{strategypower}),  is a mean field equilibrium.
\end{proposition}
\begin{proof}
	Define 
	\begin{equation}
		\hat{Z}_{t}:=\exp\left(\delta t\right)\overline{Z}_{t}, t\in[0,T].
	\end{equation}
	Then, we have from (\ref{general fixedpower}) 
	\begin{equation}\label{new fixedpower}
		\begin{split}
			&d\hat{Z}_{t}=\delta\sum_{k=1}^{K}e^{\beta_{k}t}\left(\hat{Z}_{t}\right)^{\frac{\theta_{k}p_{k}}{p_{k}-1}}\hat{G}_{k}(t,\overline{X}_{T},\hat{Z})\hat{f}_{k}(t,\overline{X}_{T},\hat{Z})F(\left\{k\right\})dt,\quad \forall t\in[0,T],\\ &\hat{Z}_{0}=z_{0},\\
			&\overline{X}_{T}=C_{2}\exp\left\{-\int_{0}^{T}\sum_{k=1}^{K}\exp\left(\frac{\theta_{k} p_{k}\delta}{1-p_{k}}t\right)\left(\hat{Z}_{t}\right)^{\frac{\theta_{k}p_{k}}{p_{k}-1}}\hat{G}_{k}(t,\overline{X}_{T},\hat{Z})dt\right\},
		\end{split}
	\end{equation}
	where \begin{equation}
		\begin{split}\label{7.17power}
			&\beta_{k}:=\frac{\delta }{1-p_{k}}\left(1+\theta_{k}p_{k}-p_{k}\right),\\
			&C_{2}:=x_{0}\exp\left\{\sum_{k=1}^{K}F(\left\{k\right\})\frac{\mu_{k}^{2}}{2\sigma_{k}^{2}}\frac{\left(1-2p_{k}\right)}{(1-p_{k})^{2}}T\right\},
		\end{split}
	\end{equation}
	and
	\begin{equation*}
		\begin{split}
			\hat{G}_{k}(t,\overline{X}_{T}&,\hat{Z})\\:=&\left(e^{a_{k}\left(T-t\right)}\left(\frac{1}{\left(\overline{X}_{T}\right)^{p_{k}\theta_{k}}}\right)^{\frac{1}{1-p_{k}}}+e^{-a_{k}t}\int_{t}^{T}e^{a_{k}s}\exp\left\{\frac{\theta_{k}p_{k}\delta}{1-p_{k}}s\right\}\left(\hat{Z}_{s}\right)^{\frac{\theta_{k}p_{k}}{p_{k}-1}}ds\right)^{-1}, \\
			\hat{f}_{k}(t,\overline{X}_{T}&,\hat{Z})\\:=&x_{0}\exp\left(\frac{\mu_{k}^{2}}{(1-p_{k})\sigma_{k}^{2}}t\right)\exp\left\{-\int_{0}^{t}\exp\left(\frac{\theta_{k} p_{k}\delta}{1-p_{k}}s\right)\left(\hat{Z}_{s}\right)^{\frac{\theta_{k}p_{k}}{p_{k}-1}}\hat{G}_{k}(s,\overline{X}_{T},\hat{Z})ds\right\}.
		\end{split}
	\end{equation*}
	Now, it is enough to consider the system (\ref{new fixedpower}). First, let us introduce the following functions:
	\begin{equation}
		\begin{split}
			M(t):=EKt+z_{0},
		\end{split}
	\end{equation}
	where $E$ is given by 
	\begin{equation*}
		E:=\delta x_{0}\max_{1\leq k\leq K}e^{-a_{k}T}z_{0}^{\frac{\theta_{k}p_{k}}{p_{k}-1}}C_{2}^{\frac{p_{k}\theta_{k}}{1-p_{k}}}\left(\max_{t\in[0,T]}\exp\left\{\left(\beta_{k}+a_{k}+\frac{\mu_{k}^{2}}{(1-p_{k})\sigma_{k}^{2}}\right)t\right\}\right).
	\end{equation*}
	In  Appendix \ref{5.26.4}, we will show that $M(\cdot)$ is a dominating function of $\hat{Z}$, i.e., $\hat{Z}_{t}\leq M(t)$ for all $t\in[0,T]$.

	For any $\left(\overline{X}_{T}, \hat{Z}\right)\in[C_{0},C_{2}]\times\mathcal{C}_{T,z_{0},M(\cdot)}$, where $C_{0}$ is defined in (\ref{6.21}), let us define $$\Phi\left(\overline{X}_{T}, \hat{Z}\right):=\left(\Phi_{1}	\left(\overline{X}_{T}, \hat{Z}\right),\Phi_{2}	\left(\cdot,\overline{X}_{T}, \hat{Z}\right)\right),$$
	where
	\begin{equation}
		\begin{split}
			&\Phi_{1}\left(\overline{X}_{T}, \hat{Z}\right):=C_{2}\exp\left\{-\int_{0}^{T}\sum_{k=1}^{K}\exp\left(\frac{\theta_{k} p_{k}\delta}{1-p_{k}}t\right)\left(\hat{Z}_{t}\right)^{\frac{\theta_{k}p_{k}}{p_{k}-1}}\hat{G}_{k}(t,\overline{X}_{T},\hat{Z})dt\right\},\\
			&\Phi_{2}\left(\cdot,\overline{X}_{T}, \hat{Z}\right):=z_{0}+\int_{0}^{\cdot}\phi(s,\overline{X}_{T},\hat{Z})ds,
		\end{split}
	\end{equation}
	where \begin{equation*}
		\phi(s,\overline{X}_{T},\hat{Z}):=\delta\sum_{k=1}^{K}e^{\beta_{k}s}\left(\hat{Z}_{s}\right)^{\frac{\theta_{k}p_{k}}{p_{k}-1}}\hat{G}_{k}(s,\overline{X}_{T},\hat{Z})\hat{f}_{k}(s,\overline{X}_{T},\hat{Z})F(\left\{k\right\}).
	\end{equation*}
 We claim: $\Phi(\overline{X}_{T},\hat{Z})\in[C_{0},C_{2}]\times\mathcal{C}_{T,z_{0},M(\cdot)}$. First of all, it is obvious that $\Phi_{1}\left(\overline{X}_{T}, \hat{Z}\right)\leq C_{2}$   and $\Phi_{2}\left(t,\overline{X}_{T}, \hat{Z}\right)\geq z_{0}$ for all $t\in[0,T]$. Similar to the proof of Proposition \ref{5.26.5}, we obtain 
	\begin{equation*}
		\begin{split}
			\Phi_{2}\left(t,\overline{X}_{T}, \hat{Z}\right)&\leq z_{0}+\delta x_{0} \sum_{k=1}^{K}e^{-a_{k}T}z_{0}^{\frac{\theta_{k}p_{k}}{p_{k}-1}}C_{2}^{\frac{p_{k}\theta_{k}}{1-p_{k}}}\int_{0}^{t}\exp\left\{\left(\beta_{k}+a_{k}+\frac{\mu_{k}^{2}}{(1-p_{k})\sigma_{k}^{2}}\right)s\right\}ds\\&\leq z_{0}+EKt=M(t),\quad \forall t\in [0,T].
		\end{split}
	\end{equation*}
Using the same method as in (\ref{6.21}), we obtain
	\begin{equation*}
		\begin{split}
			\Phi_{1}\left(\overline{X}_{T}, \hat{Z}\right)&\geq
			C_{2}\exp\left\{-\int_{0}^{T}\sum_{k=1}^{K}\exp\left(\frac{\theta _{k}p_{k}\delta}{1-p_{k}}s\right)z_{0}^{\frac{\theta_{k}p_{k}}{p_{k}-1}}\right.\\&\left.\left(\int_{s}^{T}e^{a_{k}(v-s)}\exp\left(\frac{\theta_{k}p_{k}\delta}{1-p_{k}}v\right)\left(M(T)\right)^{\frac{\theta_{k}p_{k}}{p_{k}-1}}dv\right)^{-1}ds\right\}\\&=C_{0},
		\end{split}
	\end{equation*}
   and the claim follows.
	
	 We aim to use Schauder's fixed point theorem to show that there exists at least one solution $\left(\overline{X}_{T}, \hat{Z}\right)\in[C_{0},C_{2}]\times\mathcal{C}_{T,z_{0},M(\cdot)}$ to the system (\ref{new fixedpower}).
	
	First, we shall prove $\Phi$ is continuous. Indeed, it is straightforward to verify that $\left\{\hat{G}_{k}\right\}_{k=1}^{K}$, $\left\{\hat{f}_{k}\right\}_{k=1}^{K}$ and $\phi$ are (uniformly) bounded. Then, for a given sequence $\left(\overline{X}^{n}_{T}, \hat{Z}^{n}\right)$ converging to some $\left(\overline{X}_{T}, \hat{Z}\right)$ in the sense of 
	\begin{equation}
		\vert\overline{X}^{n}_{T}-\overline{X}_{T}\vert+\vert\vert\hat{Z}^{n}-\hat{Z}\vert\vert_{\mathcal{C}_{T}}\rightarrow0,
	\end{equation}
	it is not hard to show 
	\begin{equation}
		\vert\Phi_{1}(\overline{X}^{n}_{T},\hat{Z}^{n})-\Phi_{1}(\overline{X}_{T},\hat{Z})\vert+\vert\vert\Phi_{2}(\cdot,\overline{X}^{n}_{T},\hat{Z}^{n})-\Phi_{2}(\cdot,\overline{X}_{T},\hat{Z})\vert\vert_{\mathcal{C}_{T}}\rightarrow0.
	\end{equation}
	
	Next, for a given bounded subset $Y_{1}\times Y_{2}\in[C_{0},C_{2}]\times\mathcal{C}_{T,z_{0},M(\cdot)}$, we shall prove that $\Phi( Y_{1},Y_{2})$ is relatively compact. It is enough to show that $\Phi_{2}$ is equicontinuous. In fact, for any given $\left(\overline{X}_{T},\hat{Z}\right)\in[C_{0},C_{2}]\times\mathcal{C}_{T,z_{0},M(\cdot)}$ and $s<t$, we have 
	\begin{equation}
		\begin{split}
			\left|\Phi_{2}\left(t,\overline{X}_{T}, \hat{Z}\right)-\Phi_{2}\left(s,\overline{X}_{T}, \hat{Z}\right)\right|&=\left|\int_{s}^{t}\phi(s,\overline{X}_{T},\hat{Z})ds\right|\\&\leq\int_{s}^{t}\left|\phi(s,\overline{X}_{T},\hat{Z})\right|ds,
		\end{split}
	\end{equation}
	where $\left|\phi(s,\overline{X}_{T},\hat{Z})\right|$ is bounded by 
	\begin{equation}
		\left|\phi(s,\overline{X}_{T},\hat{Z})\right|\leq EK, \quad \forall s\in[0,T],
	\end{equation}
	which completes the proof.
\end{proof}
\begin{remark}
	Proposition \ref{prop2.3power} does not guarantee the uniqueness of the fixed point.
	However, it is not hard to get 
	\begin{equation*}
		\begin{split}||&\Phi(\cdot,\overline{X}^{1}_{T}, \hat{Z}^{1})-\Phi(\cdot,\overline{X}^{2}_{T}, \hat{Z}^{2})||=
			||\Phi_{2}(\cdot,\overline{X}^{1}_{T}, \hat{Z}^{1})-\Phi_{2}(\cdot,\overline{X}^{2}_{T}, \hat{Z}^{2})||_{\mathcal{C}_{T}}\\&+\left|\Phi_{1}\left(\overline{X}^{1}_{T}, \hat{Z}^{1}\right)-\Phi_{1}\left(\overline{X}^{2}_{T}, \hat{Z}^{2}\right)\right|\leq C(T)\left(||\hat{Z}^{1}-\hat{Z}^{2}||_{\mathcal{C}_{T}}+\left|\overline{X}^{1}_{T}-\overline{X}^{2}_{T}\right|\right),
		\end{split}
	\end{equation*}
	where $C(T)$ is a positive continuous function satisfying that $\lim_{T\rightarrow0}C(T)=0$. Therefore, the existence and uniqueness should hold for short time. Due to the forward-backward-like property of system  (\ref{general fixedpower}), a local solution can not be extended to a global solution as we does in the exponential case.
\end{remark}

\section{Approximate Nash Equilibrium in N-agent Games}\label{sect4}
In this section, we will construct and verify an approximate Nash equilibrium using the mean field equilibrium for both exponential and power cases.  Let us introduce
\begin{equation}
\epsilon_{n}:= \sup_{1\leq k\leq K}|F_{n}(\left\{k\right\})-F(\left\{k\right\})|,
\end{equation}
which measures the gap between the empirical distribution $F_{n}$ and its limit $F$. In order to derive the explicit convergence rate in terms of $n$, we need the following assumption:
\begin{assumption}\label{assump4}
	Assume that  $\epsilon_{n}=O(\frac{1}{\sqrt{n}})$ when $n$ is sufficient large.
\end{assumption}
\subsection{Exponential Utility}
 For $i\in\left\{1,\cdots,n\right\}$, we recall that the objective function of agent $i$ with exponential utility can be rewritten as:
\begin{equation}
		J_{i}\left((\Pi^{i},C^{i}),(\mathbf{\Pi},\mathbf{C})^{-i}\right):=\mathbb{E}\left[\int_{0}^{T}U_{\alpha(i)}\left(C^{i}_{t}-\theta_{\alpha(i)}\overline{Z}^{n}_{t}\right)dt+U_{\alpha(i)}\left(X^{i}_{T}-\theta_{\alpha(i)}\overline{X}_{T}^{n}\right)\right],
\end{equation}
where the vector $(\mathbf{\Pi},\mathbf{C})^{-i}$ is defined by 
\begin{equation}
	(\mathbf{\Pi},\mathbf{C})^{-i}:=\left((\Pi^{1}, C^{1}),\cdots,(\Pi^{i-1},C^{i-1}),(\Pi^{i+1},C^{i+1}),\cdots,(\Pi^{n},C^{n})\right).
\end{equation}
The definition of an approximate Nash equilibrium is given as follows.
\begin{definition}{(Approximate Nash equilibrium)}\label{Appnash}
	Let $\mathcal{A}^{e}:=\prod_{i=1}^{n}\mathcal{A}_{0}^{i,e}$. An $n$-tuple of admissible controls $\left(\mathbf{\Pi}^{*},\mathbf{C}^{*}\right):=\left((\Pi^{*,i},C^{*,i})\right)_{i=1}^{n}\in\mathcal{A}^{e}$ is called an $\epsilon$-Nash equilibrium to the n-agent game, if for any $(\Pi^{i},C^{i})\in\mathcal{A}^{i,e}_{0}$, it holds that 
	\begin{equation}\label{appnash}
	J_{i}\left((\Pi^{i},C^{i}),(\mathbf{\Pi}^{*},\mathbf{C}^{*})^{-i}\right)\leq J_{i}\left(\mathbf{\Pi}^{*},\mathbf{C}^{*}\right)+\epsilon,\quad\forall i\in\left\{1,\cdots,n\right\}.
	\end{equation}
\end{definition}
For $i\in\left\{1,\cdots,n\right\}$, let us construct the following candidate investment and consumption strategy $\left(\Pi^{*,i},C^{*,i}\right)$: 
\begin{equation}\label{candi}
	\begin{split}
		&\Pi^{*,i}(t):=\beta_{\alpha(i)}\frac{\mu_{\alpha(i)}}{\sigma^{2}_{\alpha(i)}}(T+1-t) ,\\
		 &C^{*,i}(t,x):=\frac{(x-\theta_{\alpha(i)}\overline{X}_{T})}{T+1-t}+\theta_{\alpha(i)}\overline{Z}_{t}
-\frac{\theta_{\alpha(i)}}{T+1-t}\int_{t}^{T}\overline{Z}_{s}ds\\&+\frac{1}{4}\beta_{\alpha(i)}
\left(\frac{\mu_{\alpha(i)}}{\sigma_{\alpha(i)}}\right)^{2}\left((T+1-t)-\frac{1}{T+1-t}\right),
	\end{split}
\end{equation}
where $\left(\overline{Z},\overline{X}_{T}\right)$ is the fixed point of (\ref{orgfix}). Denote by $X^{*,i}$ the wealth process of agent $i$ under the candidate strategy $\left(\Pi^{*,i},C^{*,i}\right)$ in (\ref{candi}):
\begin{equation}\label{wealth}
	dX^{*,i}_{t}=\Pi^{*,i}(t)\mu_{\alpha(i)}dt+\Pi^{*,i}(t)\sigma_{\alpha(i)}dW^{i}_{t}-C^{*,i}(t,X^{*,i}_{t})dt, \quad X^{*,i}_{0}=x_{0}.
\end{equation}
 The $i$-th agent's habit formation process is then given by 
\begin{equation}\label{z}
		Z^{*,i}_{t} = e^{-\delta t}\left(z_{0}+\int_{0}^{t}\delta e^{\delta s}C^{*,i}(s,X^{*,i}_{s})ds\right),\quad t\in[0,T].
\end{equation}
Denote by
\begin{equation}\label{ZX}
	 \overline{Z}_{t}^{*,n}:=\frac{1}{n}\sum_{i=1}^{n}Z^{*,i}_{t},\quad\overline{X}^{*,n}_{T}:=\frac{1}{n}\left(\sum_{i=1}^{n}X^{*,i}_{T}\right).
 \end{equation}
It follows that \begin{equation}\label{Z}
	\overline{Z}^{*,n}_{t}:=e^{-\delta t}z_{0}+\frac{1}{n}\sum_{i=1}^{n}\int_{0}^{t}\delta e^{\delta(s-t)}C^{*,i}(s,X^{*,i}_{s})ds,\,\ t\in[0,T].
\end{equation}
For technical convenience and ease of presentation, we make an additional assumption on $\mathcal{A}^{e}$ in this section:
\begin{assumption}\label{assump5}
 We assume that the $n$-tuple of admissible controls $\left(\Pi^{i},C^{i}\right)_{i=1}^{n}$ satisfies the following integrability conditions: for all $i\in\left\{1,\cdots,n\right\}$,
	\begin{equation*}
		\begin{split}
			&\mathbb{E}\left[\int_{0}^{T}\left|U_{\alpha(i)}\left(C^{i}_{t}-\theta_{\alpha(i)}\overline{Z}^{*,n}_{t}\right)\right|^{2}dt+\left|U_{\alpha(i)}\left(X^{i}_{T}-\theta_{\alpha(i)}\overline{X}_{T}^{*,n}\right)\right|^{2}\right]<M<\infty,\\
			&	\mathbb{E}\left[\int_{0}^{T}\left|U_{\alpha(i)}\left(C^{i}_{t}\right)\right|^{2}dt+\left|U_{\alpha(i)}\left(X^{i}_{T}\right)\right|^{2}\right]<M<\infty,
		\end{split}
	\end{equation*}
where $M$ is a constant independent of $n$.
\end{assumption}
\begin{remark}
Although the above assumption seems to be strong, it does cover a large class of controls.
In fact, it is straightforward to verify that if  $\left(\Pi^{i},C^{i}\right)_{i=1}^{n}$ is the outcome of a deterministic-affine strategy, i.e., $\left(\Pi^{i}_{t},C^{i}_{t}\right)_{i=1}^{n}=\left(\tilde{\Pi}^{i}(t),\tilde{C}^{i}(t,X^{i}_{t})\right)_{i=1}^{n}$ for all $t\in[0,T]$, where $\tilde{\Pi}^{i}:[0,T]\rightarrow\mathbb{R}$ is a continuous function and $\tilde{C}:[0,T]\times\mathbb{R}\rightarrow\mathbb{R}$ is an affine function with respect to $x\in\mathbb{R}$, then $\left(\Pi^{i},C^{i}\right)_{i=1}^{n}$ satisfies the Assumption \ref{assump5}, see Appendix \ref{appendixcara}. In particular, the candidate strategies $\left(\Pi^{*,i},C^{*,i}\right)_{i=1}^{n}$ satisfies the Assumption \ref{assump5}.  
\end{remark}

Next, we introduce the main result of this section on the existence of an approximate Nash equilibrium.
\begin{theorem}\label{6.25}
	
	The pair $\left(\mathbf{\Pi}^{*},\mathbf{C}^{*}\right):=\left\{\left((\Pi^{*,1}(t),C^{*,1}(t,X^{*,1}_{t})),\cdots,(\Pi^{*,n}(t),C^{*,n}(t,X^{*,n}_{t}))\right), 0\leq t\leq 
	T\right\}$ is an $\epsilon_{n}$-Nash equilibrium to the n-agent game with the explicit order $\epsilon_{n}=O(n^{-\frac{1}{2}})$.
\end{theorem}

To prove Theorem \ref{6.25}, we need the following auxiliary results.
\begin{lemma}\label{lemma1} For any $n\geq1$, it holds that
	
	(i) For  $i\in\left\{1,\cdots,n\right\}$, let $Z^{i}$ be the habit formation process associated with an admissible control $\left(\Pi^{i},C^{i}\right)\in\mathcal{A}^{i,e}_{0}$. Then, we have
	\begin{equation}
		\mathbb{E}\left[\sup_{t\in[0,T]}\left|Z^{i}_{t}\right|^{2}\right]<\infty.
	\end{equation}
(ii) Let $\left(\overline{Z},\overline{X}_{T}\right)$ be the fixed point to (\ref{orgfix}) and  $\left(\overline{Z}^{*,n},\overline{X}^{*,n}_{T}\right)$ be defined by (\ref{ZX}). Then, we have
\begin{equation}
	\sup_{t\in[0,T]}\mathbb{E}\left[\left|\overline{Z}^{*,n}_{t}-\overline{Z}_{t}\right|^{2}\right]=O(n^{-1}),\quad 	\mathbb{E}\left[\left|\overline{X}^{*,n}_{T}-\overline{X}_{T}\right|^{2}\right]=O(n^{-1}).
\end{equation}
\end{lemma}
\begin{proof}
\textbf{(i)}  Recall that 
\begin{equation}
	Z^{i}_{t}=e^{-\delta t}z_{0}+\int_{0}^{t}\delta e^{\delta(s-t)}C^{i}_{s}ds.
\end{equation}
We have 
\begin{equation}
	\begin{split}
		|Z^{i}_{t}|^{2}&\leq C\left((z_{0})^{2}+\left(\int_{0}^{t}\delta e^{\delta(s-t)}C^{i}_{s}ds\right)^{2}\right)\\
		&\leq C\left((z_{0})^{2}+\int_{0}^{t}(\delta e^{\delta(s-t)})^{2}ds\int_{0}^{t}(C^{i}_{s})^{2}ds\right)\\
		&\leq C\left((z_{0})^{2}+\int_{0}^{T}(C^{i}_{s})^{2}ds\right),
	\end{split}
\end{equation}
where $C$ is some constant that may vary from one line to another. Therefore, we get 
\begin{equation}
	\mathbb{E}\left[\sup_{t\in[0,T]}\left|Z^{i}_{t}\right|^{2}\right]\leq C\left((z_{0})^{2}+\mathbb{E}\left[\int_{0}^{T}(C^{i}_{s})^{2}ds\right]\right)<\infty.
\end{equation}

\textbf{(ii)}  Using (\ref{fixexp}) and (\ref{Z}), we have  
\begin{equation*}
	\begin{split}
		\overline{Z}^{*,n}_{t}-\overline{Z}_{t}&=\int_{0}^{t}\delta e^{\delta(s-t)}\left(\frac{1}{n}\sum_{i=1}^{n}C^{*,i}(s,X^{*,i}_{s})-\mathbb{E}\left[C^{*}(s,X^{*}_{s})\right]\right)ds\\&=\int_{0}^{t}\delta e^{\delta(s-t)}\left(\frac{1}{n}\sum_{i=1}^{n}C^{*,i}(s,X^{*,i}_{s})-\frac{1}{n}\sum_{i=1}^{n}\mathbb{E}\left[C^{*,i}(s,X^{*,i}_{s})\right]\right)ds\\&+\int_{0}^{t}\delta e^{\delta(s-t)}\left(\frac{1}{n}\sum_{i=1}^{n}\mathbb{E}\left[C^{*,i}(s,X^{*,i}_{s})\right]-\mathbb{E}\left[C^{*}(s,X^{*}_{s})\right]\right)ds.
	\end{split}
\end{equation*}
Then, applying  {H\"{o}lder} inequality, we obtain
\begin{equation*}
	\begin{split}
		\left|\overline{Z}^{*,n}_{t}-\overline{Z}_{t}\right|^{2}&\leq C\left(\int_{0}^{t}\delta e^{\delta(s-t)}\left(\frac{1}{n}\sum_{i=1}^{n}C^{*,i}(s,X^{*,i}_{s})-\frac{1}{n}\sum_{i=1}^{n}\mathbb{E}\left[C^{*,i}(s,X^{*,i}_{s})\right]\right)ds\right)^{2}\\&+C\left(\int_{0}^{t}\delta e^{\delta(s-t)}\left(\frac{1}{n}\sum_{i=1}^{n}\mathbb{E}\left[C^{*,i}(s,X^{*,i}_{s})\right]-\mathbb{E}\left[C^{*}(s,X^{*}_{s})\right]\right)ds\right)^{2}\\&\leq C\left(\int_{0}^{t}\left(\frac{1}{n}\sum_{i=1}^{n}C^{*,i}(s,X^{*,i}_{s})-\frac{1}{n}\sum_{i=1}^{n}\mathbb{E}\left[C^{*,i}(s,X^{*,i}_{s})\right]\right)^{2}ds\right)\\&+C\left(\int_{0}^{t}\left(\frac{1}{n}\sum_{i=1}^{n}\mathbb{E}\left[C^{*,i}(s,X^{*,i}_{s})\right]-\mathbb{E}\left[C^{*}(s,X^{*}_{s})\right]\right)^{2}ds\right).
	\end{split}
\end{equation*}
Taking the expectation of both sides of above inequality yields
\begin{equation*}
	\begin{split}
		\mathbb{E}\left[\left|\overline{Z}^{*,n}_{t}-\overline{Z}_{t}\right|^{2}\right]&\leq C\int_{0}^{t}\mathbb{E}\left[\left(\frac{1}{n}\sum_{i=1}^{n}C^{*,i}(s,X^{*,i}_{s})-\frac{1}{n}\sum_{i=1}^{n}\mathbb{E}\left[C^{*,i}(s,X^{*,i}_{s})\right]\right)^{2}\right]ds\\&+C\int_{0}^{t}\mathbb{E}\left[\left(\frac{1}{n}\sum_{i=1}^{n}\mathbb{E}\left[C^{*,i}(s,X^{*,i}_{s})\right]-\mathbb{E}\left[C^{*}(s,X^{*}_{s})\right]\right)^{2}\right]ds
	\end{split}
\end{equation*}
Therefore,
\begin{equation}
	\sup_{t\in[0,T]}	\mathbb{E}\left[\left|\overline{Z}^{*,n}_{t}-\overline{Z}_{t}\right|^{2}\right]\leq C\left(\sup_{t\in[0,T]}I^{n}_{1}(t)+\sup_{t\in[0,T]}I^{n}_{2}(t)\right),
\end{equation}
where $I^{n}_{1}(t)$ and $I^{n}_{1}(t)$ are  defined by
\begin{equation}
	\begin{split}
		&I^{n}_{1}(t):= \mathbb{E}\left[\left(\frac{1}{n}\sum_{i=1}^{n}C^{*,i}(t,X^{*,i}_{t})-\frac{1}{n}\sum_{i=1}^{n}\mathbb{E}\left[C^{*,i}(t,X^{*,i}_{t})\right]\right)^{2}\right],\\
		&I^{n}_{2}(t):=\mathbb{E}\left[\left(\frac{1}{n}\sum_{i=1}^{n}\mathbb{E}\left[C^{*,i}(t,X^{*,i}_{t})\right]-\mathbb{E}\left[C^{*}(t,X^{*}_{t})\right]\right)^{2}\right].
	\end{split}
\end{equation}
For $k\in[K]$, let us introduce $\left(\hat{\Pi}^{k}, \hat{C}^{k}\right)$ :
\begin{equation*}
	\begin{cases}
		\hat{\Pi}^{k}(t):=\beta_{k}\frac{\mu_{k}}{\sigma^{2}_{k}}(T+1-t) ,\\
		\hat{C}^{k}(t,x):=\frac{(x-\theta_{k}\overline{X}_{T})}{T+1-t}+\theta_{k}\overline{Z}_{t}
-\frac{\theta_{k}}{T+1-t}\int_{t}^{T}\overline{Z}_{s}ds+\frac{1}{4}\beta_{k}\left(\frac{\mu_{k}}{\sigma_{k}}\right)^{2}\left((T+1-t)
-\frac{1}{T+1-t}\right)
	\end{cases}
\end{equation*}
and an auxiliary process $\hat{X}^{k}$ satisfying 
\begin{equation}
	d\hat{X}^{k}_{t}=\hat{\Pi}^{k}(t)\mu_{k} dt+\hat{\Pi}^{k}(t)\sigma_{k} dW_{t}-\hat{C}^{k}(t,\hat{X}^{k}_{t})dt,\quad \hat{X}_{0}^{k}=x_{0}.
\end{equation}
Then, for any $t\in[0,T]$,
\begin{equation}
	\begin{split}
		&\mathbb{E}\left[\left(X^{*,i}_{t}\right)^{2}\right]1_{\alpha(i)=k}=\mathbb{E}\left[\left(\hat{X}^{k}_{t}\right)^{2}\right]1_{\alpha(i)=k},\\
	&\mathbb{E}\left[\left(C^{*,i}(t,X^{*,i}_{t})\right)^{2}\right]1_{\alpha(i)=k}=\mathbb{E}\left[\left(\hat{C}^{k}(t,\hat{X}^{k}_{t})\right)^{2}\right]1_{\alpha(i)=k}.
	\end{split}
\end{equation}
Therefore, for the term $I^{n}_{1}(t)$, it is clear that
\begin{equation}
	\begin{split}
		\sup_{t\in[0,T]}I^{n}_{1}(t):&=\sup_{t\in[0,T]} \mathbb{E}\left[\left(\frac{1}{n}\sum_{i=1}^{n}C^{*,i}(t,X^{*,i}_{t})-\frac{1}{n}\sum_{i=1}^{n}\mathbb{E}\left[C^{*,i}(t,X^{*,i}_{t})\right]\right)^{2}\right]\\&=\sup_{t\in[0,T]}\frac{1}{n^{2}}\sum_{i=1}^{n}\mathbb{E}\left[\left(C^{*,i}(t,X^{*,i}_{t})-\mathbb{E}\left[C^{*,i}(t,X^{*,i}_{t})\right]\right)^{2}\right]\\&\leq\sup_{t\in[0,T]}\frac{1}{n}\sup_{1\leq k\leq K}\mathbb{E}\left[\left(\hat{C}^{k}(t,\hat{X}^{k}_{t})\right)^{2}\right]=O(n^{-1}).
	\end{split}
\end{equation}
It remains to estimate $I^{n}_{2}(t)$. Using Assumption \ref{assump4}, we have
\begin{equation}
	\begin{split}		\sup_{t\in[0,T]}I^{n}_{2}(t)&=\sup_{t\in[0,T]}\mathbb{E}\left[\left(\frac{1}{n}\sum_{i=1}^{n}\mathbb{E}\left[C^{*,i}(t,X^{*,i}_{t})\right]-\mathbb{E}\left[C^{*}(t,X^{*}_{t})\right]\right)^{2}\right]\\&=\sup_{t\in[0,T]}\mathbb{E}\left[\left(\frac{1}{n}\sum_{i=1}^{n}\sum_{k=1}^{K}1_{\alpha(i)=k}\mathbb{E}\left[C^{*,i}(t,X^{*,i}_{t})\right]-\sum_{k=1}^{K}F(\left\{k\right\})\mathbb{E}\left[\hat{C}^{k}(t,\hat{X}^{k}_{t})\right]\right)^{2}\right]\\&=\sup_{t\in[0,T]}\mathbb{E}\left[\left(\sum_{k=1}^{K}\frac{\sum_{i=1}^{n}1_{\alpha(i)=k}}{n}\mathbb{E}\left[\hat{C}^{k}(t,\hat{X}^{k}_{t})\right]-\sum_{k=1}^{K}F(\left\{k\right\})\mathbb{E}\left[\hat{C}^{k}(t,\hat{X}^{k}_{t})\right]\right)^{2}\right]\\&=\sup_{t\in[0,T]}\mathbb{E}\left[\left(\sum_{k=1}^{K}\left(F_{n}(\left\{k\right\})-F(\left\{k\right\})\right)\mathbb{E}\left[\hat{C}^{k}(t,\hat{X}^{k}_{t})\right]\right)^{2}\right]\\&\leq \epsilon_{n}^{2}\sup_{t\in[0,T]}\mathbb{E}\left[\left(\sum_{k=1}^{K}\mathbb{E}\left[\hat{C}^{k}(t,\hat{X}^{k}_{t})\right]\right)^{2}\right]
=O(n^{-1}).
	\end{split}
\end{equation}
Hence,  
\begin{equation}
	\sup_{t\in[0,T]}	\mathbb{E}\left[\left|\overline{Z}^{*,n}_{t}-\overline{Z}_{t}\right|^{2}\right]=O(n^{-1}).
\end{equation}
Now, we move to the estimation about $\left(\overline{X}^{*,n}_{T}, \overline{X}_{T}\right)$.  It follows that 
\begin{equation*}
	\begin{split}
		\mathbb{E}\left[\left|\overline{X}^{*,n}_{T}-\overline{X}_{T}\right|^{2}\right]&\leq C\mathbb{E}\left[\left|\frac{1}{n}\sum_{i=1}^{n}X^{*,i}_{T}-\frac{1}{n}\sum_{i=1}^{n}\mathbb{E}\left[X^{*,i}_{T}\right]\right|^{2}\right]\\&+C\mathbb{E}\left[\left|\frac{1}{n}\sum_{i=1}^{n}\mathbb{E}\left[X^{*,i}_{T}\right]-\mathbb{E}\left[X^{*}_{T}\right]\right|^{2}\right].
	\end{split}
\end{equation*} 
Using the same technique in showing the convergence error of $I^{n}_{1}(t)$ and $I^{n}_{2}(t)$, we have
\begin{align}
	&\mathbb{E}\left[\left|\frac{1}{n}\sum_{i=1}^{n}X^{*,i}_{T}-\frac{1}{n}\sum_{i=1}^{n}\mathbb{E}\left[X^{*,i}_{T}\right]\right|^{2}\right]=O(n^{-1}),\\
	&\mathbb{E}\left[\left|\frac{1}{n}\sum_{i=1}^{n}\mathbb{E}\left[X^{*,i}_{T}\right]-\mathbb{E}\left[X^{*}_{T}\right]\right|^{2}\right]=O(n^{-1}).
\end{align}
Therefore, we obtain  
\begin{equation}
	\mathbb{E}\left[\left|\overline{X}^{*,n}_{T}-\overline{X}_{T}\right|^{2}\right]=O(n^{-1}),
\end{equation}
which completes the proof.
\end{proof}

We are now ready to prove Theorem \ref{6.25}.
\begin{proof}
	For $i\in\left\{1,\cdots,n\right\}$, let $Z^{i}=\left\{Z^{i}_{t}, 0\leq t\leq T\right\}$ and $Z^{*,i}=\left\{Z^{*,i}_{t}, 0\leq t\leq T\right\}$ be the habit formation process of agent $i$ under an arbitrary admissible control $\left(\Pi^{i},C^{i}\right)\in\mathcal{A}^{i,e}_{0}$ and under the control defined via the candidate strategy $\left(\Pi^{*,i},C^{*,i}\right)$ in (\ref{candi}), respectively. For ease of presentation, let us introduce
	\begin{align}
		&\overline{Z}^{*,n,-i}_{t}:=\frac{1}{n}\sum_{j\neq i}Z^{*,j}_{t},\,\ t\in[0,T],\\
		&\overline{X}^{*,n,-i}_{T}:=\frac{1}{n}\sum_{j\neq i}X^{*,j}_{T}.
	\end{align}
 Without loss of generality, we assume $\alpha(i)=1$. Then, we have  
\begin{equation*}
	J_{i}\left((\Pi^{*,i},C^{*,i}),(\mathbf{\Pi}^{*},\mathbf{C}^{*})^{-i}\right):=\mathbb{E}\left[\int_{0}^{T}U_{1}\left(C^{*,i}_{t}-\theta_{1}\overline{Z}^{*,n}_{t}\right)dt+U_{1}\left(X^{*,i}_{T}-\theta_{1}\overline{X}_{T}^{*,n}\right)\right],
\end{equation*}
and 
\begin{align*}
		J_{i}&\left((\Pi^{i},C^{i}),(\mathbf{\Pi}^{*},\mathbf{C}^{*})^{-i}\right):=\\&\mathbb{E}\left[\int_{0}^{T}U_{1}\left(C^{i}_{t}-\theta_{1}(\overline{Z}^{*,n,-i}_{t}+\frac{1}{n}Z^{i}_{t})\right)dt+U_{1}\left(X^{i}_{T}-\theta_{1}(\overline{X}^{*,n,-i}_{T}+\frac{1}{n}X^{i}_{T})\right)\right],
\end{align*}
where $X^{*,i}$ (resp. $X^{i}$) obeys the dynamics (\ref{wealth}) (resp. (\ref{dynaexp}) under an arbitrary admissible control $\left(\Pi^{i},C^{i}\right)$). 

To show (\ref{appnash}) in Definition \ref{Appnash}, let us introduce an auxiliary optimal control problem: for  $\left(\overline{Z}, \overline{X}_{T}\right)$ being the fixed point of (\ref{orgfix}), define 
\begin{equation*}
	\sup_{\left(\Pi,C\right)\in\mathcal{A}^{i,e}_{0}}J_{i}\left(\Pi,C;\overline{Z},\overline{X}_{T}\right):=\sup_{\left(\Pi,C\right)\in\mathcal{A}^{i,e}_{0}}\mathbb{E}\left[\int_{0}^{T}U_{1}\left(C_{t}-\theta_{1}\overline{Z}_{t}\right)dt+U_{1}\left(X^{i}_{T}-\theta_{1}\overline{X}_{T}\right)\right].
\end{equation*}
It is easy to verify that the optimal  strategy of the auxiliary  problem coincides with $\left(\Pi^{*,i},C^{*,i}\right)$ constructed in (\ref{candi}). Then, we have
\begin{equation}\label{appine}
	\begin{split}
	&J_{i}\left((\Pi^{i},C^{i}),(\mathbf{\Pi}^{*},\mathbf{C}^{*})^{-i}\right)-J_{i}(\mathbf{(\Pi^{*},C^{*})})\\&=
		J_{i}\left((\Pi^{i},C^{i}),(\mathbf{\Pi}^{*},\mathbf{C}^{*})^{-i}\right)-\sup_{\left(\Pi,C\right)\in\mathcal{A}^{i,e}_{0}}J_{i}\left(\Pi,C;\overline{Z},\overline{X}_{T}\right)\\&+\sup_{\left(\Pi,C\right)\in\mathcal{A}^{i,e}_{0}}J_{i}\left(\Pi,C;\overline{Z},\overline{X}_{T}\right)
		-J_{i}(\mathbf{(\Pi^{*},C^{*})})\\&\leq\left(J_{i}\left((\Pi^{i},C^{i}),(\mathbf{\Pi}^{*},\mathbf{C}^{*})^{-i}\right)-J_{i}\left(\Pi^{i},C^{i};\overline{Z},\overline{X}_{T}\right)\right)\\&+\sup_{\left(\Pi,C\right)\in\mathcal{A}^{i,e}_{0}}J_{i}\left(\Pi,C;\overline{Z},\overline{X}_{T}\right)
		-J_{i}(\mathbf{(\Pi^{*},C^{*})}).
	\end{split}
\end{equation}
We first evaluate the first term of RHS of (\ref{appine}):
\begin{equation}\label{app1}
	\begin{split}
		J_{i}&\left((\Pi^{i},C^{i}),(\mathbf{\Pi}^{*},\mathbf{C}^{*})^{-i}\right)-J_{i}\left(\Pi^{i},C^{i};\overline{Z},\overline{X}_{T}\right)\\&=\mathbb{E}\left[\int_{0}^{T}\left(U_{1}\left(C^{i}_{t}-\theta_{1}(\overline{Z}^{*,n,-i}_{t}+\frac{1}{n}Z^{i}_{t})\right)-U_{1}\left(C^{i}_{t}-\theta_{1}\overline{Z}_{t}\right)\right)dt\right]\\&+\mathbb{E}\left[U_{1}\left(X^{i}_{T}-\theta_{1}(\overline{X}^{*,n,-i}_{T}+\frac{1}{n}X^{i}_{T})\right)-U_{1}\left(X^{i}_{T}-\theta_{1}\overline{X}_{T}\right)\right]\\&:=I^{(1)}_{i}+I^{(2)}_{i}.
	\end{split}
\end{equation}
Then,
\begin{equation*}
	\begin{split}
		I^{\left(1\right)}_{i}&=\mathbb{E}\left[\int_{0}^{T}\left(U_{1}\left(C^{i}_{t}-\theta_{1}(\overline{Z}^{*,n,-i}_{t}+\frac{1}{n}Z^{i}_{t})\right)-U_{1}\left(C^{i}_{t}-\theta_{1}\overline{Z}^{*,n}_{t}\right)\right)dt\right]\\&+\mathbb{E}\left[\int_{0}^{T}\left(U_{1}\left(C^{i}_{t}-\theta_{1}\overline{Z}^{*,n}_{t}\right)-U_{1}\left(C^{i}_{t}-\theta_{1}\overline{Z}_{t}\right)\right)dt\right]\\&:=I^{(3)}_{i}+I^{(4)}_{i},\\
		I^{(2)}_{i}&=\mathbb{E}\left[U_{1}\left(X^{i}_{T}-\theta_{1}(\overline{X}^{*,n,-i}_{T}+\frac{1}{n}X^{i}_{T})\right)-U_{1}\left(X^{i}_{T}-\theta_{1}\overline{X}^{*,n}_{T}\right)\right]\\&+\mathbb{E}\left[U_{1}\left(X^{i}_{T}-\theta_{1}\overline{X}^{*,n}_{T}\right)-U_{1}\left(X^{i}_{T}-\theta_{1}\overline{X}_{T}\right)\right]\\&:=I^{(5)}_{i}+I^{\left(6\right)}_{i}.
	\end{split}
\end{equation*}
For the term $I^{(3)}_{i}$, it is clear that 
\begin{equation}
	I^{(3)}_{i}=\mathbb{E}\left[\int_{0}^{T}\left(\exp\left\{\frac{\theta_{1}}{\beta_{1}}\overline{Z}^{*,n}_{t}-\frac{1}{\beta}C^{i}_{t}\right\}-\exp\left\{\frac{\theta_{1}}{\beta_{1}}(\overline{Z}^{*,n,-i}_{t}+\frac{1}{n}Z^{i}_{t})-\frac{1}{\beta}C^{i}_{t}\right\}\right)dt\right].
\end{equation}
Using the convexity of exponential function and  {H\"{o}lder} inequality, we obtain  
\begin{equation}
	\begin{split}
		I^{(3)}_{i}&\leq\frac{1}{n}\mathbb{E}\left[\int_{0}^{T}\frac{\theta_{1}}{\beta_{1}}\left|U_{1}\left(C^{i}_{t}-\theta_{1}\overline{Z}^{*,n}_{t}\right)\right|\left|Z^{*,i}_{t}-Z^{i}_{t}\right|dt\right]\\&\leq\frac{1}{n}\mathbb{E}\left[\int_{0}^{T}\left(\frac{\theta_{1}}{\beta_{1}}\right)^{2}\left|U_{1}\left(C^{i}_{t}-\theta_{1}\overline{Z}^{*,n}_{t}\right)\right|^{2}dt\right]^{\frac{1}{2}}\mathbb{E}\left[\int_{0}^{T}\left|Z^{*,i}_{t}-Z^{i}_{t}\right|^{2}dt\right]^{\frac{1}{2}}\\&\leq\frac{C}{n}\mathbb{E}\left[\int_{0}^{T}\left|U_{1}\left(C^{i}_{t}-\theta_{1}\overline{Z}^{*,n}_{t}\right)\right|^{2}dt\right]^{\frac{1}{2}}\mathbb{E}\left[\sup_{t\in[0,T]}\left|Z^{*,i}_{t}\right|^{2}+\sup_{t\in[0,T]}\left|Z^{i}_{t}\right|^{2}\right]^{\frac{1}{2}}\\&=O(n^{-1}),
	\end{split}
\end{equation}
where we have just used  Lemma \ref{lemma1}-(i) and Assumption \ref{assump5}.

Similarly,
\begin{equation*}
	\begin{split}
		I^{(4)}_{i}&\leq\mathbb{E}\left[\int_{0}^{T}\frac{\theta_{1}}{\beta_{1}}\left|U_{1}\left(C^{i}_{t}-\theta_{1}\overline{Z}_{t}\right)\right|\left|\overline{Z}^{*,n}_{t}-\overline{Z}_{t}\right|dt\right]\\&\leq C\mathbb{E}\left[\int_{0}^{T}\left|U_{1}\left(C^{i}_{t}-\theta_{1}\overline{Z}_{t}\right)\right|^{2}dt\right]^{\frac{1}{2}}\mathbb{E}\left[\sup_{t\in[0,T]}\left|Z^{*,i}_{t}-Z^{i}_{t}\right|^{2}\right]^{\frac{1}{2}}\\&=O(n^{-\frac{1}{2}}),
	\end{split}
\end{equation*}
where we have just used  Lemma \ref{lemma1}-(ii) and Assumption \ref{assump5}.

Hence,  
\begin{equation}
	I^{(1)}_{i}=I^{(3)}_{i}+I^{(4)}_{i}=O(n^{-\frac{1}{2}}).
\end{equation}
Recall that 
\begin{equation}
	\begin{split}
		I^{(5)}_{i}=\mathbb{E}\left[\exp\left\{\frac{\theta_{1}}{\beta_{1}}\overline{X}^{*,n}_{T}-\frac{1}{\beta_{1}}X^{i}_{T}\right\}-\exp\left\{\frac{\theta_{1}}{\beta_{1}}(\overline{X}^{*,n,-i}_{T}+\frac{1}{n}X^{i}_{T})-\frac{1}{\beta_{1}}X^{i}_{T}\right\}\right].
	\end{split}
\end{equation}
Similarly,
\begin{equation}
	\begin{split}
		I^{(5)}_{i}&\leq\frac{1}{n}\mathbb{E}\left[\frac{\theta_{1}}{\beta_{1}}\left|U_{1}\left(X^{i}_{T}-\theta_{1}\overline{X}^{*,n}_{T}\right)\right|\left|X^{*,i}_{T}-X^{i}_{T}\right|\right]\\&\leq\frac{C}{n}\mathbb{E}\left[\left|U_{1}\left(X^{i}_{T}-\theta_{1}\overline{X}^{*,n}_{T}\right)\right|^{2}\right]^{\frac{1}{2}}\mathbb{E}\left[\sup_{t\in[0,T]}\left|X^{*,i}_{t}\right|^{2}+\sup_{t\in[0,T]}\left|X^{i}_{t}\right|^{2}\right]^{\frac{1}{2}}\\&=O(n^{-1}),
	\end{split}
\end{equation}
where we have just used the estimate (\ref{estimate}) and Assumption \ref{assump5}.

For the term $I^{(6)}_{i}$,  we have 
\begin{equation*}
	\begin{split}
		I^{(6)}_{i}&\leq\mathbb{E}\left[\frac{\theta_{1}}{\beta_{1}}\left|U_{1}\left(X^{i}_{T}-\theta_{1}\overline{X}_{T}\right)\right|\left|\overline{X}^{*,n}_{T}-\overline{X}_{T}\right|\right]\\&\leq C\mathbb{E}\left[\left|U_{1}\left(X^{i}_{T}-\theta_{1}\overline{X}_{T}\right)\right|^{2}\right]^{\frac{1}{2}}\mathbb{E}\left[\sup_{t\in[0,T]}\left|\overline{X}^{*,n}_{t}-\overline{X}_{t}\right|^{2}\right]^{\frac{1}{2}}\\&=O(n^{-\frac{1}{2}}),
	\end{split}
\end{equation*}
where we have just used  Lemma \ref{lemma1}-(ii) and Assumption \ref{assump5}.

Therefore, it is clear that 
\begin{equation}
	I^{(2)}_{i}=I^{(5)}_{i}+I^{(6)}_{i}=O(n^{-\frac{1}{2}}).
\end{equation}
By (\ref{app1}), we have
\begin{equation}\label{Est1}
	\begin{split}
		J_{i}&\left((\Pi^{i},C^{i}),(\mathbf{\Pi}^{*},\mathbf{C}^{*})^{-i}\right)-J_{i}\left(\Pi^{i},C^{i};\overline{Z},\overline{X}_{T}\right)=I^{(1)}_{i}+I^{\left(2\right)}_{i}=O(n^{-\frac{1}{2}}).
	\end{split}
\end{equation}
For the second term of RHS of (\ref{appine}), we have 
\begin{equation}
\sup_{\left(\Pi,C\right)\in\mathcal{A}^{i,e}_{0}}J_{i}\left(\Pi,C;\overline{Z},\overline{X}_{T}\right)
-J_{i}\left(\mathbf{\Pi}^{*},\mathbf{C}^{*}\right)=J_{i}\left(\Pi^{*,i},C^{*,i};\overline{Z},\overline{X}_{T}\right)
-J_{i}\left(\mathbf{\Pi}^{*},\mathbf{C}^{*}\right).
\end{equation}
We  therefore follow the similar argument in showing the convergence error of $I^{(4)}_{i}$ and $I^{(6)}_{i}$ with fixed control $\left(\Pi^{*,i},C^{*,i}\right)$ to get 
\begin{equation}\label{Est2}
	J_{i}\left(\Pi^{*,i},C^{*,i};\overline{Z},\overline{X}_{T}\right)
	-J_{i}\left(\mathbf{\Pi}^{*},\mathbf{C}^{*}\right)=O(n^{-\frac{1}{2}}).
\end{equation}
Therefore, the estimates (\ref{appine}), (\ref{Est1}) and (\ref{Est2}) jointly yield (\ref{appnash}) with $\epsilon_{n}=O(n^{-\frac{1}{2}})$.
\end{proof}
\subsection{Power Utility}
For $i\in\left\{1,\cdots,n\right\}$,  the objective function of agent $i$ with power utility is
\begin{equation}
	J_{i}\left((\pi^{i},c^{i}),(\mathbf{\pi},\mathbf{c})^{-i}\right):=\mathbb{E}\left[\int_{0}^{T}U_{\alpha(i)}\left(\frac{c^{i}_{t}X^{i}_{t}}{\left(\overline{Z}^{n}_{t}\right)^{\theta_{\alpha(i)}}}\right)dt+U_{\alpha(i)}\left(\frac{X^{i}_{T}}{\left(\overline{X}_{T}^{n}\right)^{\theta_{\alpha(i)}}}\right)\right],
\end{equation}
where the vector $(\mathbf{\pi},\mathbf{c})^{-i}$ is defined by 
\begin{equation}
	(\mathbf{\pi},\mathbf{c})^{-i}:=\left((\pi^{1},c^{1}),\cdots,(\pi^{i-1},c^{i-1}),(\pi^{i+1},c^{i+1}),\cdots,(\pi^{n},c^{n})\right).
\end{equation}
The definition of an approximate Nash equilibrium  can be obtained in a similar fashion.
\begin{definition}{(Approximate Nash equilibrium)}\label{Appnashpower}
	Let $\mathcal{A}^{p}:=\prod_{i=1}^{n}\mathcal{A}^{i,p}_{0}$. An $n$-tuple of admissible controls $\left(\mathbf{\pi}^{*},\mathbf{c}^{*}\right):=\left((\pi^{*,i},c^{*,i})\right)_{i=1}^{n}\in\mathcal{A}^{p}$ is called an $\epsilon$-Nash equilibrium to the n-agent game, if for any $(\pi^{i},c^{i})\in\mathcal{A}^{i,p}_{0}$, it holds that  
	\begin{equation}\label{appnashpower}
		J_{i}\left((\pi^{i},c^{i}),(\mathbf{\pi}^{*},\mathbf{c}^{*})^{-i}\right)\leq J_{i}\left(\mathbf{\pi}^{*},\mathbf{c}^{*}\right)+\epsilon,\quad\forall i=1,\cdots,n.
	\end{equation}
\end{definition}
For technical convenience and ease of presentation, we make the following assumption on $\mathcal{A}^{p}$ in this section:
\begin{assumption}\label{assump6}
	Assume that the uniform boundedness condition holds that for all $i\in\left\{1,\cdots,n\right\}$,  $$\sup_{t\in[0,T]}|\pi_{t}^{i}|\vee|c^{i}_{t}|<M\quad a.s. $$ for some constant $M$ independent of $n$.
\end{assumption}
\begin{remark}
	Recall that $\pi$ and $c$ represent the proportion of wealth invested in risky assets and spent on consumption, respectively. Assumption \ref{assump6} indicates that no one is allowed to borrow too much wealth for investment and consumption.
\end{remark}
For  $i\in\left\{1,\cdots,n\right\}$, let us construct the  candidate investment and consumption pair $\left(\pi^{*,i},c^{*,i}\right)$ in the power case: for $t\in[0,T]$,
\begin{equation}\label{candipower}
	\begin{cases}
		\pi^{*,i}_{t}:= \frac{\mu_{\alpha(i)}}{(1-p_{\alpha(i)})\sigma_{\alpha(i)}^{2}},\\
		c^{*,i}_{t}:=\left(\overline{Z}_{t}\right)^{\frac{\theta_{\alpha(i)} p_{\alpha(i)}}{p_{\alpha(i)}-1}}g_{\alpha(i)}(t)^{\frac{1}{p_{\alpha(i)}-1}}.
	\end{cases}
\end{equation}
The function $g_{\alpha(i)}(\cdot)$ is given by, for $t\in[0,T]$,
\begin{equation*}
	g_{\alpha(i)}(t) = \left(e^{a_{\alpha(i)}\left(T-t\right)}\left(\frac{1}{\left(\overline{X}_{T}\right)^{p_{\alpha(i)}\theta_{\alpha(i)}}}\right)^{\frac{1}{1-p_{\alpha(i)}}}+e^{-a_{\alpha(i)}t}\int_{t}^{T}e^{a_{\alpha(i)}s}\left(\overline{Z}_{s}\right)^{\frac{\theta_{\alpha(i)} p_{\alpha(i)}}{p_{\alpha(i)}-1}}ds\right)^{1-p_{\alpha(i)}},  
\end{equation*}
where $	a_{\alpha(i)}$ is given by
\begin{equation*}
	a_{\alpha(i)}:=\frac{1}{2}\frac{\mu_{\alpha(i)}^{2}}{\sigma_{\alpha(i)}^{2}}\frac{p_{\alpha(i)}}{\left(1-p_{\alpha(i)}\right)^{2}}
\end{equation*}
and $\left(\overline{Z},\overline{X}_{T}\right)$ is the fixed point of (\ref{orgfixpower}). Denote by $X^{*,i}$ the wealth process of agent $i$ under the candidate strategy $\left(\pi^{*,i},c^{*,i}\right)$:
\begin{equation}\label{canwealthpower}
	\frac{dX^{*,i}_{t}}{X^{*,i}_{t}}=\pi_{t}^{*,i}\mu_{\alpha(i)}dt+\pi^{*,i}_{t}\sigma_{\alpha(i)}dW^{i}_{t}-c^{*,i}_{t}dt, \quad X^{*,i}_{0}=x_{0}.
\end{equation}
Using the estimation in Appendix \ref{5.26.4}, we get for any $q>0$ and $i\in\left\{1,\cdots,n\right\}$,
\begin{equation}\label{boundedness}
	\begin{split}
		\sup_{t\in[0,T]}\left(\pi^{*,i}_{t}\right)^{q}&=\sup_{1\leq k\leq K}\left(\frac{\mu_{k}}{(1-p_{k})\sigma_{k}^{2}}\right)^{q}:=M_{1}\\
		\sup_{t\in[0,T]}\left(c^{*,i}_{t}\right)^{q}&=\sup_{1\leq k\leq K}\sup_{t\in[0,T]}\left(\left(\overline{Z}_{t}\right)^{\frac{\theta_{k} p_{k}}{p_{k}-1}}g_{k}(t)^{\frac{1}{p_{k}-1}}\right)^{q}\\&\leq\sup_{1\leq k\leq K}\sup_{t\in[0,T]}\left(\frac{\left(\overline{Z}_{t}\right)^{\frac{\theta_{k} p_{k}}{p_{k}-1}}}{e^{a_{k}(T-t)}\left(\overline{X}_{T}\right)^{{\frac{\theta_{k} p_{k}}{p_{k}-1}}}}\right)^{q}\\&\leq\sup_{1\leq k\leq K}\left(\left(e^{-\delta T}z_{0}\right)^{{\frac{\theta_{k} p_{k}}{p_{k}-1}}}C_{2}^{{\frac{\theta_{k} p_{k}}{1-p_{k}}}}\right)^{q}:=M_{2}
	\end{split}
\end{equation}
Hence $\left(\mathbf{\pi}^{*},\mathbf{c}^{*}\right)\in\mathcal{A}^{p}$ satisfies  Assumption \ref{assump6}.

 The $i$-th agent's habit formation process is given by 
\begin{equation}\label{zpower}
	Z^{*,i}_{t} = e^{-\delta t}\left(z_{0}+\int_{0}^{t}\delta e^{\delta s}c_{s}^{*,i}X^{*,i}_{s}ds\right),\quad t\in[0,T].
\end{equation}
Denote by
\begin{equation}\label{zxpower}
	\overline{Z}_{t}^{*,n}:=\frac{1}{n}\sum_{i=1}^{n}Z^{*,i}_{t},\quad\overline{X}^{*,n}_{T}:=\left(\prod_{i=1}^{n}X^{*,i}_{T}\right)^{\frac{1}{n}}=\exp\left\{\frac{1}{n}\sum_{i=1}^{n}\log(X^{*,i}_{T})\right\}.
\end{equation}
It follows that \begin{equation}\label{Zpower}
	\overline{Z}^{*,n}_{t}:=e^{-\delta t}z_{0}+\frac{1}{n}\sum_{i=1}^{n}\int_{0}^{t}\delta e^{\delta(s-t)}c_{s}^{*,i}X^{*,i}_{s}ds,\,\ t\in[0,T].
\end{equation}

Next, we introduce the main result of this section on the existence of an approximate Nash equilibrium.
\begin{theorem}\label{6.25power}
	The control pair $\left(\mathbf{\pi}^{*},\mathbf{c}^{*}\right)=\left\{\left((\pi^{*,1},c_{t}^{*,1}),\dots,(\pi^{*,n},c_{t}^{*,n})\right), 0\leq t\leq 
	T\right\}$ in (\ref{candipower}) is an $\epsilon_{n}$-Nash equilibrium to the n-agent game with the explicit order $\epsilon_{n}=O(n^{-\frac{1}{2}})$.
\end{theorem}

To prove Theorem \ref{6.25power}, we need the following auxiliary results.
\begin{lemma}\label{lemma2power} For any $n\geq1$, it holds that
	
	(i) For $i\in\left\{1,\cdots,n\right\}$, let the process $X^{i}=\left\{X^{i}_{t}, 0\leq t\leq T\right\}$ be defined by (\ref{dynapower}) under an admissible control $\left(\pi^{i},c^{i}\right)\in\mathcal{A}^{i,p}_{0}$.  Then, under the condition of Assumption \ref{assump6}, for any $q\in\mathbb{R}$, there exists a constant $C_{q}$ independent of $i$ and $n$ such that
	\begin{equation}\label{Xest}
		\sup_{t\in[0,T]}\mathbb{E}\left[\left(X^{i}_{t}\right)^{q}\right]\leq C_{q}.
	\end{equation}
	
	(ii) For $i\in\left\{1,\cdots,n\right\}$, let $Z^{*,i}=\left\{Z^{*,i}_{t}, 0\leq t\leq T\right\}$ be defined by (\ref{zpower}). Then, for any $q\in(1,\infty)\cup(-\infty,0)$, there exists a constant $C_{q}$ independent of $i$ and $n$ such that
	\begin{equation}
		\sup_{t\in[0,T]}\mathbb{E}\left[\left(Z^{*,i}_{t}\right)^{q}\right]\leq C_{q}.
	\end{equation}
	
	(iii) Let $\left(\overline{Z},\overline{X}_{T}\right)$ be the fixed point to (\ref{orgfixpower}), and let $\left(\overline{Z}^{*,n},\overline{X}^{*,n}_{T}\right)$ be defined by (\ref{zxpower}). Then, for any  $\gamma\in\mathbb{R}$, we have
	\begin{equation}
		\sup_{t\in[0,T]}\mathbb{E}\left[\left|\overline{Z}^{*,n}_{t}-\overline{Z}_{t}\right|^{2}\right]=O(n^{-1}),\quad 	\mathbb{E}\left[\left|\left(\overline{X}^{*,n}_{T}\right)^{\gamma}-\left(\overline{X}_{T}\right)^{\gamma}\right|^{2}\right]=O(n^{-1}).
	\end{equation}
\end{lemma}
\begin{proof}
	\textbf{(i)}
	Let $\left(\pi^{i},c^{i}\right)$ be an admissible control for agent $i$, then it holds that
	\begin{equation}
		X^{i}_{t}=x_{0}\exp\left\{\int_{0}^{t}\left(\pi^{i}_{s}\mu_{\alpha(i)}-\frac{\sigma_{\alpha(i)}^{2}}{2}\left(\pi_{s}^{i}\right)^{2}-c^{i}_{s}\right)ds+\int_{0}^{t}\pi^{i}_{s}\sigma_{\alpha(i)}dW^{i}_{s}\right\}.
	\end{equation} 
	Hence, 
	\begin{equation}\label{qwealth}
		\left(X^{i}_{t}\right)^{q}=(x_{0})^{q}\exp\left\{\int_{0}^{t}q\left(\pi^{i}_{s}\mu_{\alpha(i)}-\frac{\sigma_{\alpha(i)}^{2}}{2}\left(\pi_{s}^{i}\right)^{2}-c^{i}_{s}\right)ds+\int_{0}^{t}q\pi^{i}_{s}\sigma_{\alpha(i)}dW^{i}_{s}\right\}.
	\end{equation}
	It follows from (\ref{qwealth}) and Assumption \ref{assump6} that  for any $q\in\mathbb{R}$, 
	\begin{equation}\label{keyx}
		\begin{split}
			\mathbb{E}\left[\left(X^{i}_{t}\right)^{q}\right]&\leq(x_{0})^{q}\exp\left\{|q|MT\left(\mu_{\alpha(i)}+1+\frac{\sigma^{2}_{\alpha(i)}}{2}M\right)\right\}\mathbb{E}\left[\exp\left\{\int_{0}^{t}q\pi^{i}_{s}\sigma_{\alpha(i)}dW^{i}_{s}\right\}\right]\\&\leq (x_{0})^{q}\exp\left\{|q|MT\left(\mu_{\alpha(i)}+1+\frac{\sigma^{2}_{\alpha(i)}}{2}M\left(|q|+1\right)\right)\right\}
			\\&\leq\sup_{1\leq k\leq K}(x_{0})^{q}\exp\left\{|q|MT\left(\mu_{k}+1+\frac{\sigma^{2}_{k}}{2}M\left(|q|+1\right)\right)\right\}:=C_{q}.
		\end{split}
	\end{equation}

	\textbf{(ii)} For the case $q>1$, by (\ref{boundedness}), (\ref{zpower}), (\ref{Xest}) and  {H\"{o}lder} inequality with exponents $(q,q_{0})\in(1,+\infty)^{2}$ satisfying $\frac{1}{q}+\frac{1}{q_{0}}=1$, it follows that 
	\begin{equation}
		\begin{split}
			\sup_{t\in[0,T]}\mathbb{E}\left[\left(Z^{*,i}_{t}\right)^{q}\right]&\leq C_{q}\left(z_{0}\right)^{q}+C_{q}\left(\delta\right)^{q}\mathbb{E}\left[\left(\int_{0}^{T}c^{*,i}_{s}X^{*,i}_{s}ds\right)^{q}\right]\\&\leq C_{q}\left(z_{0}\right)^{q}+C_{q}\left(\delta\right)^{q}\mathbb{E}\left[\left(\int_{0}^{T}\left(c^{*,i}_{s}\right)^{q_{0}}ds\right)^{\frac{q}{q_{0}}}\left(\int_{0}^{T}\left(X^{*,i}_{s}\right)^{q}ds\right)\right]\\&=C_{q}\left(z_{0}\right)^{q}+C_{q}\left(\delta\right)^{q}\left(\int_{0}^{T}\left(c^{*,i}_{s}\right)^{q_{0}}ds\right)^{\frac{q}{q_{0}}}\mathbb{E}\left[\left(\int_{0}^{T}\left(X^{*,i}_{s}\right)^{q}ds\right)\right]\\&\leq C_{q}\left(z_{0}\right)^{q}+C_{q}\left(\delta\right)^{q}T\left(\int_{0}^{T}\left(c^{*,i}_{s}\right)^{q_{0}}ds\right)^{\frac{q}{q_{0}}}\\&\leq C_{q}\left(z_{0}\right)^{q}+C_{q}\left(\delta\right)^{q}T^{1+\frac{q}{q_{0}}}\left(\sup_{t\in[0,T]}c^{*,i}_{s}\right)^{q},
		\end{split}
	\end{equation}
	where $C_{q}$ is a constant that may vary from one line to another.
	
	For $q<0$, it follows from (\ref{zpower}) that, for all $t\in[0,T]$, 
	\begin{equation}
		Z^{*,i}_{t} = e^{-\delta t}\left(z_{0}+\int_{0}^{t}\delta e^{\delta s}c_{s}^{*,i}X^{*,i}_{s}ds\right)\geq e^{-\delta t}z_{0}
	\end{equation}
	yielding
	\begin{equation}
		\sup_{t\in[0,T]}\mathbb{E}\left[\left(Z^{*,i}_{t}\right)^{q}\right]\leq\left(e^{-\delta T}z_{0}\right)^{q}:=C_{q}.
	\end{equation}
	Thus, we obtain the estimation in \textbf{(ii)}.
	
	\textbf{(iii)}  The estimation about $\left(\overline{Z}^{*,n},\overline{Z}\right)$ can be derived by a similar argument as the proof of Lemma \ref{lemma1}-\textbf{(ii)} and we only consider the estimation about $\left(\overline{X}^{*,n}_{T}, \overline{X}_{T}\right)$. 
	
	Recall the inequality $\left|e^{a}-e^{b}\right|\leq\max\left(e^{a}, e^{b}\right)\left|a-b\right|$. Then, it follows from (\ref{orgfixpower}), (\ref{zxpower}) and  {H\"{o}lder} inequality  that 
	\begin{equation*}
		\begin{split}
			&\mathbb{E}\left[\left|\left(\overline{X}^{*,n}_{T}\right)^{\gamma}-\left(\overline{X}_{T}\right)^{\gamma}\right|^{2}\right]\\&\leq|\gamma|^{2}\mathbb{E}\left[\max\left[\left(\overline{X}^{*,n}_{T}\right)^{\gamma}, \left(\overline{X}_{T}\right)^{\gamma}\right]^{2}\left|\frac{1}{n}\sum_{i=1}^{n}\log\left(X^{*,i}_{T}\right)-\mathbb{E}\left[\log\left(X^{*}_{T}\right)\right]\right|^{2}\right]\\&\leq|\gamma|^{2}\mathbb{E}\left[\max\left[\left(\overline{X}^{*,n}_{T}\right)^{\gamma}, \left(\overline{X}_{T}\right)^{\gamma}\right]^{4}\right]^{\frac{1}{2}}\mathbb{E}\left[\left|\frac{1}{n}\sum_{i=1}^{n}\log\left(X^{i,*}_{T}\right)-\mathbb{E}\left[\log\left(X^{*}_{T}\right)\right]\right|^{4}\right]^{\frac{1}{2}}. 
		\end{split}
	\end{equation*}
	Note that 
	\begin{equation}
		\mathbb{E}\left[\max\left[\left(\overline{X}^{*,n}_{T}\right)^{\gamma}, \left(\overline{X}_{T}\right)^{\gamma}\right]^{4}\right]\leq\mathbb{E}\left[\left(\overline{X}^{*,n}_{T}\right)^{4\gamma }\right]+\mathbb{E}\left[\left(\overline{X}_{T}\right)^{4\gamma}\right].
	\end{equation}
	By the independence of $X^{*,i}$, we have 
	\begin{equation}
		\begin{split}
			\mathbb{E}\left[\left(\overline{X}^{*,n}_{T}\right)^{4\gamma }\right]=\mathbb{E}\left[\prod_{i=1}^{n}\left(X^{*,i}_{T}\right)^{\frac{4\gamma }{n}}\right]=\prod_{i=1}^{n}\mathbb{E}\left[\left(X^{*,i}_{T}\right)^{\frac{4\gamma }{n}}\right].
		\end{split}
	\end{equation}
	Then, by Lemma \ref{lemma2power}-(i), we have that there exists a constant $C$ independent of $n$ such that
	\begin{equation}\label{6.27}
		\begin{split}
			\prod_{i=1}^{n}&\mathbb{E}\left[\left(X^{*,i}_{T}\right)^{\frac{4\gamma }{n}}\right]\leq C.
		\end{split}
	\end{equation}
	Proposition \ref{5.26.2} yields 
	\begin{equation}
		\mathbb{E}\left[\left(\overline{X}_{T}\right)^{4\gamma}\right]\leq (C_{2})^{4\gamma},
	\end{equation}
	as such, combining with (\ref{6.27}), there exists a constant $C$ independent of  $n$ such that
	\begin{equation}
		\mathbb{E}\left[\max\left[\left(\overline{X}^{*,n}_{T}\right)^{\gamma}, \left(\overline{X}_{T}\right)^{\gamma}\right]^{4}\right]^{\frac{1}{2}}\leq C.
	\end{equation}
	Therefore, it is enough to estimate the following term:
	\begin{equation}
		\mathbb{E}\left[\left|\frac{1}{n}\sum_{i=1}^{n}\log\left(X^{*,i}_{T}\right)-\mathbb{E}\left[\log\left(X^{*}_{T}\right)\right]\right|^{4}\right]^{\frac{1}{2}}.
	\end{equation}
     It is easily seen that
	\begin{equation*}
		\begin{split}
			\mathbb{E}\left[\left|\frac{1}{n}\sum_{i=1}^{n}\log\left(X^{*,i}_{T}\right)-\mathbb{E}\left[\log\left(X^{*}_{T}\right)\right]\right|^{4}\right]&\leq C\mathbb{E}\left[\left|\frac{1}{n}\sum_{i=1}^{n}\log\left(X^{*,i}_{T}\right)-\frac{1}{n}\sum_{i=1}^{n}\mathbb{E}\left[\log\left(X^{*,i}_{T}\right)\right]\right|^{4}\right]\\&+C\mathbb{E}\left[\left|\frac{1}{n}\sum_{i=1}^{n}\mathbb{E}\left[\log\left(X^{*,i}_{T}\right)\right]-\mathbb{E}\left[\log\left(X^{*}_{T}\right)\right]\right|^{4}\right].
		\end{split}
	\end{equation*} 
By a similar argument as the proof of  Lemma \ref{lemma1}-(ii), we have
	\begin{align}
		&\mathbb{E}\left[\left|\frac{1}{n}\sum_{i=1}^{n}\mathbb{E}\left[\log\left(X^{*,i}_{T}\right)\right]-\mathbb{E}\left[\log\left(X^{*}_{T}\right)\right]\right|^{4}\right]=O(n^{-2}),\\
		&\mathbb{E}\left[\left|\frac{1}{n}\sum_{i=1}^{n}\log\left(X^{i,*}_{T}\right)-\frac{1}{n}\sum_{i=1}^{n}\mathbb{E}\left[\log\left(X^{*,i}_{T}\right)\right]\right|^{4}\right]=O(n^{-2}).
	\end{align}
	Hence,
	\begin{equation}
	\mathbb{E}\left[\left|\frac{1}{n}\sum_{i=1}^{n}\log\left(X^{i,*}_{T}\right)-\mathbb{E}\left[\log\left(X^{*}_{T}\right)\right]\right|^{4}\right]^{\frac{1}{2}}=O(n^{-1}).
	\end{equation}
	Therefore, we obtain  
	\begin{equation}
		\mathbb{E}\left[\left|\left(\overline{X}^{*,n}_{T}\right)^{\gamma}-\left(\overline{X}_{T}\right)^{\gamma}\right|^{2}\right]=O(n^{-1}),
	\end{equation}
	which completes the proof.
\end{proof}
We are now ready to prove Theorem \ref{6.25power}.
\begin{proof}
	For $i\in\left\{1,\cdots,n\right\}$, let $Z^{i}$ and $Z^{*,i}$ be the habit formation process of agent $i$ under an arbitrary admissible control $\left(\pi^{i},c^{i}\right)\in\mathcal{A}^{i,p}_{0}$ and under the control defined via the candidate strategy $\left(\pi^{*,i},c^{*,i}\right)$ in (\ref{candipower}), respectively. For ease of presentation, let us introduce
	\begin{align}
		&\overline{Z}^{*,n,-i}:=\frac{1}{n}\sum_{j\neq i}Z^{*,j}_{t},\,\ t\in[0,T],\\
		&\overline{X}^{*,n,-i}_{T}:=\left(\prod_{j\neq i}X^{*,j}_{T}\right)^{\frac{1}{n}}.
	\end{align}
	Without loss of generality, we assume $\alpha(i)=1$. Then,  
	\begin{equation*}
		J_{i}\left((\pi^{*,i},c^{*,i}),(\mathbf{\pi}^{*},\mathbf{c}^{*})^{-i}\right):=\mathbb{E}\left[\int_{0}^{T}\frac{1}{p_{1}}\left(\frac{c^{*,i}_{t}X^{*,i}_{t}}{\left(\overline{Z}^{*,n}_{t}\right)^{\theta_{1}}}\right)^{p_{1}}dt+\frac{1}{p_{1}}\left(\frac{X^{*,i}_{T}}{\left(\overline{X}^{*,n}_{T}\right)^{\theta_{1}}}\right)^{p_{1}}\right]
	\end{equation*}
	and 
	\begin{equation*}
		\begin{split}
			J_{i}&\left((\pi^{i},c^{i}),(\mathbf{\pi}^{*},\mathbf{c}^{*})^{-i}\right):=\\&\mathbb{E}\left[\int_{0}^{T}\frac{1}{p_{1}}\left(\frac{c^{i}_{t}X^{i}_{t}}{\left(\overline{Z}^{*,n,-i}_{t}+\frac{1}{n}Z^{i}_{t}\right)^{\theta_{1}}}\right)^{p_{1}}dt+\frac{1}{p_{1}}\left(\frac{X^{i}_{T}}{\left(\overline{X}^{*,n,-i}_{T}\cdot\left(X^{i}_{T}\right)^{\frac{1}{n}}\right)^{\theta_{1}}}\right)^{p_{1}}\right],
		\end{split}
	\end{equation*}
	where $X^{*,i}$ (resp. $X^{i}$) obeys the dynamics (\ref{canwealthpower}) (resp. (\ref{dynapower}) under an arbitrary admissible control $\left(\pi^{i},c^{i}\right)$). 
	
	To show (\ref{appnashpower}) in Definition \ref{Appnashpower}, let us introduce an auxiliary optimal control problem: for $\left(\overline{Z}, \overline{X}_{T}\right)$ being a fixed point of (\ref{orgfixpower}), define
	\begin{equation*}
		\sup_{\left(\pi,c\right)\in\mathcal{A}^{i,p}_{0}}J_{i}\left(\pi,c;\overline{Z},\overline{X}_{T}\right):=\sup_{\left(\pi,c\right)\in\mathcal{A}^{i,p}_{0}}\mathbb{E}\left[\int_{0}^{T}\frac{1}{p_{1}}\left(\frac{c_{t}X^{i}_{t}}{\left(\overline{Z}_{t}\right)^{\theta_{1}}}\right)^{p_{1}}dt+\frac{1}{p_{1}}\left(\frac{X^{i}_{T}}{\left(\overline{X}_{T}\right)^{\theta_{1}}}\right)^{p_{1}}\right].
	\end{equation*}
	It is easy to verify that the optimal strategy of the auxiliary  problem coincides with $\left(\pi^{*,i},c^{*,i}\right)$ constructed in (\ref{candipower}).  Then, 
	\begin{equation}\label{appinepower}
		\begin{split}
			&J_{i}\left((\pi^{i},c^{i}),(\mathbf{\pi}^{*},\mathbf{c}^{*})^{-i}\right)-J_{i}(\mathbf{(\pi^{*},c^{*})})\\&=
			J_{i}\left((\pi^{i},c^{i}),(\mathbf{\pi}^{*},\mathbf{c}^{*})^{-i}\right)-\sup_{\left(\pi,c\right)\in\mathcal{A}^{i,p}_{0}}J_{i}\left(\pi,c;\overline{Z},\overline{X}_{T}\right)\\&+\sup_{\left(\pi,c\right)\in\mathcal{A}^{i,p}_{0}}J_{i}\left(\pi,c;\overline{Z},\overline{X}_{T}\right)
			-J_{i}(\mathbf{(\pi^{*},c^{*})})\\&\leq J_{i}\left((\pi^{i},c^{i}),(\mathbf{\pi}^{*},\mathbf{c}^{*})^{-i}\right)-J_{i}\left(\pi^{i},c^{i};\overline{Z},\overline{X}_{T}\right)\\&+\sup_{\left(\pi,c\right)\in\mathcal{A}^{i,p}_{0}}J_{i}\left(\pi,c;\overline{Z},\overline{X}_{T}\right)
			-J_{i}(\mathbf{(\pi^{*},c^{*})}).
		\end{split}
	\end{equation}
For  the first term of RHS of (\ref{appinepower}), we have 
	\begin{equation}\label{app1power}
		\begin{split}
			J_{i}&\left((\pi^{i},c^{i}),(\mathbf{\pi}^{*},\mathbf{c}^{*})^{-i}\right)-J_{i}\left(\pi^{i},c^{i};\overline{Z},\overline{X}_{T}\right)\\&=\mathbb{E}\left[\int_{0}^{T}\frac{1}{p_{1}}\left(\frac{c^{i}_{t}X^{i}_{t}}{\left(\overline{Z}^{*,n,-i}_{t}+\frac{1}{n}Z^{i}_{t}\right)^{\theta_{1}}}\right)^{p_{1}}dt-\int_{0}^{T}\frac{1}{p_{1}}\left(\frac{c^{i}_{t}X^{i}_{t}}{\left(\overline{Z}_{t}\right)^{\theta_{1}}}\right)^{p_{1}}dt\right]\\&+\mathbb{E}\left[\frac{1}{p_{1}}\left(\frac{X^{i}_{T}}{\left(\overline{X}^{*,n,-i}_{T}\cdot\left(X^{i}_{T}\right)^{\frac{1}{n}}\right)^{\theta_{1}}}\right)^{p_{1}}-\frac{1}{p_{1}}\left(\frac{X^{i}_{T}}{\left(\overline{X}_{T}\right)^{\theta_{1}}}\right)^{p_{1}}\right]\\&:=I^{(1)}_{i}+I^{(2)}_{i}.
		\end{split}
	\end{equation}
	Then,
	\begin{equation*}
		\begin{split}
			I^{\left(1\right)}_{i}&=\mathbb{E}\left[\int_{0}^{T}\frac{1}{p_{1}}\left(\frac{c^{i}_{t}X^{i}_{t}}{\left(\overline{Z}^{*,n,-i}_{t}+\frac{1}{n}Z^{i}_{t}\right)^{\theta_{1}}}\right)^{p_{1}}dt-\int_{0}^{T}\frac{1}{p_{1}}\left(\frac{c^{i}_{t}X^{i}_{t}}{\left(\overline{Z}^{*,n}_{t}\right)^{\theta_{1}}}\right)^{p_{1}}dt\right]\\&+\mathbb{E}\left[\int_{0}^{T}\frac{1}{p_{1}}\left(\frac{c^{i}_{t}X^{i}_{t}}{\left(\overline{Z}^{*,n}_{t}\right)^{\theta_{1}}}\right)^{p_{1}}dt-\int_{0}^{T}\frac{1}{p_{1}}\left(\frac{c^{i}_{t}X^{i}_{t}}{\left(\overline{Z}_{t}\right)^{\theta_{1}}}\right)^{p_{1}}dt\right]\\&:=I^{(3)}_{i}+I^{(4)}_{i},\\
				\end{split}
		\end{equation*}
	\begin{equation*}
	\begin{split}
				I^{(2)}_{i}&=\mathbb{E}\left[\frac{1}{p_{1}}\left(\frac{X^{i}_{T}}{\left(\overline{X}^{*,n,-i}_{T}\cdot\left(X^{i}_{T}\right)^{\frac{1}{n}}\right)^{\theta_{1}}}\right)^{p_{1}}-\frac{1}{p_{1}}\left(\frac{X^{i}_{T}}{\left(\overline{X}^{*,n}_{T}\right)^{\theta_{1}}}\right)^{p_{1}}\right]\\&+\mathbb{E}\left[\frac{1}{p_{1}}\left(\frac{X^{i}_{T}}{\left(\overline{X}^{*,n}_{T}\right)^{\theta_{1}}}\right)^{p_{1}}-\frac{1}{p_{1}}\left(\frac{X^{i}_{T}}{\left(\overline{X}_{T}\right)^{\theta_{1}}}\right)^{p_{1}}\right]\\&:=I^{(5)}_{i}+I^{\left(6\right)}_{i}.
		\end{split}
	\end{equation*}
	For the term $I^{(3)}_{i}$, it is clear that 
	\begin{equation}
		I^{(3)}_{i}=\mathbb{E}\left[\int_{0}^{T}\frac{1}{p_{1}}\left(c^{i}_{t}X^{i}_{t}\right)^{p_{1}}\left[\left(\overline{Z}^{*,n,-i}_{t}+\frac{1}{n}Z^{i}_{t}\right)^{-\theta_{1}p_{1}}-\left(\overline{Z}^{*,n}_{t}\right)^{-\theta_{1}p_{1}}\right]dt\right].
	\end{equation}
	Note that $p_{1}\in(0,1)$ and $\theta_{1}\in(0,1]$, using the inequality $$\left|a^{\theta_{1}p_{1}}-b^{\theta_{1}p_{1}}\right|\leq\theta_{1}p_{1}|a-b|\max\left\{a^{\theta_{1}p_{1}-1}, b^{\theta_{1}p_{1}-1}\right\}$$ for all $a,b>0$, we can derive on $\left\{Z^{*,i}_{t}>Z^{i}_{t}\right\}$:
	\begin{equation*}
		\begin{split}
			&\left(\overline{Z}^{*,n,-i}_{t}+\frac{1}{n}Z^{i}_{t}\right)^{-\theta_{1}p_{1}}-\left(\overline{Z}^{*,n}_{t}\right)^{\theta_{1}p_{1}}=\frac{\left[\left(\overline{Z}^{*,n}_{t}\right)^{\theta_{1}p_{1}}-\left(\overline{Z}^{*,n,-i}_{t}+\frac{1}{n}Z^{i}_{t}\right)^{\theta_{1}p_{1}}\right]}{\left(\overline{Z}^{*,n,-i}_{t}+\frac{1}{n}Z^{i}_{t}\right)^{\theta_{1}p_{1}}\left(\overline{Z}^{*,n}_{t}\right)^{\theta_{1}p_{1}}}\\
			&\leq\frac{\theta_{1}p_{1}}{n}\frac{1}{\left(\overline{Z}^{*,n,-i}_{t}\right)^{\theta_{1}p_{1}}\left(\overline{Z}^{*,n}_{t}\right)^{\theta_{1}p_{1}}}\left(Z^{*,i}_{t}-Z^{i}_{t}\right)\max\left\{\left(\overline{Z}^{*,n}_{t}\right)^{\theta_{1}p_{1}-1}, \left(\overline{Z}^{*,n,-i}_{t}+\frac{1}{n}Z^{i}_{t}\right)^{\theta_{1}p_{1}-1}\right\}\\&\leq\frac{\theta_{1}p_{1}}{n}\frac{1}{\left(\overline{Z}^{*,n,-i}_{t}\right)^{2\theta_{1}p_{1}}}Z^{*,i}_{t}\left(\overline{Z}^{*,n,-i}_{t}\right)^{\theta_{1}p_{1}-1}\\&=\frac{\theta_{1}p_{1}}{n}Z^{*,i}_{t}\left(\overline{Z}^{*,n,-i}_{t}\right)^{-\theta_{1}p_{1}-1}.
		\end{split}
	\end{equation*}
	Obviously, the above inequality trivially holds on  $\left\{Z^{*,i}_{t}\leq Z^{i}_{t}\right\}$. Using Jensen's inequality, we obtain 
	\begin{equation*}
		\left(\overline{Z}^{*,n,-i}_{t}\right)^{-\theta_{1}p_{1}-1}=\left(\frac{1}{n}\sum_{j\neq i}Z^{*,j}_{t}\right)^{-\theta_{1}p_{1}-1}\leq\left(\frac{n-1}{n}\right)^{-\theta_{1}p_{1}-1}\frac{1}{n-1}\sum_{j\neq i}\left(Z^{*,j}_{t}\right)^{-\theta_{1}p_{1}-1}.
	\end{equation*}
	Hence, using {H\"{o}lder} inequality for any $q_{1}, q_{2}>1$ satisfying $\frac{1}{q_{1}}+\frac{1}{q_{2}}+p_{1}=1$, we have from  Lemma \ref{lemma2power}:
	\begin{equation*}
		\begin{split}
			I^{(3)}_{i}&\leq\mathbb{E}\left[\int_{0}^{T}\frac{1}{p_{1}}\left(c^{i}_{t}X^{i}_{t}\right)^{p_{1}}\frac{\theta_{1}p_{1}}{n}Z^{*,i}_{t}\left(\overline{Z}^{*,n,-i}_{t}\right)^{-\theta_{1}p_{1}-1}dt\right]\\&\leq
			\frac{\theta_{1}}{n}\left(\frac{n-1}{n}\right)^{-\theta_{1}p_{1}-1}\frac{1}{n-1}\sum_{j\neq i}\mathbb{E}\left[\int_{0}^{T}\left(c^{i}_{t}X^{i}_{t}\right)^{p_{1}}Z^{*,i}_{t}\left(Z^{*,j}_{t}\right)^{-\theta_{1}p_{1}-1}dt\right]\\&\leq	\frac{\theta_{1}}{n}\left(\frac{n-1}{n}\right)^{-\theta_{1}p_{1}-1}\frac{1}{n-1}\sum_{j\neq i}\int_{0}^{T}\mathbb{E}\left[c^{1}_{t}X^{1}_{t}\right]^{p_{1}}\mathbb{E}\left[\left(Z^{*,i}_{t}\right)^{q_{1}}\right]^{\frac{1}{q_{1}}}\mathbb{E}\left[\left(Z^{*,j}_{t}\right)^{(-\theta_{1}p_{1}-1)q_{2}}\right]^{\frac{1}{q_{2}}}dt\\&\leq	\frac{\theta_{1}}{n}\left(\frac{n-1}{n}\right)^{-\theta_{1}p_{1}-1}\frac{1}{n-1}\sum_{j\neq i}C_{q_{1}}^{\frac{1}{q_{1}}}C^{\frac{1}{q_{2}}}_{(-\theta_{1}p_{1}-1)q_{2}}\left(\int_{0}^{T}\mathbb{E}\left[c^{i}_{t}X^{i}_{t}\right]^{p_{1}}dt\right)\\&\leq\frac{\theta_{1}}{n}\left(\frac{n-1}{n}\right)^{-\theta_{1}p_{1}-1}C_{q_{1}}^{\frac{1}{q_{1}}}C^{\frac{1}{q_{2}}}_{(-\theta_{1}p_{1}-1)q_{2}}T^{1-p_{1}}\left(\int_{0}^{T}\mathbb{E}\left[c^{i}_{t}X^{i}_{t}\right]dt\right)^{p_{1}}=O(n^{-1}),
		\end{split}
	\end{equation*}
	where the constant $C_{q}$ with $q\in(1,\infty)\cup(-\infty,0)$ is given in Lemma \ref{lemma2power}-\textbf{(ii)}.  In the last line of above estimation, we have just used {H\"{o}lder} inequality and the fact that for any $q>0$ and $\left(\pi^{i},c^{i}\right)\in\mathcal{A}^{i,p}_{0}$, 
	\begin{equation}\label{generalest}
		\begin{split}
			\sup_{t\in[0,T]}\mathbb{E}\left[\left(c^{i}_{t}X^{i}_{t}\right)^{q}\right]\leq M^{q}C_{q}=O(1),
		\end{split}
	\end{equation}
	where $M$ and  $C_{q}$ are constants defined in Assumption \ref{assump6} and  Lemma \ref{lemma2power}-\textbf{(i)}.
	
	Similarly,
	\begin{equation*}
		\begin{split}
			I^{(4)}_{i}&\leq\frac{1}{p_{1}}\mathbb{E}\left[\int_{0}^{T}\left(c^{i}_{t}X^{i}_{t}\right)^{p_{1}}\frac{1}{\left(\overline{Z}^{*,n}_{t}\right)^{\theta_{1}p_{1}}\left(\overline{Z}_{t}\right)^{\theta_{1}p_{1}}}\left[\left(\overline{Z}_{t}\right)^{\theta_{1}p_{1}}-\left(\overline{Z}^{*,n}_{t}\right)^{\theta_{1}p_{1}}\right]dt\right]\\\leq\theta_{1}&\mathbb{E}\left[\int_{0}^{T}\left(c^{i}_{t}X^{i}_{t}\right)^{p_{1}}\frac{1}{\left(\overline{Z}^{*,n}_{t}\right)^{\theta_{1}p_{1}}\left(\overline{Z}_{t}\right)^{\theta_{1}p_{1}}}\left|\overline{Z}_{t}-\overline{Z}^{*,n}_{t}\right|\max\left\{\left(\overline{Z}_{t}\right)^{\theta_{1}p_{1}-1},\left(\overline{Z}^{*,n}_{t}\right)^{\theta_{1}p_{1}-1}\right\}dt\right]\\\leq\theta_{1}&\mathbb{E}\left[\int_{0}^{T}\left(c^{i}_{t}X^{i}_{t}\right)^{p_{1}}\left|\overline{Z}_{t}-\overline{Z}^{*,n}_{t}\right|\max\left\{\left(\overline{Z}_{t}\right)^{-\theta_{1}p_{1}-1},\left(\overline{Z}^{*,n}_{t}\right)^{-\theta_{1}p_{1}-1}\right\}dt\right]\\\leq\theta_{1}&\mathbb{E}\left[\int_{0}^{T}\left(c^{i}_{t}X^{i}_{t}\right)^{p_{1}}\left|\overline{Z}_{t}-\overline{Z}^{*,n}_{t}\right|\left[\left(\overline{Z}_{t}\right)^{-\theta_{1}p_{1}-1}+\left(\overline{Z}^{*,n}_{t}\right)^{-\theta_{1}p_{1}-1}\right]dt\right].
		\end{split}
	\end{equation*}
	Using {H\"{o}lder} inequality, Lemma \ref{lemma2power}-\textbf{(iii)}, Proposition \ref{5.26.5} and (\ref{generalest}), we have for any given  admissible control $\left(\pi^{i},c^{i}\right)\in\mathcal{A}^{i,p}_{0}$,
	\begin{equation*}
		\begin{split}
			\mathbb{E}&\left[\int_{0}^{T}\left(c^{i}_{t}X^{i}_{t}\right)^{p_{1}}\left|\overline{Z}_{t}-\overline{Z}^{*,n}_{t}\right|\left(\overline{Z}_{t}\right)^{-\theta_{1}p_{1}-1}dt\right]\\&\leq\left(e^{-\delta T}z_{0}\right)^{-\theta_{1}p_{1}-1}\int_{0}^{T}\mathbb{E}\left[\left(c^{i}_{t}X^{i}_{t}\right)^{2p_{1}}\right]^{\frac{1}{2}}\mathbb{E}\left[\left|\overline{Z}_{t}-\overline{Z}^{*,n}_{t}\right|^{2}\right]^{\frac{1}{2}}dt\\&\leq C\sup_{t\in[0,T]}\mathbb{E}\left[\left|\overline{Z}_{t}-\overline{Z}^{*,n}_{t}\right|^{2}\right]^{\frac{1}{2}}=O(n^{-\frac{1}{2}}).
		\end{split}
	\end{equation*}
	On the other hand, by Jensen's inequality and {H\"{o}lder} inequality with $q_{1},q_{2}\geq2$ satisfying $\frac{1}{q_{1}}+\frac{1}{q_{2}}+\frac{1}{2}=1$, it follows that
	\begin{equation}
		\begin{split}
			\mathbb{E}&\left[\int_{0}^{T}\left(c^{i}_{t}X^{i}_{t}\right)^{p_{1}}\left|\overline{Z}_{t}-\overline{Z}^{*,n}_{t}\right|\left(\overline{Z}^{*,n}_{t}\right)^{-\theta_{1}p_{1}-1}dt\right]\\&\leq\frac{1}{n}\sum_{j=1}^{n}\int_{0}^{T}\mathbb{E}\left[\left(c^{i}_{t}X^{i}_{t}\right)^{p_{1}}\left|\overline{Z}_{t}-\overline{Z}^{*,n}_{t}\right|\left(\overline{Z}^{*,j}_{t}\right)^{-\theta_{1}p_{1}-1}dt\right]\\&\leq\frac{1}{n}\sum_{j=1}^{n}\int_{0}^{T}\mathbb{E}\left[\left(c^{i}_{t}X^{i}_{t}\right)^{q_{1}p_{1}}\right]^{\frac{1}{q_{1}}}\mathbb{E}\left[\left|\overline{Z}_{t}-\overline{Z}^{*,n}_{t}\right|^{2}\right]^{\frac{1}{2}}\mathbb{E}\left[\left(\overline{Z}^{*,j}_{t}\right)^{-(\theta_{1}p_{1}+1)q_{2}}\right]^{\frac{1}{q_{2}}}dt\\&\leq C\sup_{t\in[0,T]}\mathbb{E}\left[\left|\overline{Z}_{t}-\overline{Z}^{*,n}_{t}\right|^{2}\right]^{\frac{1}{2}}=O(n^{-\frac{1}{2}}),
		\end{split}
	\end{equation}
	where we have just used Lemma \ref{lemma2power}-\textbf{(ii)} and (\ref{generalest}) in the last inequality.
	Hence,  
	\begin{equation}
		I^{(1)}_{i}=I^{(3)}_{i}+I^{(4)}_{i}=O(n^{-\frac{1}{2}}).
	\end{equation}
	Recall that 
	\begin{equation}
		\begin{split}
			I^{(5)}_{i}=\frac{1}{p_{1}}\mathbb{E}\left[\left(X^{i}_{T}\right)^{p_{1}}\left(\frac{1}{\left(\overline{X}^{*,n,-i}_{T}\cdot\left(X^{i}_{T}\right)^{\frac{1}{n}}\right)^{\theta_{1}p_{1}}}-\frac{1}{\left(\overline{X}^{*,n}_{T}\right)^{\theta_{1}p_{1}}}\right)\right].
		\end{split}
	\end{equation}
	Similarly, using the inequality $\left|a^{\theta_{1}p_{1}}-b^{\theta_{1}p_{1}}\right|\leq\theta_{1}p_{1}|a-b|\max\left\{a^{\theta_{1}p_{1}-1}, b^{\theta_{1}p_{1}-1}\right\}$ for all $a,b>0$, we can derive on $\left\{X^{*,i}_{T}>X^{i}_{T}\right\}$:
	\begin{equation}\label{6.29.1}
		\begin{split}
			&\frac{1}{p_{1}}\left(X^{i}_{T}\right)^{p_{1}}\frac{\left(X^{*,i}_{T}\right)^{\frac{\theta_{1}p_{1}}{n}}-\left(X^{i}_{T}\right)^{\frac{\theta_{1}p_{1}}{n}}}{\left(\overline{X}^{*,n,-i}_{T}\left(X^{i}_{T}\right)^{\frac{1}{n}}\left(X^{*,i}_{T}\right)^{\frac{1}{n}}\right)^{\theta_{1}p_{1}}}\\&\leq\frac{\theta_{1}}{n}\left(X^{i}_{T}\right)^{p_{1}}\frac{\max\left\{\left(X^{*,i}_{T}\right)^{\frac{\theta_{1}p_{1}}{n}-1},         
				\left(X^{i}_{T}\right)^{\frac{\theta_{1}p_{1}}{n}-1}\right\}}{\left(\overline{X}^{*,n,-i}_{T}\left(X^{i}_{T}\right)^{\frac{1}{n}}\left(X^{*,i}_{T}\right)^{\frac{1}{n}}\right)^{\theta_{1}p_{1}}}\left(X^{*,i}_{T}-X^{i}_{T}\right)\\&\leq\frac{\theta_{1}}{n}\left(X^{i}_{T}\right)^{p_{1}}\frac{      
				\left(X^{i}_{T}\right)^{\frac{\theta_{1}p_{1}}{n}-1}}{\left(\overline{X}^{*,n,-i}_{T}\left(X^{i}_{T}\right)^{\frac{1}{n}}\left(X^{*,i}_{T}\right)^{\frac{1}{n}}\right)^{\theta_{1}p_{1}}}X^{*,i}_{T}\\&=\frac{\theta_{1}}{n}\left(X^{i}_{T}\right)^{p_{1}-1}\left(X^{*,i}_{T}\right)^{1-\frac{\theta_{1}p_{1}}{n}}\left(\overline{X}^{*,n,-i}_{T}\right)^{-\theta_{1}p_{1}}.
		\end{split}
	\end{equation}
	Obviously, the above inequality trivially holds on $\left\{X^{*,i}_{T}\leq X^{i}_{T}\right\}$. It follows from Jensen's inequality that
	\begin{equation}\label{6.29.2}
		\begin{split}
			\left(\overline{X}^{*,n,-i}_{T}\right)^{-\theta_{1}p_{1}}&=\exp\left\{-\theta_{1}p_{1}\frac{n-1}{n}\frac{1}{n-1}\sum_{j\neq i}\log\left(X^{*,j}_{T}\right)\right\}\\&\leq\frac{1}{n-1}\sum_{j\neq i}\left(X^{*,j}_{T}\right)^{-\theta_{1}p_{1}\left(\frac{n-1}{n}\right)}.
		\end{split}
	\end{equation}
	Hence, combining (\ref{6.29.1}) and (\ref{6.29.2}), and using  {H\"{o}lder} inequality with $q_{1},q_{2},q_{3}>1$ satisfying $\frac{1}{q_{1}}+\frac{1}{q_{2}}+\frac{1}{q_{3}}=1$, we obtain
	\begin{equation*}
		\begin{split}
			I^{(5)}_{i}&\leq\frac{\theta_{1}}{n}\frac{1}{n-1}\sum_{j\neq i}\mathbb{E}\left[\left(X^{i}_{T}\right)^{p_{1}-1}\left(X^{*,i}_{T}\right)^{1-\frac{\theta_{1}p_{1}}{n}}\left(X^{*,j}_{T}\right)^{-\theta_{1}p_{1}\left(\frac{n-1}{n}\right)}\right]\\&\leq\frac{\theta_{1}}{n}\frac{1}{n-1}\sum_{j\neq i}\mathbb{E}\left[\left(X^{i}_{T}\right)^{(p_{1}-1)q_{1}}\right]^{\frac{1}{q_{1}}}\mathbb{E}\left[\left(X^{*,i}_{T}\right)^{(1-\frac{\theta_{1}p_{1}}{n})q_{2}}\right]^{\frac{1}{q_{2}}}\mathbb{E}\left[\left(X^{*,j}_{T}\right)^{-\theta_{1}p_{1}\left(\frac{n-1}{n}\right)q_{3}}\right]^{\frac{1}{q_{3}}}\\&\leq\frac{\theta_{1}}{n}\frac{1}{n-1}\sum_{j\neq i}C^{\frac{1}{q_{1}}}_{(p_{1}-1)q_{1}}\left(\sup_{n\in\mathbb{N}}C^{\frac{1}{q_{2}}}_{(1-\frac{\theta_{1}p_{1}}{n})q_{2}}\right)\left(\sup_{n\in\mathbb{N}}C^{\frac{1}{q_{3}}}_{-\theta_{1}p_{1}\left(\frac{n-1}{n}\right)q_{3}}\right)=O(\frac{1}{n}),
		\end{split}
	\end{equation*}
	where the constant $C_{q}$ is defined in (\ref{keyx}).
	
	For the term $I^{(6)}_{i}$, using {H\"{o}lder} inequality and Lemma  \ref{lemma2power}-\textbf{(iii)}, we have 
	\begin{equation*}
		\begin{split}
			I^{(6)}_{i}&=\frac{1}{p_{1}}\mathbb{E}\left[\left(X^{i}_{T}\right)^{p_{1}}\left(\left(\overline{X}^{*,n}_{T}\right)^{-\theta_{1}p_{1}}-\left(\overline{X}_{T}\right)^{-\theta_{1}p_{1}}\right)\right]\\
			&\leq\frac{1}{p_{1}}\mathbb{E}\left[\left(X^{i}_{T}\right)^{p_{1}}\left|\left(\overline{X}^{*,n}_{T}\right)^{-\theta_{1}p_{1}}-\left(\overline{X}_{T}\right)^{-\theta_{1}p_{1}}\right|\right]\\&\leq\frac{1}{p_{1}}\mathbb{E}\left[\left(X^{i}_{T}\right)^{2p_{1}}\right]^{\frac{1}{2}}\mathbb{E}\left[\left|\left(\overline{X}^{*,n}_{T}\right)^{-\theta_{1}p_{1}}-\left(\overline{X}_{T}\right)^{-\theta_{1}p_{1}}\right|^{2}\right]^{\frac{1}{2}}\\&=O(n^{-\frac{1}{2}}).
		\end{split}
	\end{equation*}
	Therefore, it is clear that 
	\begin{equation}
		I^{(2)}_{i}=I^{(5)}_{i}+I^{(6)}_{i}=O(n^{-\frac{1}{2}}).
	\end{equation}
	Using (\ref{app1power}), we have
	\begin{equation}\label{Est1power}
		\begin{split}
			J_{i}&\left((\pi^{i},c^{i}),(\mathbf{\pi}^{*},\mathbf{c}^{*})^{-i}\right)-J_{i}\left(\pi^{i},c^{i};\overline{Z},\overline{X}_{T}\right)=I^{(1)}_{i}+I^{\left(2\right)}_{i}=O(n^{-\frac{1}{2}}).
		\end{split}
	\end{equation}
	For the second term of RHS of (\ref{appinepower}), we have 
	\begin{equation}
		\sup_{\left(\pi^{i},c^{i}\right)\in\mathcal{A}^{i,p}_{0}}J_{i}\left(\pi^{i},c^{i};\overline{Z},\overline{X}_{T}\right)
		-J_{i}(\mathbf{(\pi^{*},c^{*})})=J_{i}\left(\pi^{*,i},c^{*,i};\overline{Z},\overline{X}_{T}\right)
		-J_{i}(\mathbf{(\pi^{*},c^{*})}).
	\end{equation}
	Then, we  follow the similar argument in showing the convergence error of $I^{(4)}_{i}$ and $I^{(6)}_{i}$ with fixed strategy $\left(\pi^{*,i},c^{*,i}\right)$ to get 
	\begin{equation}\label{Est2power}
		J_{i}\left(\pi^{*,i},c^{*,i};\overline{Z},\overline{X}_{T}\right)
		-J_{i}(\mathbf{(\pi^{*},c^{*})})=O(n^{-\frac{1}{2}}).
	\end{equation}
	Therefore, the estimates (\ref{appinepower}), (\ref{Est1power}) and (\ref{Est2power}) jointly yield (\ref{appnashpower}) with $\epsilon_{n}=O(n^{-\frac{1}{2}})$.
\end{proof}
\section{ Numerical Illustrations of Mean Field Equilibrium}\label{sect5}
In this section, we  numerically illustrate the mean field equilibrium (MFE). {We only focus on the power utility case as it is common in applications and makes more sense in economic}.   For a given  type vector $o=(\mu,\sigma,p, \theta)$, we first numerically illustrate the MFE  in the case of homogeneous agents.  Given proper choices of parameters, we will {plot} some sensitivity results and interpret their financial implications.
 
From Figure \ref{fig1}, Figure \ref{fig2} and Figure \ref{fig3}, it is interesting to see that the monotonicity of mean field habit formation process $\overline{Z}$ depends on the relationship  between the initial habit level $z_{0}$ and the initial wealth level $x_{0}$: when the wealth level is adequate,  $\overline{Z}$ is increasing with respect to time $t$. In this case, the MFE consumption is also increasing in $t$. This observation can be explained by the fact that the competition nature of the problem motivates each agent to consume aggressively. The increasing of consumption can not only make agent obtain higher utility from consumption, but also strategically increase agent's habit level such that the population's average habit level can be lifted up and other competitor's  utility may drop; While, when the initial habit level $z_{0}$ is large and the initial wealth level $x_{0}$ is relatively low,  $\overline{Z}$ is either decreasing or exhibiting an initial decrease followed by an increase over time. One can interpret this pattern that when the habit level of the population is too high, the mean field habit formation process often drops down to a sustainable level and continues with another wave of growth. Furthermore, 
from  Figure \ref{fig:f}, it can be observed that the average habit formation reaches the sustainable level earlier as the population's risk aversion decreases. A similar kind of mean-reverting mechanism is also presented in \cite{2022habit}. The monotonicity of  MFE consumption  in this scenario becomes more {subtle}, highly relying on the model parameters. It turns out that the relationship between $z_{0}$ and $x_{0}$ is also a crucial component of the sensitivity results with regard to the model parameters. We will present some illustrations in the following content.

\begin{figure*} [htbp]
	\centering
	\subfloat[Consumption]{\label{fig:a}
		\includegraphics[scale=0.2]{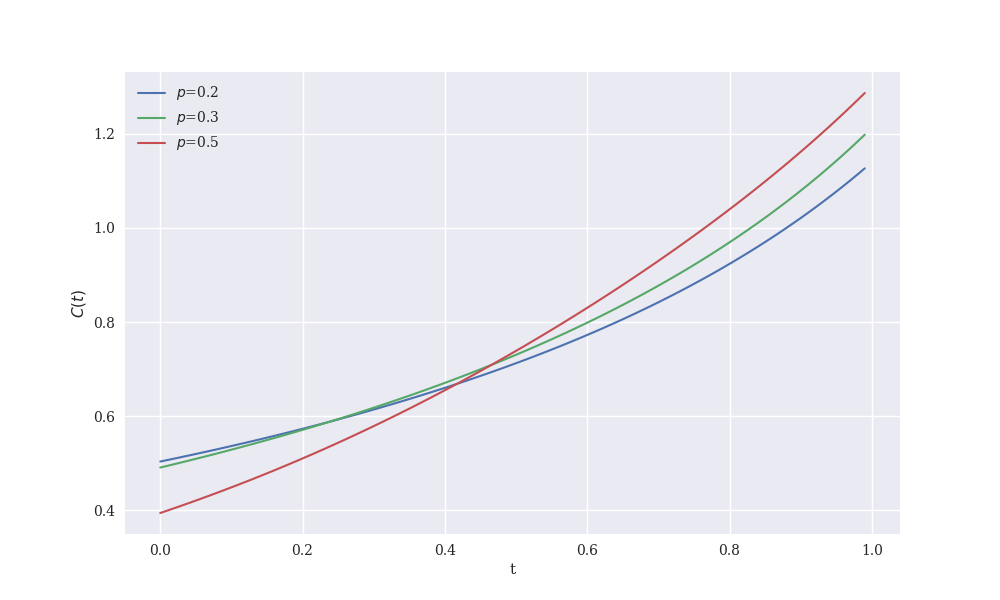}}
	\subfloat[Portfolio]{\label{fig:b}
		\includegraphics[scale=0.2]{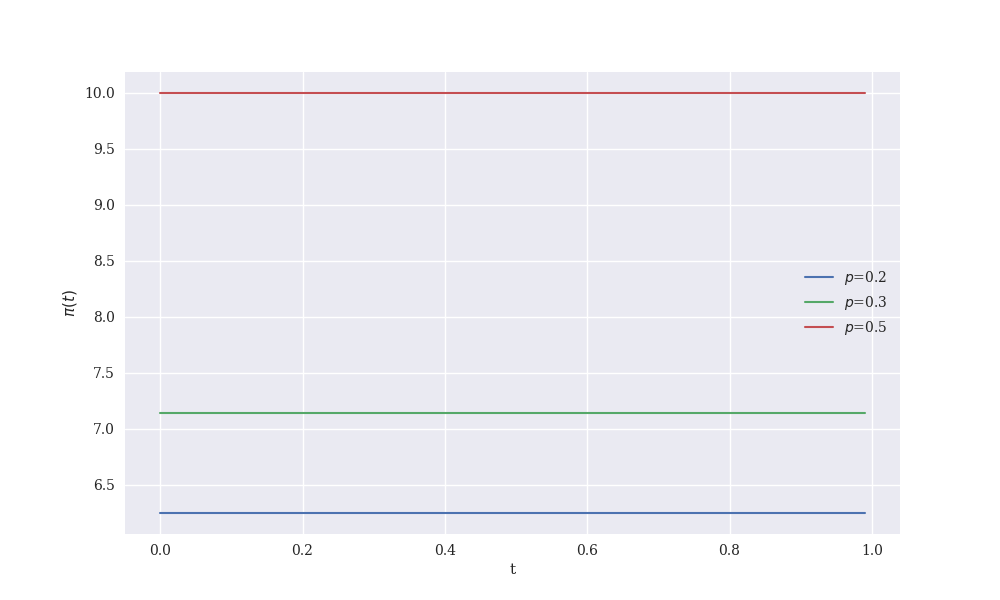}}
	\subfloat[Habit process (solid lines) and  average terminal wealth (dashed lines)]{\label{fig:c}
		\includegraphics[scale=0.2]{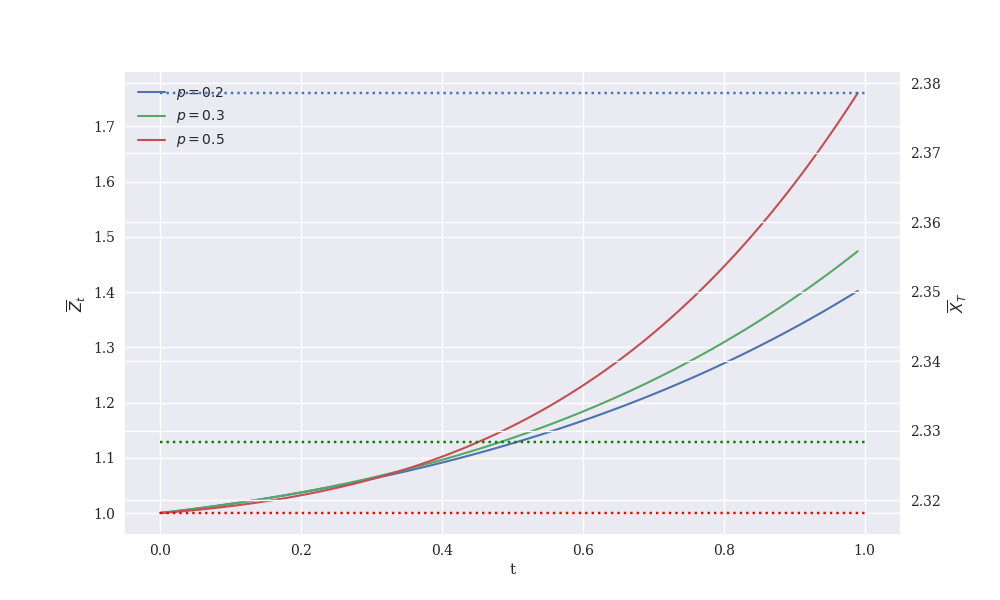}}
	
	\subfloat[Consumption]{\label{fig:d}
		\includegraphics[scale=0.2]{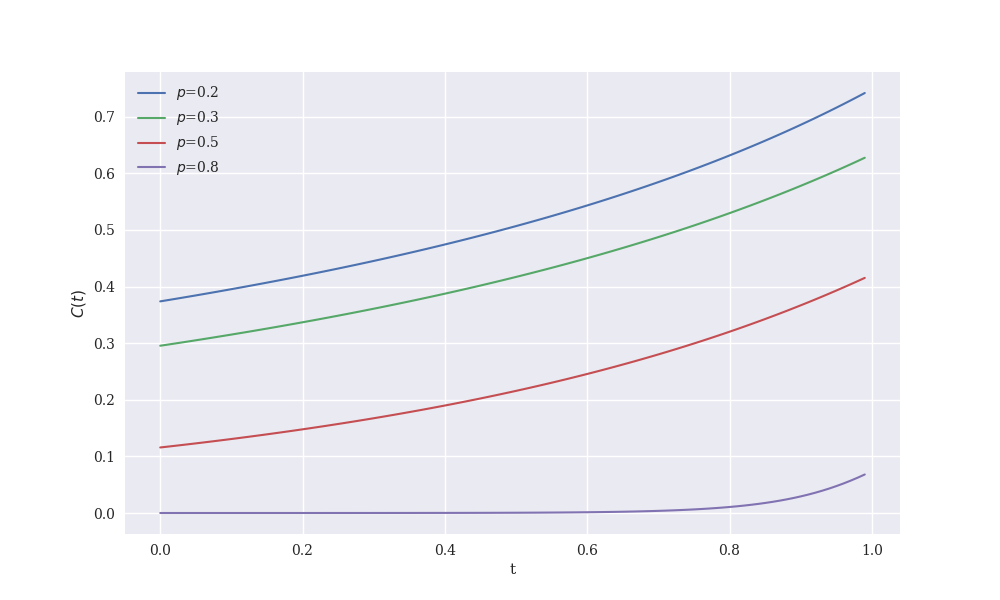}}
	\subfloat[Portfolio]{\label{fig:e}
		\includegraphics[scale=0.2]{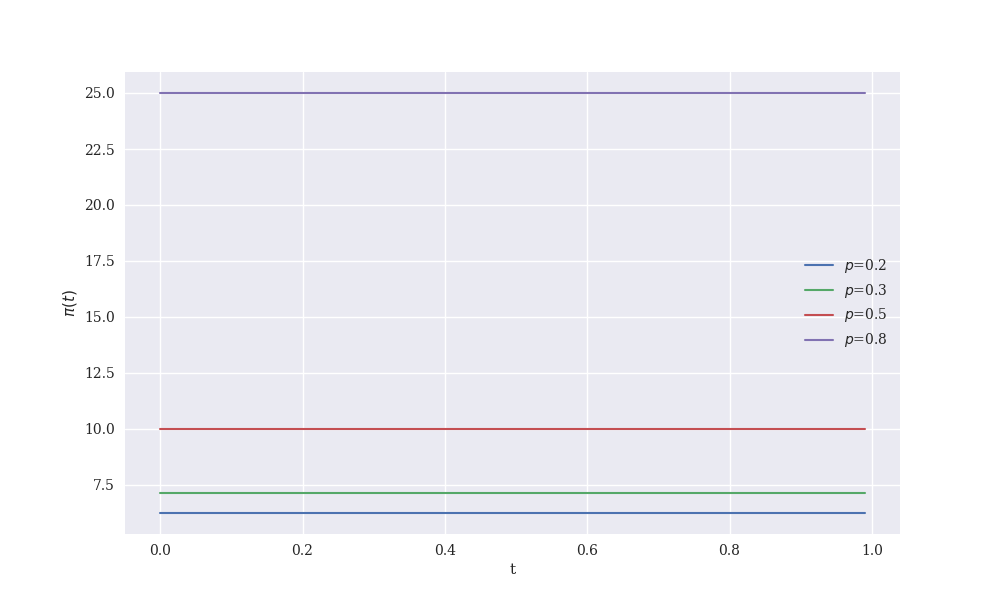}}
	\subfloat[Habit process (solid lines) and  average terminal wealth (dashed lines)]{\label{fig:f}
		\includegraphics[scale=0.2]{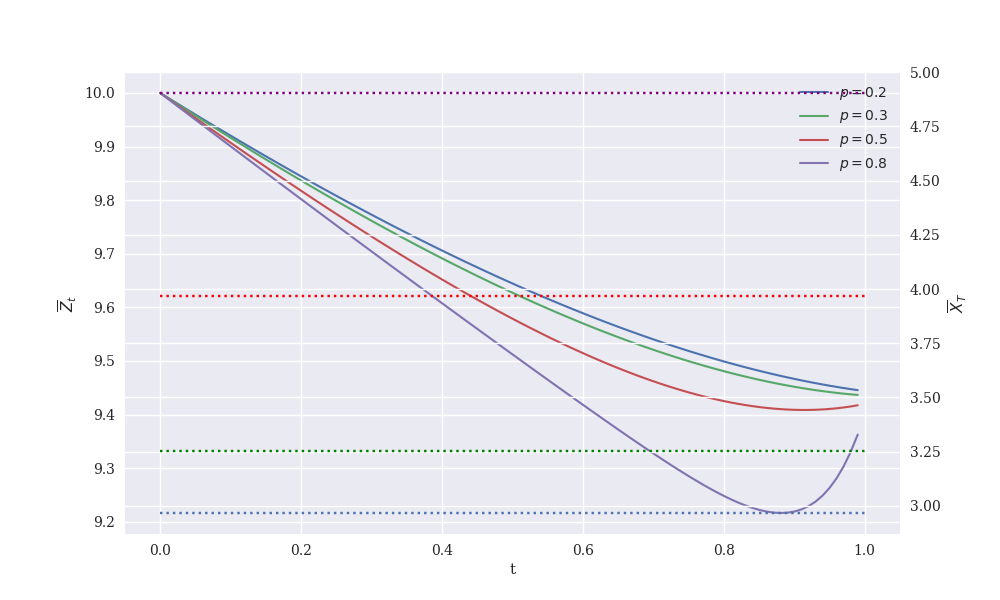}}
	\caption{\textbf{Top panel}: The MFE consumption  $c$, the MFE portfolio $\pi$, the habit formation process $\overline{Z}$ and average terminal wealth $\overline{X}_{T}$ with risk aversion parameters $p=0.2, 0.3$ and $0.5$  in low initial habit level case ($z_{0}=1, x_{0}=5$). \textbf{Bottom panel}: The MFE consumption  $c$, the MFE portfolio $\pi$, the habit formation process $\overline{Z}$ and average terminal wealth $\overline{X}_{T}$ with risk aversion parameters $p=0.2, 0.3, 0.5$ and $0.8$  in high initial habit level case ($z_{0}=10, x_{0}=5$).}  
	\label{fig1} 
\end{figure*}

Let us first illustrate in Figure \ref{fig1} the sensitivity results of the MFE on the risk aversion parameter $p$.
We choose and fix the model parameters $T=1$, $\delta=0.1$, $x_{0}=5$, $z_{0}=1$, $\mu=0.2$, $\sigma=0.2$, $\theta=1$ and take different values $p=0.2$, $0.3$ and $0.5$ in the low initial habit level case, and choose and fix the model parameters $T=1$, $\delta=0.1$, $x_{0}=5$, $z_{0}=10$, $\mu=0.2$, $\sigma=0.2$, $\theta=1$ and take different values $p=0.2$, $0.3$, $0.5$ and  $0.8$ in the high initial habit level case. It can be observed that the MFE portfolio increases with $p$, indicating that an individual investor allocates less money to the risky asset when she is more risk averse. From the top panel, it is interesting to note that after the accumulation over a period of time, both the consumption and the habit formation process  turn out to grow with $p$, showing that agents consume more aggressively in a more risk-seeking environment. As a result, Figure \ref{fig:c} shows that the cautious agents  accumulate more wealth. However, the bottom panel shows that the consumption and the habit formation process decrease with $p$ when the initial habit level is high.  In this case, the risk-seeking agents get richer. This can be explained by the fact that the larger $p$ value implies that the negative impact of the average habit level on expected utility becomes more significant. Therefore, the agent with larger $p$ value chooses to spend less on consumption to rapidly  bring down the average habit process to an appropriate level while investing more in risky assets to increase wealth.

In Figure \ref{fig2}, we numerically illustrate the impact of the habit formation intensity parameter $\delta$. We choose and fix the model parameters $T=1$, $p=0.3$, $x_{0}=5$, $z_{0}=1$, $\mu=0.2$, $\sigma=0.2$, $\theta=1$ and take different values $\delta=0.1$, $0.3$ and $0.5$ in the low initial habit level case, and choose and fix the model parameters $T=1$, $p=0.8$, $x_{0}=8$, $z_{0}=10$, $\mu=0.2$, $\sigma=0.2$, $\theta=1$ and take different values $\delta=0.1$, $0.3$ and $0.5$ in the high initial habit level case. As the MFE portfolio does not depend on $\delta$,  we only consider the MFE consumption  and mean field habit formation process.  From the right panel, we observe that the habit formation process does not exhibit the same monotonicity in $\delta$ in both circumstances. This can be explained by the fact that the habit formation process is defined as the combination of the discounted initial value $z_{0}e^{-\delta t}$ (decreasing in $\delta$) and the weighted average $\int_{0}^{t}\delta e^{\delta(s-t)}C_{s}ds$ (usually increasing in $\delta$).  When the initial level $z_{0}$ is small, as plotted in the top-right panel, the term $z_{0}e^{-\delta t}$ becomes negligible compared to the integral term $\int_{0}^{t}\delta e^{\delta(s-t)}C_{s}ds$, and hence  $\overline{Z}$ is increasing in $\delta$; when the initial habit  $z_{0}$ is large, the term $z_{0}e^{-\delta t}$  initially dominates  the integral term $\int_{0}^{t}\delta e^{\delta(s-t)}C_{s}ds$, and $\overline{Z}$ is decreasing in $\delta$ in the first stage;  The integral term gradually dominates the term $z_{0}e^{-\delta t}$ over time and as a result, after accumulation over some period $\overline{Z}$ is increasing in $\delta$.  We can see from the top-left panel that consumption growth slows down as the value of $\delta$ increases.   Essentially,  $\delta$ affects MFE consumption solely by altering the value of $\overline{Z}$. A higher $\overline{Z}$ leads to a reduced growth in consumption as slowing down consumption growth allows the agent to  avoid the expected utility loss caused by the rapid growth of the habit process. From the bottom-left panel, we observe that the consumption has a hump-shaped pattern, with spending typically  first increasing  and then decreasing in time. This is a new feature comparing to the result in \cite{2020Many}. It can also be observed from Figure \ref{fig:i} that a higher $\delta$  leads to an earlier consumption hump and if $\delta$ is sufficiently small, the consumption exhibits a monotonically increasing trend over time indicating that the agents with high $\delta$ value are more easily to be self-satisfaction.  These
effects are expected since increasing $\delta$ strengthens the habit formation, and all the aforementioned effects are  consequences of the strengthened habit formation.
 \begin{figure*} [htbp]
	\centering
	\subfloat[Consumption]{\label{fig:g}
		\includegraphics[scale=0.25]{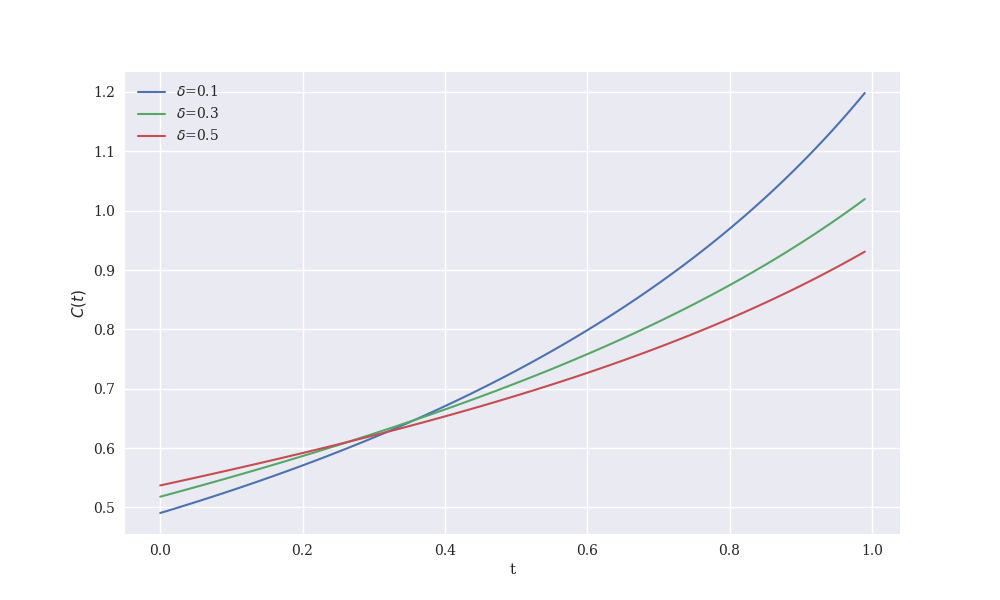}}
	\subfloat[Habit process (solid lines) and average terminal wealth (dashed lines)]{\label{fig:h}
		\includegraphics[scale=0.25]{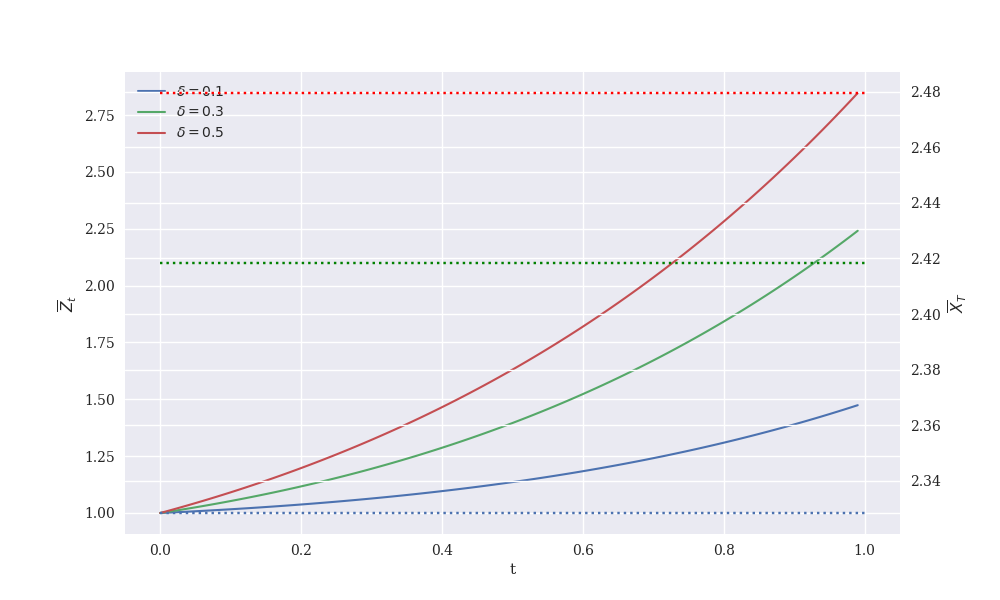}}
	
	\subfloat[Consumption]{\label{fig:i}
		\includegraphics[scale=0.25]{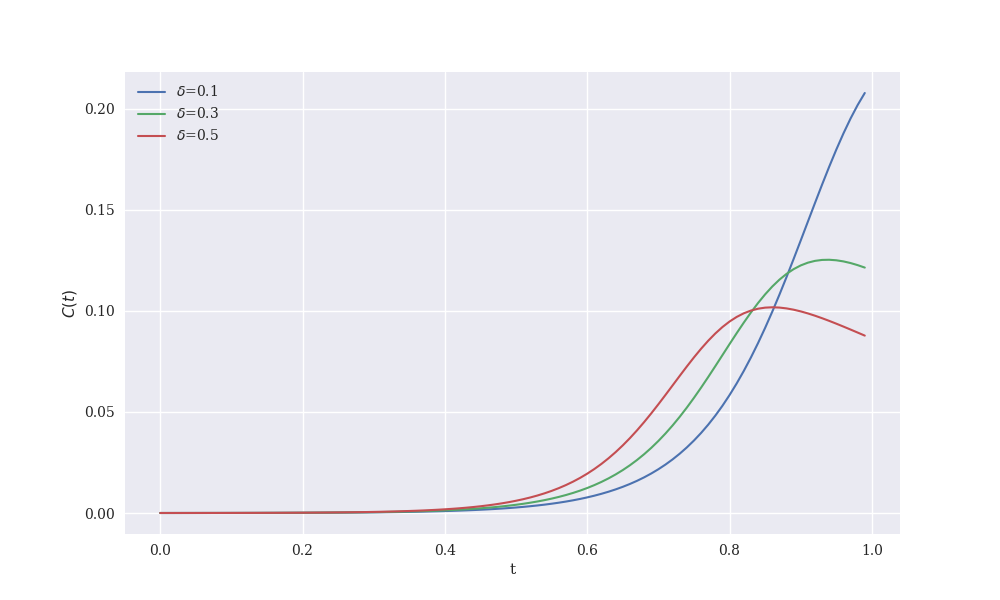}}
	\subfloat[Habit process (solid lines) and average terminal wealth (dashed lines)]{\label{fig:j}
		\includegraphics[scale=0.25]{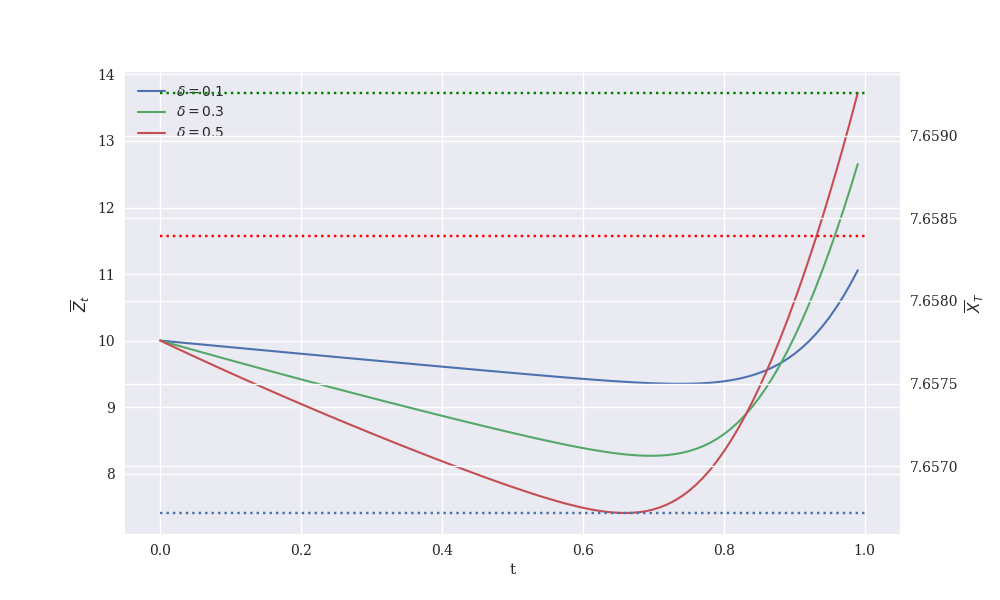}}
	\caption{\textbf{Top panel}: The MFE consumption  $c$, the habit formation process $\overline{Z}$ and average terminal wealth $\overline{X}_{T}$ with habit persistence parameters $\delta=0.1, 0.3$ and $0.5$  in low initial habit case ($z_{0}=1, x_{0}=5$). \textbf{Bottom panel}: The MFE consumption  $c$, the habit formation process $\overline{Z}$ and average terminal wealth $\overline{X}_{T}$ with habit persistence parameters $\delta=0.1, 0.3$ and $0.5$  in high initial habit case ($z_{0}=10, x_{0}=8$).}
	\label{fig2} 
\end{figure*}

We illustrate in Figure \ref{fig3} how the competition parameter $\theta$ affects the MFE.  We choose and fix the model parameters $T=1$, $p=0.3$, $x_{0}=5$, $z_{0}=1$, $\mu=0.2$, $\sigma=0.2$, $\delta=0.1$ and take different values $\theta=0.5$, $0.8$ and $1.0$ in the low initial habit level case, and choose and fix the model parameters $T=1$, $p=0.3$, $x_{0}=5$, $z_{0}=10$, $\mu=0.2$, $\sigma=0.2$, $\delta=0.1$ and take different values $\theta=0.5$, $0.8$ and $1.0$ in the high initial habit level case. By the same reason,  we only consider the MFE consumption and mean field habit formation process.  From the top panel, we  see that both  $\overline{Z}$ and  $c$ increase with $\theta$, indicating that the more competitive the agent is, the higher consumption policy she chooses and the average habit level of the population also gets larger. However, when the initial habit level $z_{0}$ is large, the bottom panel shows that both  $\overline{Z}$ and $c$  decrease with $\theta$. 
This is not surprising since the more competitive agent in this  scenario  aims to consume less in order to  pull down the excessive habit process rapidly  and attain a higher expected utility. 
\begin{figure*} [htbp]
	\centering
	\subfloat[Consumption]{\label{fig:m}
		\includegraphics[scale=0.25]{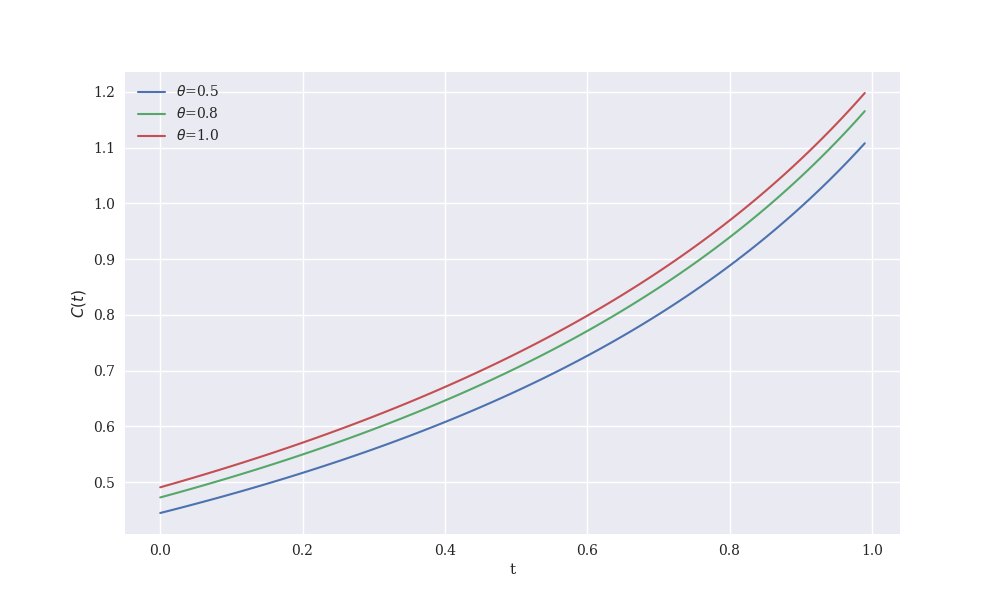}}
	\subfloat[Habit process (solid lines) and average terminal wealth (dashed lines)]{\label{fig:n}
		\includegraphics[scale=0.25]{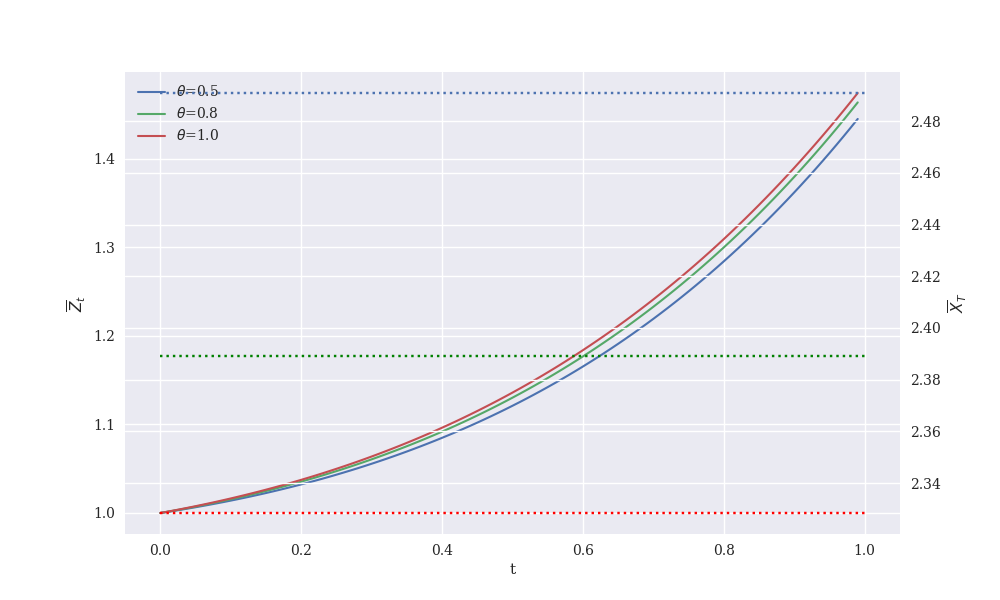}}
	
	\subfloat[Consumption]{\label{fig:o}
		\includegraphics[scale=0.25]{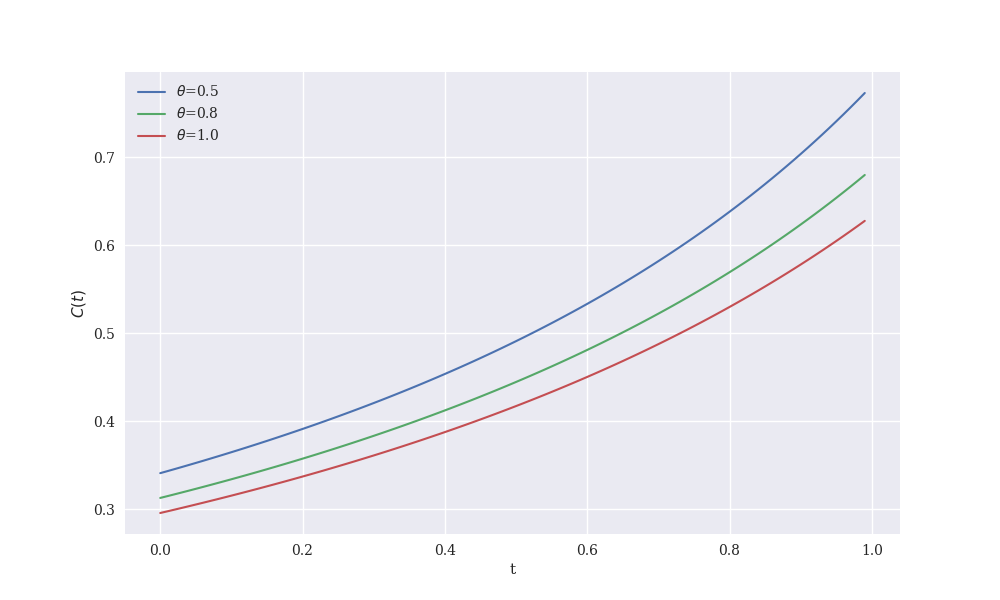}}
	\subfloat[Habit process (solid lines) and average terminal wealth (dashed lines)]{\label{fig:p}
		\includegraphics[scale=0.25]{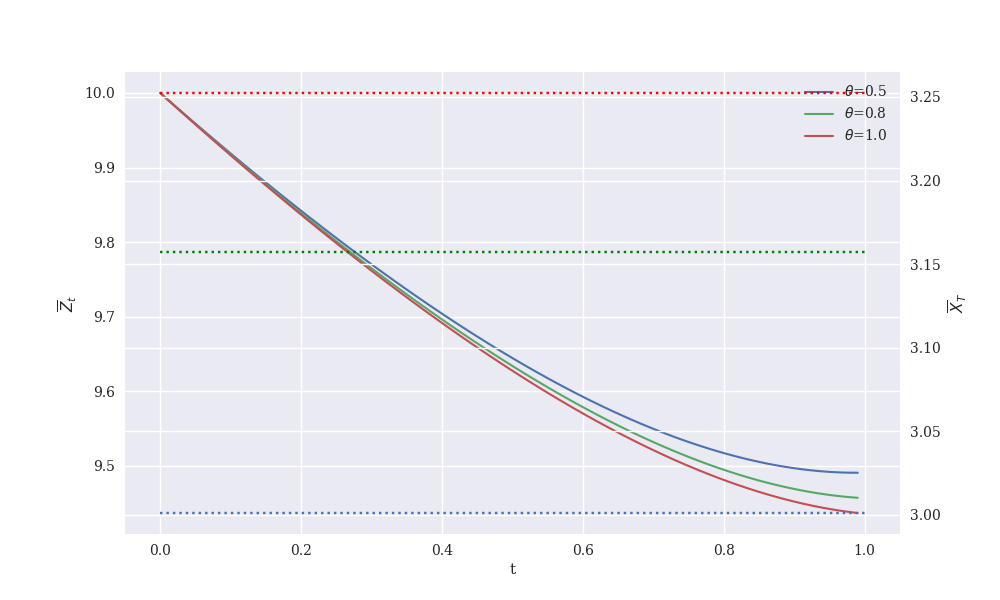}}
	\caption{\textbf{Top panel}: The MFE consumption  $c$, the habit formation process $\overline{Z}$ and average terminal wealth $\overline{X}_{T}$ with competition parameters $\theta=0.5, 0.8$ and $1.0$  in low initial habit case ($z_{0}=1, x_{0}=5$). \textbf{Bottom panel}: The MFE consumption  $c$, the habit formation process $\overline{Z}$ and average terminal wealth $\overline{X}_{T}$ with competition parameters $\theta=0.5, 0.8$ and $1.0$  in high initial habit case ($z_{0}=10, x_{0}=5$).}
	\label{fig3} 
\end{figure*}

\begin{figure*} [htbp]
	\centering
	\subfloat[Consumption]{\label{fig:revise1}
		\includegraphics[scale=0.25]{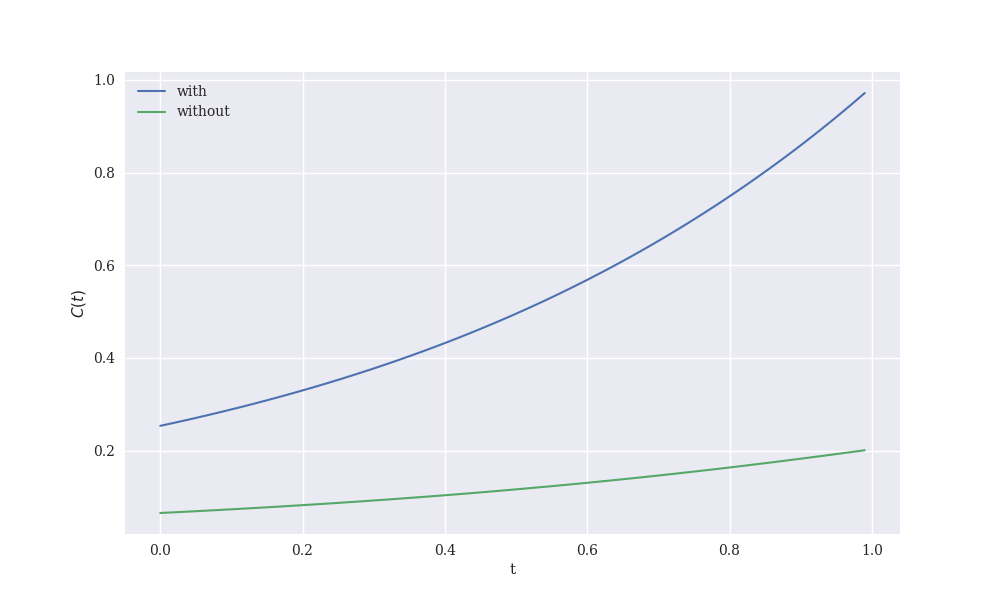}}
	\subfloat[Habit process]{\label{fig:revise2}
		\includegraphics[scale=0.25]{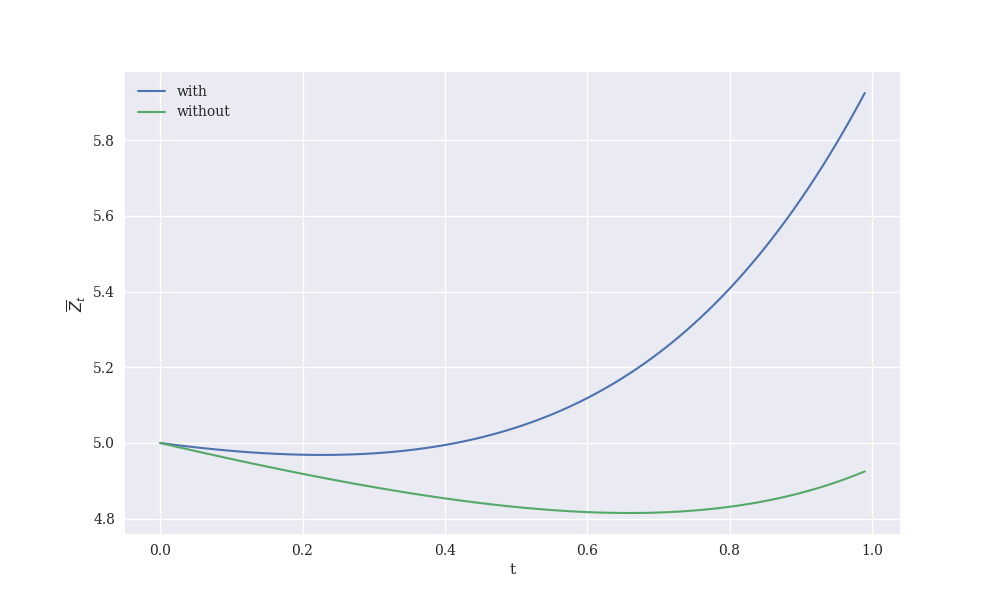}}
	
	\subfloat[Consumption]{\label{fig:revise3}
		\includegraphics[scale=0.25]{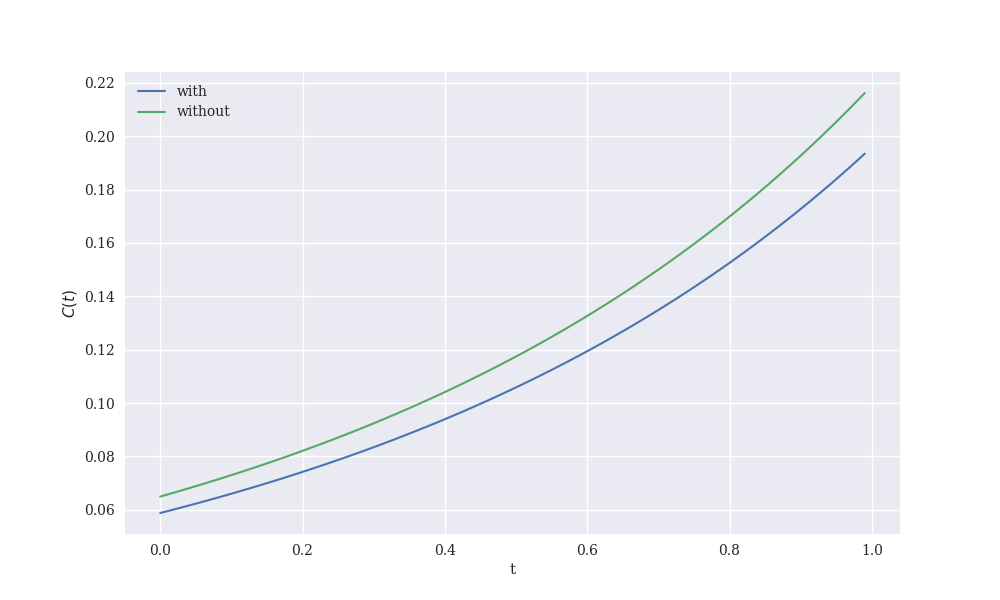}}
	\subfloat[Habit process]{\label{fig:revise4}
		\includegraphics[scale=0.25]{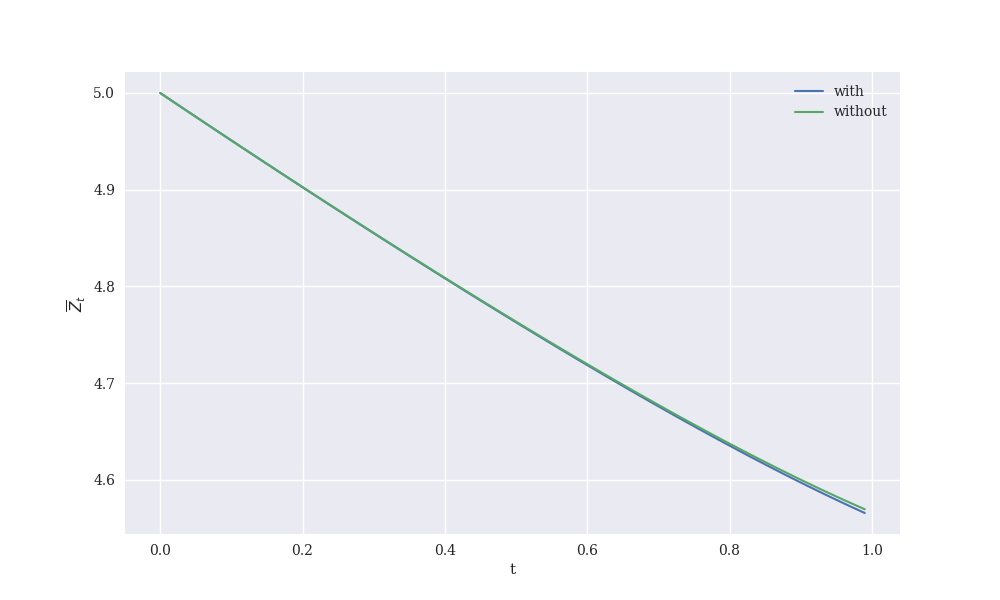}}
	\caption{{\textbf{Top panel}: The MFE consumption  $c$, the habit formation process $\overline{Z}$  in high initial wealth case ($z_{0}=5, x_{0}=10$). \textbf{Bottom panel}: The MFE consumption  $c$, the habit formation process $\overline{Z}$  in low initial wealth case ($z_{0}=5, x_{0}=1$). The blue lines labeled 'with' represent the consumption or habit formation process with relative wealth concerns. The green lines labeled 'without' depict the consumption or habit formation process without relative wealth concerns.}}
	\label{newfig} 
\end{figure*}

{Compared with the results only considering consumption habit formation in \cite{2022habit}, we illustrate the influence of the newly introduced average wealth benchmark on the MFE in Figure \ref{newfig}. We choose and fix the model parameters $T=1$, $p=0.5$, $x_{0}=10$, $z_{0}=5$, $\mu=0.2$, $\sigma=0.2$, $\delta=0.1$ and $\theta=1$  in the high initial wealth level case, and choose and fix the model parameters $T=1$, $p=0.5$, $x_{0}=1$, $z_{0}=5$, $\mu=0.2$, $\sigma=0.2$, $\delta=0.1$  and $\theta=1$ in the low initial wealth level case. We focus solely on the MFE consumption and mean field habit formation process because the the MFE portfolios in both cases are the same. From the top panel, we observe that both the MFE consumption and the mean field habit formation process with relative wealth concerns are higher compared to the scenario without relative wealth concerns.  However, in the case with low initial wealth and high initial habit, the bottom panel shows that both the MFE consumption and the mean field habit formation process with relative wealth concerns are lower. This is reasonable as in the first case, higher initial wealth leads to a significant negative impact from the average wealth benchmark on expected utility. Consequently, rather than choosing to accumulate wealth, agents are more inclined to increase consumption to attain higher expected utility. 
In the second scenario with low initial wealth and high initial habit, 
 concerning the average wealth benchmark, agents are more inclined to decrease consumption, accumulate wealth, and thus attain higher expected  utility.}

\begin{figure*} [htbp]
	\centering
	\subfloat[Consumption]{\label{fig:x}
		\includegraphics[scale=0.2]{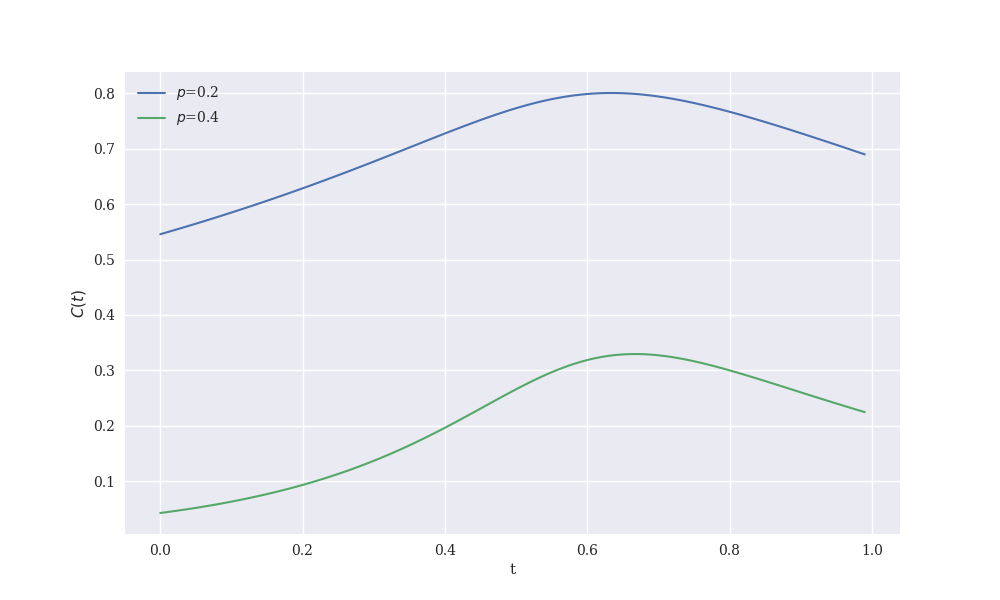}}
	\subfloat[Portfolio]{\label{fig:y}
		\includegraphics[scale=0.2]{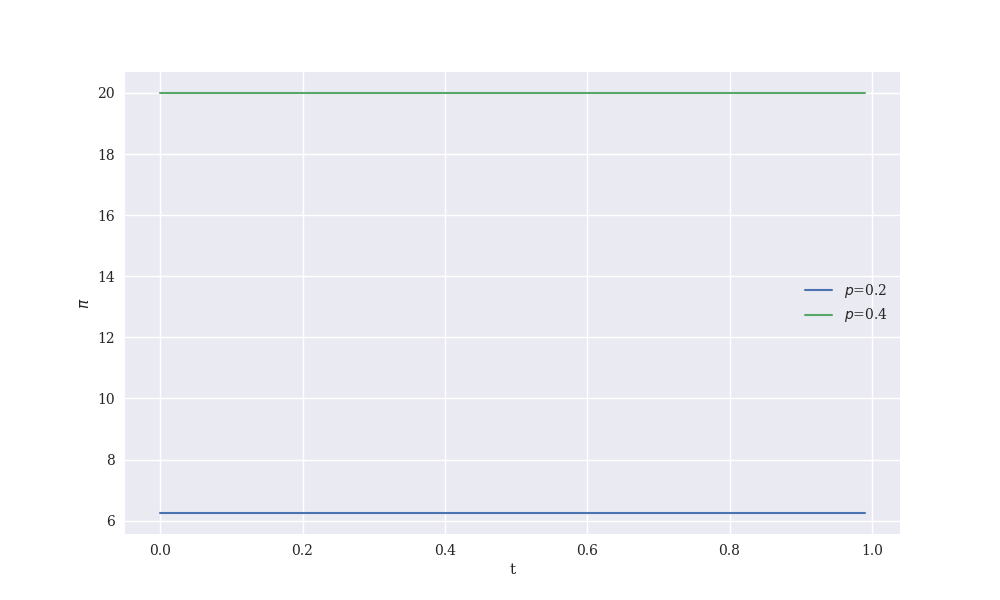}}
	\subfloat[Habit process]{\label{fig:z}
		\includegraphics[scale=0.2]{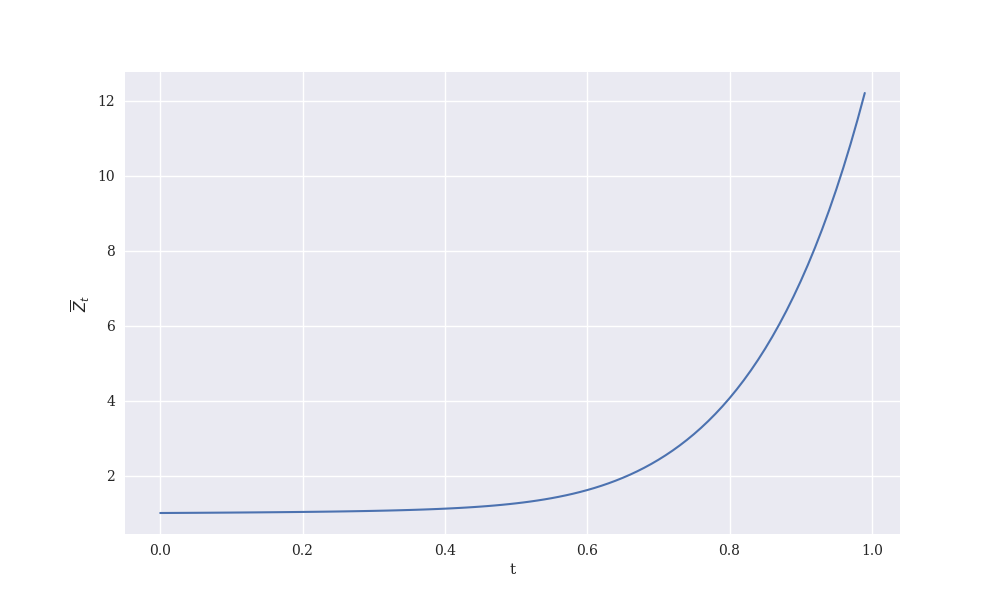}}
	\caption{ The MFE consumption  $c$, the MFE portfolio $\pi$ and the habit formation process $\overline{Z}$ with two types of  agents.}  
	\label{fig4} 
\end{figure*}

At last, we do some numerical experiments for the scenario including heterogeneous agents.  We only show results that differ from those of homogeneous agents.  We assume that there are two different agents in the market with type vectors $(\mu_{1}=0.2, \sigma_{1}=0.2, p_{1}= 0.2, \theta_{1}=1.0)$ and $(\mu_{2}=0.4, \sigma_{2}=0.2, p_{2}= 0.5, \theta_{2}=1.0)$, respectively.  Let us fix model parameters  $T=1$, $x_{0}=5$, $z_{0}=1$, $\delta=0.1$, $F\left(\left\{1\right\}\right)=0.7$ and $F\left(\left\{2\right\}\right)=0.3$. Although the initial habit level is relative low, we still observe form the left panel of Figure \ref{fig4} that both agents display the hump-shaped consumption path. In this case, the risk-averse agent allocates more wealth to consumption and less to risky asset. Combining Figure \ref{fig2} and Figure \ref{fig4}, we can now provide an explanation to the hump-shaped consumption. At the beginning of time horizon, the population's average habit level either decreases over time or increases slowly (see Figure \ref{fig:j} and Figure \ref{fig:z}), and the competition mechanism takes the leading role in each agent's decision making: each agent would increase her consumption to attain a larger expected utility. As time moves on, the average habit level starts to increase rapidly, indicating a high future standard of living.  Each agent is more likely to decrease her consumption   after a period of time such that the growth of the population's habit level may slow down and everyone may suffer less from the high future benchmark.

\section{Conclusions and Open Problems}\label{sect6}
We study the relative investment-consumption games  with  habit formation. Each agent regards the average of habit formation and wealth from all peers  as  benchmarks to evaluate the performance of her decision.   We formulate and solve the mean field game problems for agents having exponential or power utilities. The mean field equilibrium is characterized in analytical form in each problem. For both problems, we also establish the connection to the n-agent game by constructing its approximate Nash equilibrium using the obtained mean field equilibrium and the explicit order of approximation error is derived.

{It is still open how to tackle the our problem with common noise. In this scenario, the mean field habit formation process $\overline{Z}$ and  mean field terminal wealth $\overline{X}_{T}$ are no longer deterministic. The mathematical arguments in the present paper are no longer applicable. One potential approach to address this problem is to follow the arguments in \cite{2022Meanfieldpg} and \cite{2022Meanfieldpgc}, attempting to establish a one-to-one correspondence between the mean field equilibria and solutions to some  Forward-Backward Stochastic Differential Equation (FBSDE).

In the first problem with exponential utility, it is also an open question when imposing a nonnegativity constraints on consumption. The dual transformation will be involved in handling the corresponding HJB equation. It is not hard to show that the solution of the linearised dual PDE admits a probabilistic representation through the Feynman-Kac formula. Then the value function and the the best response  can be obtained in terms of the primal variables after an inverse transform. However, the best response is not fully explicit which brings new challenges for subsequent fixed point problems. Moreover, it is interesting to consider this problem with an additional constraint on consumption; that is, each agent cannot allow their consumption to fall below the external habit level as in \cite{2022habit}. Similar difficulties arise in solving the corresponding stochastic control problem,  which also poses challenges for subsequent fixed point problems.
 
  }
\vskip 10pt
{\bf Acknowledgements.}
The authors acknowledge the support from the National Natural Science Foundation of China (Grant No.12271290 and No.11871036). The authors also thank the members of the group of Actuarial Science and Mathematical Finance at the Department of Mathematical Sciences, Tsinghua University for their feedbacks and useful conversations. We are also particularly grateful to the two anonymous referees and
the Editor  whose suggestions greatly improve the manuscript’s quality.
\vskip 10pt
{\bf Data availability statement.} 
Data sharing  not applicable to this article as no datasets were generated or analyzed during the current study.
\appendix
\renewcommand{\appendixname}{Appendix~\Alph{section}}
\renewcommand{\theequation}{\thesection.\arabic{equation}}

\section{Some Estimations for the Fixed Point System (\ref{new fixedpower})}\label{5.26.4}

In this section, we make some prior estimations on $\left(\overline{X}_{T},\hat{Z}\right)$ in the system (\ref{new fixedpower}).
\begin{proposition}\label{5.26.2}
	For any given $\hat{Z}\in\mathcal{C}_{T,+}$, there exists a unique solution $\overline{X}^{\hat{Z}}_{T}$ to the following equation
	\begin{equation}\label{5.26.1}
		\overline{X}_{T}=C_{2}\exp\left\{-\int_{0}^{T}\sum_{k=1}^{K}\exp\left(\frac{\theta_{k} p_{k}\delta}{1-p_{k}}t\right)\left(\hat{Z}_{t}\right)^{\frac{\theta_{k}p_{k}}{p_{k}-1}}\hat{G}_{k}(t,\overline{X}_{T},\hat{Z})dt\right\},
	\end{equation}
	where $C_{2}$ is defined in (\ref{7.17power}).
	Moreover, we have the following estimation:
	\begin{equation}
		C_{1}\leq\overline{X}^{\hat{Z}}_{T}\leq C_{2},
	\end{equation} 
	where $C_{1}$ and $C_{2}$ are two positive constants only depending on $\left\{o_{k}\right\}_{k=1}^{K}, z_{0}, \delta$ and $T$.
\end{proposition}
\begin{proof}
	As the $LHS$ of (\ref{5.26.1}) is strictly increasing with respect to $\overline{X}_{T}$ and the $RHS$ of (\ref{5.26.1}) is strictly decreasing with respect to $\overline{X}_{T}$, the equation (\ref{5.26.1}) has at most one solution. 
	
	When $\overline{X}_{T}=0$, we get  $LHS<RHS$; When $\overline{X}_{T}$ tends to infinity, we get  $LHS>RHS$. Hence, for any given $\hat{Z}\in\mathcal{C}_{T,+}$, there exists a unique solution $\overline{X}^{\hat{Z}}_{T}$ to (\ref{5.26.1}).
	
	It is obviously to see that $\overline{X}^{\hat{Z}}_{T}\leq C_{2}$. Note that
	\begin{equation}
		\hat{G}_{k}(s,\overline{X}_{T},\hat{Z})\leq e^{ a_{k}(s-T)}\left(\overline{X}_{T}\right)^{\frac{p_{k}\theta_{k}}{1-p_{k}}},
	\end{equation}
	we have from (\ref{5.26.1})
	\begin{equation}
		\begin{split}
			\overline{X}_{T}&\geq C_{2}\exp\left\{-\int_{0}^{T}\sum_{k=1}^{K}\exp\left(\frac{\theta_{k} p_{k}\delta}{1-p_{k}}s\right)z_{0}^{\frac{\theta_{k}p_{k}}{p_{k}-1}}e^{ a_{k}(s-T)}\left(\overline{X}_{T}\right)^{\frac{p_{k}\theta_{k}}{1-p_{k}}}ds\right\}\\
			&= C_{2}\exp\left\{-\sum_{k=1}^{K}D_{k}\left(\overline{X}_{T}\right)^{\frac{p_{k}\theta_{k}}{1-p_{k}}}\right\},
		\end{split}
	\end{equation}
	where  $D_{k}:=z_{0}^{\frac{\theta_{k}p_{k}}{p_{k}-1}}e^{-a_{k}T}\int_{0}^{T}\exp\left[(\frac{\theta_{k}p_{k}\delta}{1-p_{k}}+a_{k})s\right]ds$. 
	Then it is straightforward to get $\overline{X}_{T}>C_{1}$, where $C_{1}$ is the root of the following equation
	\begin{equation}
		C_{1} = C_{2}\exp\left(-\sum_{k=1}^{K}D_{k}C_{1}^{\frac{p_{k}\theta_{k}}{1-p_{k}}}\right),
	\end{equation}
	which completes the proof.
\end{proof}
\begin{proposition}\label{5.26.5}
	If there exists a solution $\left(\overline{X}_{T},\hat{Z}\right)$ of (\ref{new fixedpower}), it holds that
	$$z_{0}\leq\hat{Z}_{t}\leq M(t)\leq M(T),\quad \, \forall t \in [0,T],$$
	where  $M(\cdot)$ is a positive increasing function only depending on $\left(x_{0},z_{0}, \delta,T\right)$ and $\left\{o_{k}\right\}_{k=1}^{K}$.
\end{proposition}
\begin{proof}
	It is enough to find a function $M(\cdot)$ and prove  $\hat{Z}_{t}\leq M(t)$ for all $t\in[0,T]$. 
	
	We have from (\ref{new fixedpower})
	\begin{equation}\label{5.26.3}
		\begin{split}
			d\hat{Z}_{t}&=\delta\sum_{k=1}^{K}e^{\beta_{k}t}\left(\hat{Z}_{t}\right)^{\frac{\theta_{k}p_{k}}{p_{k}-1}}\hat{G}_{k}(t,\overline{X}_{T},\hat{Z})\hat{f}_{k}(t,\overline{X}_{T},\hat{Z})F(\left\{k\right\})dt\\&\leq \delta x_{0}\sum_{k=1}^{K}e^{\beta_{k}t}z_{0}^{\frac{\theta_{k}p_{k}}{p_{k}-1}}e^{a_{k}(t-T)}\overline{X}_{T}^{\frac{p_{k}\theta_{k}}{1-p_{k}}}\exp\left\{\frac{\mu_{k}^{2}}{(1-p_{k})\sigma_{k}^{2}}t\right\}dt\\&\leq \delta x_{0}\sum_{k=1}^{K}e^{-a_{k}T}z_{0}^{\frac{\theta_{k}p_{k}}{p_{k}-1}}C_{2}^{\frac{p_{k}\theta_{k}}{1-p_{k}}}\exp\left\{\left(\beta_{k}+a_{k}+\frac{\mu_{k}^{2}}{(1-p_{k})\sigma_{k}^{2}}\right)t\right\}dt\\&\leq EKdt,\quad \forall t\in [0,T],
		\end{split}
	\end{equation}
	where\begin{equation}
		\begin{split}
			&\beta_{k}:= \frac{\delta }{1-p_{k}}\left(1+\theta_{k}p_{k}-p_{k}\right),\\
			&E:=\delta x_{0}\max_{1\leq k\leq K}e^{-a_{k}T}z_{0}^{\frac{\theta_{k}p_{k}}{p_{k}-1}}C_{2}^{\frac{p_{k}\theta_{k}}{1-p_{k}}}\left(\max_{t\in[0,T]}\exp\left\{\left(\beta_{k}+a_{k}+\frac{\mu_{k}^{2}}{(1-p_{k})\sigma_{k}^{2}}\right)t\right\}\right).
		\end{split}
	\end{equation}
	For the first inequality, we have just used 
	\begin{equation}
		\begin{split}
			&\hat{f}_{k}(t,\overline{X}_{T},\hat{Z})\leq x_{0}\exp\left\{\frac{\mu_{k}^{2}}{(1-p_{k})\sigma_{k}^{2}}t\right\},\\
			&\hat{G}_{k}(t,\overline{X}_{T},\hat{Z})\leq e^{ a_{k}(t-T)}\left(\overline{X}_{T}\right)^{\frac{p_{k}\theta_{k}}{1-p_{k}}},\\
			&\hat{Z}_{t}\geq z_{0}.
		\end{split}
	\end{equation}
	And for the second inequality, we have just used $\overline{X}_{T}\leq C_{2}$. 
	
	Integrating (\ref{5.26.3}), we obtain 
	\begin{equation}
		\hat{Z}_{t}\leq M(t):= EKt+z_{0}\leq M(T),\quad \forall t\in[0,T].
	\end{equation}
\end{proof}
Combining (\ref{5.26.1}) with Proposition \ref{5.26.5}, we  get a new lower bound for $\overline{X}_{T}$:
\begin{equation}\label{6.21}
	\begin{split}
		\overline{X}_{T}\geq C_{2}\exp&\left\{-\int_{0}^{T}\sum_{k=1}^{K}\exp\left(\frac{\theta _{k}p_{k}\delta}{1-p_{k}}s\right)z_{0}^{\frac{\theta_{k}p_{k}}{p_{k}-1}}\right.\\&\left.\left(\int_{s}^{T}e^{a_{k}(v-s)}\exp\left(\frac{\theta_{k}p_{k}\delta}{1-p_{k}}v\right)\left(\hat{Z}_{v}\right)^{\frac{\theta_{k}p_{k}}{p_{k}-1}}dv\right)^{-1}ds\right\}\\\geq C_{0}\\:=
		C_{2}\exp&\left\{-\int_{0}^{T}\sum_{k=1}^{K}\exp\left(\frac{\theta _{k}p_{k}\delta}{1-p_{k}}s\right)z_{0}^{\frac{\theta_{k}p_{k}}{p_{k}-1}}\right.\\&\left.\left(\int_{s}^{T}e^{a_{k}(v-s)}\exp\left(\frac{\theta_{k}p_{k}\delta}{1-p_{k}}v\right)\left(M(T)\right)^{\frac{\theta_{k}p_{k}}{p_{k}-1}}dv\right)^{-1}ds\right\},
	\end{split}
\end{equation}
where the first inequality follows from $$\hat{G}_{k}(s,\overline{X}_{T},\hat{Z})<\left(\int_{s}^{T}e^{a_{k}(v-s)}\exp\left(\frac{\theta_{k} p_{k}\delta}{1-p_{k}}v\right)\left(\hat{Z}_{v}\right)^{\frac{\theta_{k}p_{k}}{p_{k}-1}}dv\right)^{-1}$$ and the second inequality follows from $\hat{Z}_{t}\leq M(T)$ for any $t\in[0,T]$.

\section{Additional Explanation to Assumption 5}\label{appendixcara}
Let $i\in\left\{1,\cdots,n\right\}$ and $x_{0}\in\mathbb{R}$ be given. Assume that $\Pi^{i}:[0,T]\rightarrow\mathbb{R}$ is a continuous function and $C^{i}:[0,T]\times\mathbb{R}\rightarrow\mathbb{R}$ is of an affine form:
$$C^{i}(t,x):=p^{i}(t)x+q^{i}(t),\quad (t,x)\in[0,T]\times\mathbb{R},$$
for some continuous functions $p^{i},q^{i}:[0,T]\rightarrow\mathbb{R}$. Then, solving the following closed-loop system for $\left(X^{*,1},\cdots,X^{*,n},X^{i}\right)$:
\begin{equation}
	\begin{cases}
		dX^{*,j}_{t}=\left(\Pi^{*,j}(t)\mu_{\alpha(j)}-C^{*,j}(t,X^{*,j}_{t})\right)dt+\Pi^{*,j}(t)\sigma_{\alpha(j)}dW^{j}_{t},\quad j=1,\cdots, n,\\
		dX^{i}_{t}=\left(\Pi^{i}(t)\mu_{\alpha(i)}-C^{i}(t,X^{i}_{t})\right)dt+\Pi^{i}(t)\sigma_{\alpha(i)}dW^{i}_{t},\\
		X^{*,j}_{0}=x_{0},\quad j=1,\cdots,n,\\
		X^{i}_{0}=x_{0},
	\end{cases}
\end{equation}
where $\left(\Pi^{*,i},C^{*,i}\right)_{i=1}^{n}$ is given by (\ref{candi}), we obtain that for $t\in[0,T]$,
\begin{equation*}
	\begin{split}
		&X^{*,j}_{t}=f^{j}(t)+\left(T+1-t\right)\frac{\mu_{\alpha(j)}}{\sigma_{\alpha(j)}}\beta_{\alpha(j)}W^{j}_{t},\quad j=1,\cdots,n,\\
		&X^{i}_{t}=g^{i}(t)+e^{-\int_{0}^{t}p^{i}(s)ds}\int_{0}^{t}e^{\int_{0}^{s}p^{i}(v)dv}\Pi^{i}(s)\sigma_{\alpha(i)}dW^{i}_{s},
	\end{split}
\end{equation*}
where 
\begin{equation}
	\begin{split}
		&f^{j}(t):=x_{0}\frac{T+1-t}{T+1}+\left(T+1-t\right)\int_{0}^{t}\left[\frac{\theta_{\alpha(j)}}{\left(T+1-s\right)^{2}}\left(\int_{s}^{T}\overline{Z}_{v}dv+\overline{X}_{T}\right)-\frac{\theta_{\alpha(j)}}{T+1-s}\overline{Z}_{s}\right.\\&\left.+\frac{1}{4}\left(\frac{\mu_{\alpha(j)}}{\sigma_{\alpha(j)}}\right)^{2}\beta_{\alpha(j)}\frac{1}{\left(T+1-s\right)^{2}}+\frac{3}{4}\left(\frac{\mu_{\alpha(j)}}{\sigma_{\alpha(j)}}\right)^{2}\beta_{\alpha(j)}\right]ds,\quad j=1,\cdots,n,\\
		&g^{i}(t):=x_{0}e^{-\int_{0}^{t}p^{i}(s)ds}+e^{-\int_{0}^{t}p^{i}(s)ds}\int_{0}^{t}e^{\int_{0}^{s}p^{i}(v)dv}\left(\Pi^{i}(s)\mu_{\alpha(i)}-q^{i}(s)\right)ds.
	\end{split}
\end{equation}
It is obvious that there exists a constant $M$ independent of $j$ and $n$, such that for all $j\in\left\{1,\cdots,n\right\}$, 
\begin{equation}
	\sup_{t\in[0,T]}|f^{j}(t)|\leq M.
\end{equation}
 Note that 
\begin{equation}
	\begin{split}
		&\sup_{t\in[0,T]}|\Pi^{*,i}(t)|\leq\sup_{1\leq k\leq K}\left(\beta_{k}\frac{\mu_{k}}{(\sigma_{k})^{2}}\right)(T+1),\quad i=1,\cdots,n,\\
		&C^{*,i}(t,x)=\frac{1}{T+1-t}x+h^{i}(t),\quad i=1,\cdots,n,
	\end{split}
\end{equation}
where $h^{i}:[0,T]\rightarrow\mathbb{R}$ is continuous and has a uniform bound $M$, to abuse the notation, which is independent of $i$ and $n$, i.e., $\sup_{t\in[0,T]}|h^{i}(t)|\leq M$ for all $i\in\left\{1,\cdots,n\right\}$.

In this section, we only show 
\begin{equation}
\mathbb{E}\left[\int_{0}^{T}\left|U_{\alpha(i)}\left(C^{i}_{t}-\theta_{\alpha(i)}\overline{Z}^{*,n}_{t}\right)\right|^{2}dt+\left|U_{\alpha(i)}\left(X^{i}_{T}-\theta_{\alpha(i)}\overline{X}_{T}^{*,n}\right)\right|^{2}\right]<M<\infty,
\end{equation}
where $M$ is a constant independent of $n$, as the other condition in Assumption \ref{assump5} is much easier to verify.

Without loss of generality, we assume $\alpha(i)=1$. Then, 
\begin{equation*}
	\begin{split}
		&\left|U_{1}\left(C^{i}_{t}-\theta_{1}\overline{Z}^{*,n}_{t}\right)\right|^{2}\\&=\exp\left\{-\frac{2}{\beta_{1}}\left(p^{i}(t)X^{i}_{t}+q^{i}(t)\right)+2\frac{\theta_{1}}{\beta_{1}}\overline{Z}^{*,n}_{t}\right\}\\&=\exp\left\{-\frac{2}{\beta_{1}}\left(p^{i}(t)X^{i}_{t}+q^{i}(t)\right)+2\frac{\theta_{1}}{\beta_{1}}e^{-\delta t}z_{0}+\frac{\theta_{1}}{\beta_{1}}\frac{2}{n}\sum_{j=1}^{n}\int_{0}^{t}\delta e^{\delta(s-t)}C^{*,j}(s,X^{*,j}_{s})ds\right\}\\&\leq C\exp\left\{-\frac{2}{\beta_{1}}p^{i}(t)X^{i}_{t}+\delta\frac{\theta_{1}}{\beta_{1}}\frac{2}{n}\sum_{j=1}^{n}\int_{0}^{t}\frac{e^{\delta(s-t)}}{T+1-s}X^{*,j}_{s}ds\right\}\\&
		\leq C\exp\left\{-\frac{2}{\beta_{1}}p^{i}(t)e^{-\int_{0}^{t}p^{i}(s)ds}\int_{0}^{t}e^{\int_{0}^{s}p^{i}(v)dv}\Pi^{i}(s)\sigma_{1}dW^{i}_{s}+\delta\frac{\theta_{1}}{\beta_{1}}\frac{2}{n}\sum_{j=1}^{n}\int_{0}^{t}e^{\delta(s-t)}\frac{\mu_{\alpha(j)}}{\sigma_{\alpha(j)}}\beta_{\alpha(j)}W^{j}_{s}ds\right\}.
	\end{split}
\end{equation*}
By {It\^{o}'s} formula, we get 
\begin{equation}
\int_{0}^{t}e^{\delta s}\frac{\mu_{\alpha(j)}}{\sigma_{\alpha(j)}}\beta_{\alpha(j)}W^{j}_{s}ds=\int_{0}^{t}\left(F^{j}(t)-F^{j}(s)\right)dW^{j}_{s},
\end{equation}
where $$F^{j}(t):=\frac{\mu_{\alpha(j)}\beta_{\alpha(j)}}{\delta\sigma_{\alpha(j)}}\left(e^{\delta t}-1\right).$$
Therefore, we have
\begin{equation*}
	\begin{split}
		\left|U_{1}\left(C^{i}_{t}-\theta_{1}\overline{Z}^{*,n}_{t}\right)\right|^{2}
		\leq C\exp\left\{\int_{0}^{t}G(s,t)dW^{i}_{s}+\frac{1}{n}\sum_{j=1}^{n}\int_{0}^{t}H^{j}(s,t)dW^{j}_{s}\right\},
	\end{split}
\end{equation*}
for some bounded functions $G,H^{j}:[0,T]\times[0,T]\rightarrow\mathbb{R}$. It is worth noting that $H^{j}$ has a uniform bound, which is still denoted by $M$.

Hence, 
\begin{equation}
	\begin{split}
		\mathbb{E}\left[\left|U_{1}\left(C^{i}_{t}-\theta_{1}\overline{Z}^{*,n}_{t}\right)\right|^{2}\right]&\leq C\mathbb{E}\left[\exp\left\{\int_{0}^{t}G(s,t)dW^{i}_{s}+\frac{1}{n}\sum_{j=1}^{n}\int_{0}^{t}H^{j}(s,t)dW^{j}_{s}\right\}\right]\\&=\exp\left\{\int_{0}^{t}\left(G(s,t)+\frac{1}{n}H^{i}(s,t)\right)^{2}ds+\frac{1}{n}\sum_{j\neq i}^{n}\int_{0}^{t}\left(H^{j}(s,t)\right)^{2}ds\right\} .
	\end{split}
\end{equation}
Then, it is straightforward to get
\begin{equation}
	\mathbb{E}\left[\int_{0}^{T}\left|U_{\alpha(i)}\left(C^{i}_{t}-\theta_{\alpha(i)}\overline{Z}^{*,n}_{t}\right)\right|^{2}dt\right]<M<\infty,
\end{equation}
where $M$ is a constant independent of $n$.
The estimation on term $\mathbb{E}\left[\left|U_{1}\left(X^{i}_{T}-\theta_{1}\overline{X}_{T}^{*,n}\right)\right|^{2}\right]$ is similar and we omit it here.

\bibliographystyle{plainnat}
\bibliography{ref}

\begin{thebibliography}{26}
\providecommand{\natexlab}[1]{#1}
\providecommand{\url}[1]{\texttt{#1}}
\expandafter\ifx\csname urlstyle\endcsname\relax
  \providecommand{\doi}[1]{doi: #1}\else
  \providecommand{\doi}{doi: \begingroup \urlstyle{rm}\Url}\fi

\bibitem[Abel(1990)]{abel90}
Andrew~B. Abel.
\newblock Asset prices under habit formation and catching up with the joneses.
\newblock \emph{The American Economic Review}, 80\penalty0 (2):\penalty0 38--42, 1990.
\newblock ISSN 00028282.
\newblock URL \url{http://www.jstor.org/stable/2006539}.

\bibitem[Abel(1999)]{ABEL19993}
Andrew~B. Abel.
\newblock Risk premia and term premia in general equilibrium.
\newblock \emph{Journal of Monetary Economics}, 43\penalty0 (1):\penalty0 3--33, 1999.
\newblock \doi{https://doi.org/10.1016/S0304-3932(98)00039-7}.
\newblock URL \url{https://www.sciencedirect.com/science/article/pii/S0304393298000397}.

\bibitem[Angoshtari et~al.(2023)Angoshtari, Bayraktar, and Young]{doi:10.1137/22M1471560}
Bahman Angoshtari, Erhan Bayraktar, and Virginia~R. Young.
\newblock Optimal consumption under a habit-formation constraint: The deterministic case.
\newblock \emph{SIAM Journal on Financial Mathematics}, 14\penalty0 (2):\penalty0 557--597, 2023.
\newblock \doi{10.1137/22M1471560}.
\newblock URL \url{https://doi.org/10.1137/22M1471560}.

\bibitem[Bielagk et~al.(2017)Bielagk, Lionnet, and Reis]{2017price}
Jana Bielagk, Arnaud Lionnet, and Gon\c{c}alo~Dos Reis.
\newblock Equilibrium pricing under relative performance concerns.
\newblock \emph{SIAM Journal on Financial Mathematics}, 8\penalty0 (1):\penalty0 435--482, 2017.
\newblock \doi{10.1137/16M1082536}.
\newblock URL \url{https://doi.org/10.1137/16M1082536}.

\bibitem[Bo et~al.(2022)Bo, Wang, and Yu]{2022habit}
Lijun Bo, Shihua Wang, and Xiang Yu.
\newblock A mean field game approach to equilibrium consumption under external habit formation.
\newblock \emph{Preprint,\, available at arXiv:2206.13341}, 2022.

\bibitem[Campbell and Cochrane(1999)]{dda705ff-8999-3974-bcb9-25a27cb27b25}
John~Y. Campbell and John~H. Cochrane.
\newblock By force of habit: A consumption‐based explanation of aggregate stock market behavior.
\newblock \emph{Journal of Political Economy}, 107\penalty0 (2):\penalty0 205--251, 1999.
\newblock URL \url{http://www.jstor.org/stable/10.1086/250059}.

\bibitem[Constantinides(1990)]{f8b89b64-c09e-3f5f-b47c-1a4cf3168fdc}
George~M. Constantinides.
\newblock Habit formation: A resolution of the equity premium puzzle.
\newblock \emph{Journal of Political Economy}, 98\penalty0 (3):\penalty0 519--543, 1990.
\newblock URL \url{http://www.jstor.org/stable/2937698}.

\bibitem[Detemple and Zapatero(1991)]{bd6761ac-561d-37c7-856a-f5e187d68cd9}
Jerome~B. Detemple and Fernando Zapatero.
\newblock Asset prices in an exchange economy with habit formation.
\newblock \emph{Econometrica}, 59\penalty0 (6):\penalty0 1633--1657, 1991.
\newblock URL \url{http://www.jstor.org/stable/2938283}.

\bibitem[dos Reis and Platonov(2022)]{doi:10.1137/20M138421X}
Gon\c{c}alo dos Reis and Vadim Platonov.
\newblock {Forward Utility and Market Adjustments in Relative Investment-Consumption Games of Many Players}.
\newblock \emph{SIAM Journal on Financial Mathematics}, 13\penalty0 (3):\penalty0 844--876, 2022.
\newblock \doi{10.1137/20M138421X}.

\bibitem[Espinosa and Touzi(2015)]{2013OPTIMAL}
Gilles-Edouard Espinosa and Nizar Touzi.
\newblock Optimal investment under relative performance concerns.
\newblock \emph{Mathematical Finance}, 25:\penalty0 221--257, 2015.
\newblock \doi{10.1111/mafi.12034}.

\bibitem[Fernández-Villaverde and Krueger(2007)]{10.1162/rest.89.3.552}
Jesús Fernández-Villaverde and Dirk Krueger.
\newblock {Consumption over the Life Cycle: Facts from Consumer Expenditure Survey Data}.
\newblock \emph{The Review of Economics and Statistics}, 89\penalty0 (3):\penalty0 552--565, 2007.
\newblock \doi{10.1162/rest.89.3.552}.
\newblock URL \url{https://doi.org/10.1162/rest.89.3.552}.

\bibitem[Fu(2023)]{2022Meanfieldpgc}
Guanxing Fu.
\newblock {Mean Field Portfolio Games with Consumption}.
\newblock \emph{Mathematics and Financial Economics}, 17:\penalty0 79--99, 2023.
\newblock \doi{10.1007/s11579-022-00328-2}.

\bibitem[Fu and Zhou(2022)]{2022Meanfieldpg}
Guanxing Fu and Chao Zhou.
\newblock {Mean Field Portfolio Games}.
\newblock \emph{Finance and Stochastics}, 27:\penalty0 189--231, 2022.
\newblock \doi{10.1007/s00780-022-00492-9}.

\bibitem[Gali(1994)]{Gali94}
Jordi Gali.
\newblock {Keeping Up with the Joneses: Consumption Externalities, Portfolio Choice, and Asset Prices}.
\newblock \emph{Journal of Money, Credit and Banking}, 26\penalty0 (1):\penalty0 1--8, 1994.
\newblock URL \url{https://ideas.repec.org/a/mcb/jmoncb/v26y1994i1p1-8.html}.

\bibitem[Gómez(2007)]{GOMEZ200795}
Juan-Pedro Gómez.
\newblock The impact of keeping up with the joneses behavior on asset prices and portfolio choice.
\newblock \emph{Finance Research Letters}, 4\penalty0 (2):\penalty0 95--103, 2007.
\newblock ISSN 1544-6123.
\newblock \doi{https://doi.org/10.1016/j.frl.2007.01.002}.
\newblock URL \url{https://www.sciencedirect.com/science/article/pii/S1544612307000037}.

\bibitem[Hamaguchi(2021)]{2021hamag}
Y.~Hamaguchi.
\newblock {Time-Inconsistent Consumption-Investment Problems in Incomplete Markets under General Discount Functions}.
\newblock \emph{SIAM Journal on Control and Optimization}, 59\penalty0 (3):\penalty0 2121--2146, 2021.

\bibitem[Hu and Zariphopoulou(2022)]{2021Meanfieldito}
Ruimeng Hu and Thaleia Zariphopoulou.
\newblock \emph{N-Player and Mean-Field Games in Itˆo-Diffusion Markets with Competitive or Homophilous Interaction}, pages 209--237.
\newblock Springer International Publishing, Cham, 2022.
\newblock ISBN 978-3-030-98519-6.
\newblock \doi{10.1007/978-3-030-98519-6_9}.
\newblock URL \url{https://doi.org/10.1007/978-3-030-98519-6_9}.

\bibitem[Huang et~al.(2006)Huang, Malhame, and Caines]{2006Large}
M.~Huang, R.~P. Malhame, and P.~E. Caines.
\newblock {Large population stochastic dynamic games: Closed-loop McKean-Vlasov systems and the Nash certainty equivalence principle}.
\newblock \emph{Communications in Information and Systems}, 6\penalty0 (3):\penalty0 221--251, 2006.

\bibitem[Lacker and Soret(2020)]{2020Many}
Daniel Lacker and Agathe Soret.
\newblock Many-player games of optimal consumption and investment under relative performance criteria.
\newblock \emph{Mathematics and Financial Economics}, 2020.
\newblock \doi{10.1007/s11579-019-00255-9}.

\bibitem[Lacker and Zariphopoulou(2019)]{D2017Mean}
Daniel Lacker and Thaleia Zariphopoulou.
\newblock Mean field and n -agent games for optimal investment under relative performance criteria.
\newblock \emph{Mathematical Finance}, 2019.
\newblock \doi{10.1111/mafi.12206}.

\bibitem[Lasry and Lions(2007)]{Lasry-2007}
Jean-Michel Lasry and Pierre-Louis Lions.
\newblock Mean field games.
\newblock \emph{Japanese Journal of Mathematics}, 2:\penalty0 229--260, 2007.
\newblock \doi{10.1007/s11537-007-0657-8}.

\bibitem[Liu(2004)]{2004Liu}
Hong Liu.
\newblock Optimal consumption and investment with transaction costs and multiple risky assets.
\newblock \emph{The Journal of Finance}, 59:\penalty0 289--338, 2004.
\newblock ISSN 0022-1082,1540-6261.
\newblock \doi{10.1111/j.1540-6261.2004.00634.x}.
\newblock URL \url{http://doi.org/10.1111/j.1540-6261.2004.00634.x}.

\bibitem[Ma and Zhu(2019)]{2019ma}
Guiyuan Ma and Song-Ping Zhu.
\newblock Optimal investment and consumption under a continuous-time cointegration model with exponential utility.
\newblock \emph{Quantitative Finance}, 19\penalty0 (7):\penalty0 1135--1149, 2019.
\newblock \doi{10.1080/14697688.2019.1570317}.
\newblock URL \url{https://doi.org/10.1080/14697688.2019.1570317}.

\bibitem[Merton(1969)]{1969Merton}
Robert~C. Merton.
\newblock Lifetime portfolio selection under uncertainty: The continuous-time case.
\newblock \emph{Review of Economics and Statistics}, 51:\penalty0 247--257, 1969.
\newblock ISSN 0034-6535,1530-9142.
\newblock \doi{10.2307/1926560}.
\newblock URL \url{http://doi.org/10.2307/1926560}.

\bibitem[Thurow(1969)]{RePEc:aea:aecrev:v:59:y:1969:i:3:p:324-30}
Lester~C Thurow.
\newblock {The Optimum Lifetime Distribution of Consumption Expenditures}.
\newblock \emph{American Economic Review}, 59\penalty0 (3):\penalty0 324--330, 1969.
\newblock URL \url{https://ideas.repec.org/a/aea/aecrev/v59y1969i3p324-30.html}.

\bibitem[Vayanos(2015)]{2015Dimitri}
Dimitri Vayanos.
\newblock {Transaction Costs and Asset Prices: A Dynamic Equilibrium model}.
\newblock \emph{The Review of Financial Studies}, 11\penalty0 (1):\penalty0 1--58, 06 2015.
\newblock ISSN 0893-9454.
\newblock \doi{10.1093/rfs/11.1.1}.
\newblock URL \url{https://doi.org/10.1093/rfs/11.1.1}.

\end{thebibliography}

\end{document}